\documentclass[12pt,reqno]{amsart}
\usepackage[margin=1in]{geometry}
\usepackage{amsmath,amssymb,amsthm,graphicx,amsxtra, setspace}
\usepackage[utf8]{inputenc}
\usepackage{mathrsfs}
\usepackage{hyperref}
\usepackage{upgreek}
\usepackage{mathtools}
\allowdisplaybreaks

\usepackage[pagewise]{lineno}

\newtheorem{theorem}{Theorem}[section]
\newtheorem{lemma}[theorem]{Lemma}
\newtheorem{proposition}[theorem]{Proposition}

\newtheorem{corollary}[theorem]{Corollary}
\newtheorem{definition}[theorem]{Definition}

\newtheorem{remark}[theorem]{Remark}
\newtheorem{hypothesis}[theorem]{Hypothesis}

\let\originalleft\left
\let\originalright\right
\renewcommand{\left}{\mathopen{}\mathclose\bgroup\originalleft}
\renewcommand{\right}{\aftergroup\egroup\originalright}

\newcommand{\Tr}{\mathop{\mathrm{Tr}}}
\renewcommand{\d}{\/\mathrm{d}\/}

\def\w{\textbf{W}^{\varepsilon}_{{\theta}^{\varepsilon}}}

\def\e{\varepsilon}

\def\t{t\wedge\tau_N^n}
\def\s{t\wedge\tau_N}
\def\S{\mathrm{S}}
\def\T{T\wedge\tau_N^n}
\def\L{\mathbb{L}}
\def\A{\mathrm{A}}

\def\C{\mathrm{C}}
\def\f{\mathbf{f}}

\def\B{\mathrm{B}}
\def\D{\mathrm{D}}
\def\y{\mathbf{y}}

\def\Z{\mathrm{Z}}
\def\E{\mathbb{E}}
\def\X{\mathbb{X}}
\def\x{\mathbf{x}}

\def\z{\mathbf{z}}
\def\v{\mathbf{v}}
\def\V{\mathbb{v}}
\def\w{\mathbf{w}}
\def\W{\mathrm{W}}
\def\G{\mathrm{G}}
\def\Q{\mathrm{Q}}
\def\M{\mathrm{M}}
\def\N{\mathbb{N}}

\def\no{\nonumber}
\def\V{\mathbb{V}}
\def\wi{\widetilde}
\def\Q{\mathrm{Q}}
\def\u{\mathrm{U}}
\def\P{\mathrm{P}}
\def\u{\mathbf{u}}
\def\H{\mathbb{H}}

\newcommand{\R}{\mathbb{R}}

\renewcommand{\d}{\/\mathrm{d}\/}

\newcommand{\Addresses}{{
		\footnote{
			
			\noindent \textsuperscript{1}Department of Mathematics, Indian Institute of Technology Roorkee-IIT Roorkee,
			Haridwar Highway, Roorkee, Uttarakhand 247667, INDIA.\par\nopagebreak
			\noindent  \textit{e-mail:} \texttt{manilfma@iitr.ac.in, maniltmohan@gmail.com.}
			
			\noindent \textsuperscript{*}Corresponding author.

			\textit{Key words:} convective Brinkman-Forchheimer equations, jump noise, strong solution, exponential stability, invariant measure.
			
			Mathematics Subject Classification (2010): Primary 60H15; Secondary 35R60, 35Q30, 76D05.

}}}

\begin{document}
	
	
	\title[Stochastic convective Brinkman-Forchheimer equations]{Well-posedness and asymptotic behavior of the stochastic convective Brinkman-Forchheimer equations perturbed by pure jump noise 
			\Addresses}
	\author[M. T. Mohan ]{Manil T. Mohan\textsuperscript{1*}}

	\maketitle
	
	\begin{abstract}
	This paper is concerned about the stochastic convective Brinkman-Forchheimer (SCBF)  equations subjected to multiplicative pure jump noise in  bounded or periodic domains.  Our first goal is to establish the existence of a pathwise unique strong solution satisfying the energy equality (It\^o's formula) to the SCBF equations.  We resolve the issue of the global solvability of SCBF equations, by using  a monotonicity property of the linear and nonlinear operators and  a stochastic generalization of  the Minty-Browder technique. The major difficulty is that an It\^o's formula  in infinite dimensions is not available  for such systems. This difficulty is overcame by approximating the solution using approximate functions composing of the elements of eigenspaces of the Stokes operator in such a way that the approximations are bounded and converge in both Sobolev and Lebesgue spaces simultaneously.  Due to technical difficulties, we discuss about the global in time regularity results of such strong solutions in periodic domains only. Once the system is well-posed, we look for the asymptotic behavior of strong solutions.   The exponential stability results (in mean square and pathwise sense) for the stationary solutions  is established in this work for large effective viscosity. Moreover, a stabilization result of the SCBF  equations by using a multiplicative pure jump noise is also obtained. Finally, we prove the existence of a unique ergodic and strongly mixing invariant measure for the SCBF  equations subject to multiplicative pure jump noise, by using the exponential stability of strong solutions.
	\end{abstract}

	\section{Introduction}\label{sec1}\setcounter{equation}{0}
	The global solvability of  the classical 3D Navier-Stokes equations (see \cite{OAL,Te,Te1,FMRT,GGP,JCR3}, etc) is one of the biggest mysteries in Mathematical Physics. Several mathematicians came forward with some modifications of this classical model and they established the well-posedness of such modified systems. The authors in \cite{ZCQJ}, showed that the Cauchy problem for the Navier-Stokes equations with damping $r|\u|^{r-1}\u$ in the whole space has global weak solutions, for any $r\geq 1$. The authors in  \cite{MRXZ} proved the existence and uniqueness of a smooth solution to a tamed 3D Navier-Stokes equation in the whole space and they showed that if there exists a bounded smooth solution to the classical 3D Navier-Stokes equation, then this solution satisfies the tamed equation.  The authors in \cite{SNA} considered the Navier-Stokes problem in bounded domains with compact boundary, modified by the absorption term $|\u|^{r-2}\u$, for $r>2$. For this modified problem, they proved the existence of weak solutions in the Leray-Hopf sense, for any dimension $n\geq 2$ and its uniqueness for $n=2$. But  in three dimensions, they were not able to establish the energy equality satisfied by the weak solutions, later this issue was also resolved (see \cite{KT2,CLF}, etc). 
	
	In this work, we consider the convective Brinkman-Forchheimer equations in bounded or periodic domains. The model we are going to describe is on bounded domains, but one can rewrite the same model in periodic domains also (cf. \cite{KWH}). 	Let $\mathcal{O}\subset\R^n$ ($n=2,3$) be a bounded domain with a smooth boundary $\partial\mathcal{O}$. The	convective Brinkman-Forchheimer (CBF)  equations are given by (see \cite{KT2} for Brinkman-Forchheimer equations with fast growing nonlinearities)
\begin{equation}\label{1}
\left\{
	\begin{aligned}
	\frac{\partial \u}{\partial t}-\mu \Delta\u+(\u\cdot\nabla)\u+\alpha\u+\beta|\u|^{r-1}\u+\nabla p&=\mathbf{f}, \ \text{ in } \ \mathcal{O}\times(0,T), \\ \nabla\cdot\u&=0, \ \text{ in } \ \mathcal{O}\times(0,T), \\
	\u&=\mathbf{0}\ \text{ on } \ \partial\mathcal{O}\times(0,T), \\
	\u(0)&=\u_0 \ \text{ in } \ \mathcal{O},\\
	\int_{\mathcal{O}}p(x,t)\d x&=0, \ \text{ in } \ (0,T).
	\end{aligned}
	\right.
	\end{equation}
The convective Brinkman-Forchheimer equations \eqref{1} describe the motion of incompressible fluid flows in a saturated porous medium.  Here $\u(t , x) \in \R^n$ represents the velocity field at time $t$ and position $x$, $p(t,x)\in\R$ denotes the pressure field, $\f(t,x)\in\R^n$ is an external forcing. The final condition in \eqref{1} is imposed for the uniqueness of the pressure $p$. The constant $\mu$ represents the positive Brinkman coefficient (effective viscosity), the positive constants $\alpha$ and $\beta$ represent the Darcy (permeability of porous medium) and Forchheimer (proportional to the porosity of the material) coefficients, respectively. The absorption exponent $r\in[1,\infty)$ and the case $r=3$ is known as the critical exponent.  Note that for $\alpha=\beta=0$, we obtain the classical 3D Navier-Stokes equations.

Now, we discuss about some of the solvability results available in the literature for the 3D CBF equations and related models in the whole space as well as periodic domains. The authors in \cite{ZCQJ} considered the Cauchy problem corresponding to \eqref{1} in the whole space (with $\alpha=0$ and $\beta=r$) and showed that the system has global weak solutions, for any $r\geq 1$, global strong solutions, for any $r\geq 7/2$ and that the strong solution is unique, for any $7/2\leq r\leq 5$. The authors in \cite{ZZXW} improved this result and they showed that the above mentioned problem possesses global strong solutions, for any $r>3$ and the strong solution is unique, when $3<r\leq 5$. Later, the authors in \cite{YZ}  proved that the strong solution exists globally for $r\geq 3$, and they established two regularity criteria, for $1\leq r<3$. Moreover, for any $r\geq 1$, they proved that the strong solution is unique even among weak solutions.  A simple proof of the existence of global-in-time smooth solutions for the CBF equations \eqref{1} with $r>3$ on a 3D periodic domain is obtained in \cite{KWH}. The authors also proved that unique global, regular solutions exist also for the critical value $r=3$, provided that the coefficients satisfy the relation $4\beta\mu\geq 1$. Furthermore, they showed that in the critical case every weak solution verifies the energy equality and hence is continuous into the phase space $\L^2(\mathcal{O})$. Recently, the authors in \cite{KWH1} showed that the strong solutions of three dimensional CBF  equations in periodic domains with  the absorption exponent $r\in[1,3]$ remain strong under small enough changes of initial condition and forcing function.

Whereas in the bounded domains, as we discuss earlier, the existence of a global weak solution to the system was established in \cite{SNA}. Unlike whole space or periodic domains, in bounded domains, there is a technical difficulty for obtaining strong solutions to \eqref{1} with the regularity given in \eqref{2} for the velocity field $\u(\cdot)$ (see below).   It is  discussed in the paper \cite{KT2} that the major difficulty in working with bounded domains is that $\mathrm{P}_{\H}(|\u|^{r-1}\u)$ (here $\mathrm{P}_{\H}:\L^2(\mathcal{O})\to\H$ is the Helmholtz-Hodge orthogonal projection) need not be zero on the boundary, and $\mathrm{P}_{\H}$ and $-\Delta$ are not necessarily commuting (for a counter example, see Example 2.19, \cite{JCR4}). Moreover, $\Delta\u\big|_{\partial\mathcal{O}}\neq 0$ in general and the term with pressure will not disappear (see \cite{KT2}), while taking inner product with $\Delta\u$ to the first equation in \eqref{1}. Therefore, the equality 
\begin{align}\label{3}
&\int_{\mathcal{O}}(-\Delta\u(x))\cdot|\u(x)|^{r-1}\u(x)\d x\nonumber\\&=\int_{\mathcal{O}}|\nabla\u(x)|^2|\u(x)|^{r-1}\d x+4\left[\frac{r-1}{(r+1)^2}\right]\int_{\mathcal{O}}|\nabla|\u(x)|^{\frac{r+1}{2}}|^2\d x\nonumber\\&=\int_{\mathcal{O}}|\nabla\u(x)|^2|\u(x)|^{r-1}\d x+\frac{r-1}{4}\int_{\mathcal{O}}|\u(x)|^{r-3}|\nabla|\u(x)|^2|^2\d x,
\end{align}
may not be useful in the context of bounded domains. The authors in \cite{KT2} considered the Brinkman-Forchheimer equations with fast growing nonlinearities and they showed the existence of regular dissipative solutions and global attractors for the system \eqref{1} with $r> 3$. Note that the existence of global smooth solution assures that the energy equality is satisfied by the weak solutions in bounded domains. But the critical case of $r=3$ was open until it was resolved by the authors in \cite{CLF}. They were able to construct functions that can approximate functions defined on smooth bounded domains by elements of eigenspaces of linear operators (for example, the Laplacian or the Stokes operator) in such a way that the approximations are bounded and converge in both Sobolev and Lebesgue spaces simultaneously. As a simple application of this result, they proved that all weak solutions of the critical CBF equations ($r=3$) posed on a bounded domain in $\mathbb{R}^3$ satisfy the energy equality.

Next, we discuss about the stochastic counterpart of the problem \eqref{1} and related models in the whole space or on a torus. The existence of a unique strong solution 
\begin{align}\label{2}
\u\in\mathrm{L}^2(\Omega;\mathrm{L}^{\infty}(0,T;\H^1(\mathcal{O})))\cap\mathrm{L}^2(0,T;\H^2(\mathcal{O})),
\end{align} 
with $\mathbb{P}$-a.s. continuous paths in $\C([0,T];\H^1(\mathcal{O})),$  for $\u_0\in\mathrm{L}^2(\Omega;\H^1(\mathcal{O}))$,  to the stochastic tamed 3D Navier-Stokes equation perturbed by multiplicative Gaussian noise in the whole space as well as in the periodic boundary case is obtained in \cite{MRXZ1}. They also prove the existence of a  unique  invariant measure for the corresponding transition semigroup.  Recently, \cite{ZBGD} improved their results  for a slightly simplified system. The authors in \cite{WLMR} established the global existence and uniqueness of strong solutions for general  stochastic nonlinear evolution equations with coefficients satisfying some local monotonicity and generalized coercivity conditions subjected to multiplicative Gaussian noise. In \cite{WL}, the author showed the existence and uniqueness of strong solutions for a large class of SPDEs perturbed by multiplicative Gaussian noise, where the coefficients satisfy the local monotonicity and Lyapunov condition,  and he provided the stochastic tamed 3D Navier-Stokes equations as an example. A large deviation principle of Freidlin-Wentzell type for the stochastic tamed 3D Navier-Stokes equations driven by multiplicative Gaussian noise  in the whole space or on a torus  is established in \cite{MRTZ1}. The works described above established the existence and uniqueness of strong solutions in the regularity class given in \eqref{2}. The authors in  \cite{HBAM} described the global solvability of the 3D Navier-Stokes equations in the whole space with a Brinkman-Forchheimer type term subject to an anisotropic viscosity and a random perturbation of multiplicative Gaussian  type.

Let us now discuss about the results available in the literature for the stochastic convective Brinkman-Forchheimer (SCBF) and related models in bounded domains.   The authors in \cite{MRTZ} showed the existence and uniqueness of strong solutions to stochastic 3D tamed Navier-Stokes equations perturbed by multiplicative Gaussian noise on bounded domains with Dirichlet boundary conditions. They also  proved a small time large deviation principle for the solution. The author in \cite{BYo} proved the existence of a random attractor for the three-dimensional damped Navier-Stokes equations in bounded domains with additive noise by verifying the pullback flattening property. The existence of a random attractor  ($r>3$, for any $\beta>0$) as well as the existence of a unique invariant measure ($3<r\leq 5$, for any $\beta>0$ and $\beta\geq\frac{1}{2}$ for $r=3$) for the stochastic 3D Navier-Stokes equations with damping driven by a multiplicative Gaussian noise is established in the paper \cite{LHGH}. By using classical Faedo-Galerkin approximation and compactness method, the existence of martingale solutions for the stochastic 3D Navier-Stokes equations with nonlinear damping subjected to multiplicative Gaussian noise is obtained in \cite{LHGH1}. The exponential behavior and stabilizability of the strong solutions of the stochastic 3D Navier–Stokes equations with damping is discussed in the work  \cite{LHGH2}.  In the paper \cite{ZDRZ}, the authors showed the existence and uniqueness of a strong solution to stochastic 3D tamed Navier-Stokes equations driven by multiplicative L\'evy noise based on Galerkin's approximation and a kind of local monotonicity of the coefficients. They also established a large deviation principle of the strong solution using a weak convergence approach. A class of stochastic 3D Navier-Stokes equation with damping driven by pure jump noise is considered in the paper \cite{HGHL}. They proved the existence and uniqueness of strong solutions for the SPDEs like stochastic 3D Navier-Stokes equations with damping, stochastic tamed 3D Navier-Stokes equations, stochastic three-dimensional Brinkman-Forchheimer-extended Darcy model, etc. By using the exponential stability of solutions, the existence of a unique invariant measures is proved for such models perturbed by additive jump noise. As far as strong solutions in bounded domains are concerned, most  of these works proved regularity results in the space given in \eqref{2}, by using the estimate given in \eqref{3}, which may not hold true always (see the discussions before \eqref{3}). 

Recently, the author in \cite{MTM4} proved the global existence and uniqueness of pathwise strong solutions of SCBF  equations subjected to multiplicative Gaussian noise.  We extend that work to multiplicative jump noise. In this work, we consider the SCBF equations perturbed by multiplicative jump noise and show the existence and uniqueness of strong solutions in a larger space than the one given in \eqref{2} and discuss about some asymptotic behavior. We got the main motivation from the works \cite{KWH} and \cite{CLF}, for establishing an It\^o's formula in infinite dimensions for the solution process. The work  \cite{KWH} helped us to construct functions that can approximate functions defined on smooth bounded domains by elements of eigenspaces of Stokes operator in such a way that the approximations are bounded and converge in both Sobolev and Lebesgue spaces simultaneously. On the $n$-dimensional torus, one can approximate functions in $\L^p$-spaces using truncated Fourier expansions (see \cite{CLF}).   Due to the difficulty explained before \eqref{3}, one may  not expect the regularity of $\u(\cdot)$ given in \eqref{2} in bounded domains. Now, we list some of the major contributions of this work. The existence and uniqueness of strong solutions to SCBF equations ($r> 3,$ for any $\mu$ and $\beta$, and $r=3$ for $2\beta\mu\geq 1$) perturbed by multiplicative jump noise in bounded domains with $\u_0\in\mathrm{L}^2(\Omega;\L^2(\mathcal{O})))$ is obtained in the space
	\begin{center} $\mathrm{L}^2(\Omega;\mathrm{L}^{\infty}(0,T;\L^2(\mathcal{O}))\cap\mathrm{L}^2(0,T;\H_0^1(\mathcal{O})))\cap\mathrm{L}^{r+1}(\Omega;\mathrm{L}^{r+1}(0,T;\L^{r+1}(\mathcal{O}))),
	$\end{center} 
with $\mathbb{P}$-a.s. paths in $\D([0,T];\L^2(\mathcal{O}))$, where $\D([0,T];\L^2(\mathcal{O}))$ is the space of all c\`adl\`ag functions (right continuous functions with left limits) from $[0,T]$ to $\L^2(\mathcal{O})$. The energy equality (It\^o's formula) satisfied by the SCBF equations driven by multiplicative jump noise is established by approximating the strong solution using the finite-dimensional space spanned by the first $n$ eigenfunctions of the Stokes operator. The exponential stability (in the mean square and almost sure sense) of the stationary solutions is obtained for large effective viscosity $\mu$ and the lower bound on $\mu$ does not depend on the bounds of the stationary solutions. A stabilization result of the SCBF  equations by using a multiplicative jump noise is also obtained. The existence of a unique ergodic and strongly mixing invariant measure for the SCBF  equations perturbed by multiplicative jump noise is established by using the exponential stability of strong solutions.

The  paper is organized as follows. In section \ref{sec3}, we define the linear and nonlinear operators, and provide the necessary function spaces needed to obtain the global solvability results of the system \eqref{1}. For $r>3$, we show that the sum of linear and nonlinear operators is monotone (Theorem \ref{thm2.2}),  and for the critical case $r=3$ and $2\beta\mu\geq 1$, we show that the  sum is globally monotone (Theorem \ref{thm2.3}). The important properties like demicontinuity and hence the hemicontinuity property of these operators is also obtained in the same section (Lemma \ref{lem2.8}). The SCBF equations perturbed by multiplicative jump noise is formulated in section \ref{sec4}. We first provide an abstract formulation of the SCBF equations in bounded or periodic domains. Then, we establish the existence and uniqueness of global pathwise strong solution by  using  the monotonicity property of the linear and nonlinear operators as well as a stochastic generalization of the Minty-Browder technique (see Proposition \ref{prop1} for a-priori energy estimates satisfied by a Faedo-Galerkin approximated system and Theorem \ref{exis} for global solvability results). We overcame the major difficulty of establishing the energy equality (It\^o's formula) for the SCBF equations by approximating the solution using the finite-dimensional space spanned by the first $n$ eigenfunctions of the Stokes operator. Due to the  technical difficulties explained earlier, we prove the regularity results of the global strong solutions under smoothness assumptions on the initial data and further assumptions on noise coefficient, in periodic domains (on torus) only (Theorem \ref{thm3.10}). The section \ref{se5} is devoted for establishing the exponential stability (in the mean square and almost sure sense) of the stationary solutions (Theorems \ref{exp1} and \ref{exp2}) for large effective viscosity $\mu$. In both Theorems, the lower bound of $\mu$ does not depend on the bounds of the stationary solutions.  A stabilization result of the stochastic convective Brinkman-Forchheimer  equations by using a multiplicative  jump noise is also obtained in the same section (Theorem \ref{thm4.7}). In the final section, we prove the existence of a unique ergodic and strongly mixing invariant measure for the SCBF  equations driven by multiplicative jump noise by making use of the exponential stability of strong solutions (Theorem \ref{UEIM}).

\section{Mathematical Formulation}\label{sec3}\setcounter{equation}{0}
The necessary function spaces needed to obtain the global solvability results of the system \eqref{1} is provided in this section. We prove monotonicity as well as hemicontinuity properties of the linear and nonlinear operators in this section.   In our analysis, the parameter $\alpha$ does not play a major role and we set $\alpha$ to be zero in \eqref{1} in the entire paper.

\subsection{Function spaces} Let $\C_0^{\infty}(\mathcal{O};\R^n)$ be the space of all infinitely differentiable functions  ($\R^n$-valued) with compact support in $\mathcal{O}\subset\R^n$.  Let us define 
\begin{align*} 
\mathcal{V}&:=\{\u\in\C_0^{\infty}(\mathcal{O},\R^n):\nabla\cdot\u=0\},\\
\mathbb{H}&:=\text{the closure of }\ \mathcal{V} \ \text{ in the Lebesgue space } \L^2(\mathcal{O})=\mathrm{L}^2(\mathcal{O};\R^n),\\
\mathbb{V}&:=\text{the closure of }\ \mathcal{V} \ \text{ in the Sobolev space } \H_0^1(\mathcal{O})=\mathrm{H}_0^1(\mathcal{O};\R^n),\\
\widetilde{\L}^{p}&:=\text{the closure of }\ \mathcal{V} \ \text{ in the Lebesgue space } \L^p(\mathcal{O})=\mathrm{L}^p(\mathcal{O};\R^n),
\end{align*}
for $p\in(2,\infty)$. Then under some smoothness assumptions on the boundary, we characterize the spaces $\H$, $\V$ and $\widetilde{\L}^p$ as 
$
\H=\{\u\in\L^2(\mathcal{O}):\nabla\cdot\u=0,\u\cdot\mathbf{n}\big|_{\partial\mathcal{O}}=0\}$,  with norm  $\|\u\|_{\H}^2:=\int_{\mathcal{O}}|\u(x)|^2\d x,
$
where $\mathbf{n}$ is the outward normal to $\partial\mathcal{O}$,
$
\V=\{\u\in\H_0^1(\mathcal{O}):\nabla\cdot\u=0\},$  with norm $ \|\u\|_{\V}^2:=\int_{\mathcal{O}}|\nabla\u(x)|^2\d x,
$ and $\widetilde{\L}^p=\{\u\in\L^p(\mathcal{O}):\nabla\cdot\u=0, \u\cdot\mathbf{n}\big|_{\partial\mathcal{O}}\},$ with norm $\|\u\|_{\widetilde{\L}^p}^p=\int_{\mathcal{O}}|\u(x)|^p\d x$, respectively.
Let $(\cdot,\cdot)$ denotes the inner product in the Hilbert space $\H$ and $\langle \cdot,\cdot\rangle $ denotes the induced duality between the spaces $\V$  and its dual $\V'$ as well as $\widetilde{\L}^p$ and its dual $\widetilde{\L}^{p'}$, where $\frac{1}{p}+\frac{1}{p'}=1$. Note that $\H$ can be identified with its dual $\H'$. We endow the space $\V\cap\widetilde{\L}^{p}$ with the norm $\|\u\|_{\V}+\|\u\|_{\widetilde{\L}^{p}},$ for $\u\in\V\cap\widetilde{\L}^p$ and its dual $\V'+\widetilde{\L}^{p'}$ with the norm $$\inf\left\{\max\left(\|\v_1\|_{\V'},\|\v_1\|_{\widetilde{\L}^{p'}}\right):\v=\v_1+\v_2, \ \v_1\in\V', \ \v_2\in\widetilde{\L}^{p'}\right\}.$$ Moreover, we have the continuous embedding $\V\cap\widetilde{\L}^p\hookrightarrow\H\hookrightarrow\V'+\widetilde{\L}^{p'}$. For the functional set up in periodic domains, interested readers are referred to see \cite{Te1,MSS,KWH}, etc. 

\subsection{Linear operator}
Let $\mathrm{P}_{\H} : \L^2(\mathcal{O}) \to\H$ denotes the \emph{Helmholtz-Hodge orthogonal projection} (see \cite{OAL,AC}). Let us define
\begin{equation*}
\left\{
\begin{aligned}
\A\u:&=-\mathrm{P}_{\H}\Delta\u,\;\u\in\D(\A),\\ \D(\A):&=\V\cap\H^{2}(\mathcal{O}).
\end{aligned}
\right.
\end{equation*}
It can be easily seen that the operator $\A$ is a non-negative self-adjoint operator in $\H$ with $\V=\D(\A^{1/2})$ and \begin{align}\label{2.7a}\langle \A\u,\u\rangle =\|\u\|_{\V}^2,\ \textrm{ for all }\ \u\in\V, \ \text{ so that }\ \|\A\u\|_{\V'}\leq \|\u\|_{\V}.\end{align}
For a bounded domain $\mathcal{O}$, the operator $\A$ is invertible and its inverse $\A^{-1}$ is bounded, self-adjoint and compact in $\H$. Thus, using spectral theorem, the spectrum of $\A$ consists of an infinite sequence $0< \uplambda_1\leq \lambda_2\leq\ldots\leq \lambda_k\leq \ldots,$ with $\lambda_k\to\infty$ as $k\to\infty$ of eigenvalues. 
Moreover, there exists an orthogonal basis $\{w_k\}_{k=1}^{\infty} $ of $\H$ consisting of eigenvectors of $\A$ such that $\A w_k =\lambda_kw_k$,  for all $ k\in\mathbb{N}$.  We know that $\u$ can be expressed as $\u=\sum_{k=1}^{\infty}\langle\u,w_k\rangle w_k$ and $\A\u=\sum_{k=1}^{\infty}\lambda_k\langle\u,w_k\rangle w_k$. Thus, it is immediate that 
\begin{align}\label{poin}
\|\nabla\u\|_{\mathbb{H}}^2=\langle \A\u,\u\rangle =\sum_{k=1}^{\infty}\lambda_k|\langle \u,w_k\rangle|^2\geq \uplambda_1\sum_{k=1}^{\infty}|\langle\u,w_k\rangle|^2=\uplambda_1\|\u\|_{\mathbb{H}}^2.
\end{align}

It should be noted that, in this work, we are not using the Gagliardo-Nirenberg, Ladyzhenskaya or Agmon's inequalities. Thus, the results obtained in this work are true for $2\leq n\leq 4$ in bounded domains (see the discussions above \eqref{333}) and $n\geq 2$ in periodic domains. The  following  interpolation inequality is also frequently in the paper.
Assume $1\leq s\leq r\leq t\leq \infty$, $\theta\in(0,1)$ such that $\frac{1}{r}=\frac{\theta}{s}+\frac{1-\theta}{t}$ and $\u\in\L^s(\mathcal{O})\cap\L^t(\mathcal{O})$, then we have 
\begin{align}\label{211}
\|\u\|_{\L^r}\leq\|\u\|_{\L^s}^{\theta}\|\u\|_{\L^t}^{1-\theta}. 
\end{align}

\subsection{Bilinear operator}
Let us define the \emph{trilinear form} $b(\cdot,\cdot,\cdot):\V\times\V\times\V\to\R$ by $$b(\u,\v,\w)=\int_{\mathcal{O}}(\u(x)\cdot\nabla)\v(x)\cdot\w(x)\d x=\sum_{i,j=1}^n\int_{\mathcal{O}}\u_i(x)\frac{\partial \v_j(x)}{\partial x_i}\w_j(x)\d x.$$ If $\u, \v$ are such that the linear map $b(\u, \v, \cdot) $ is continuous on $\V$, the corresponding element of $\V'$ is denoted by $\B(\u, \v)$. We also denote (with an abuse of notation) $\B(\u) = \B(\u, \u)=\mathrm{P}_{\H}(\u\cdot\nabla)\u$.
An integration by parts gives 
\begin{equation}\label{b0}
\left\{
\begin{aligned}
b(\u,\v,\v) &= 0,\text{ for all }\u,\v \in\V,\\
b(\u,\v,\w) &=  -b(\u,\w,\v),\text{ for all }\u,\v,\w\in \V.
\end{aligned}
\right.\end{equation}
In the trilinear form, an application of H\"older's inequality yields
\begin{align*}
|b(\u,\v,\w)|=|b(\u,\w,\v)|\leq \|\u\|_{\widetilde{\L}^{r+1}}\|\v\|_{\widetilde{\L}^{\frac{2(r+1)}{r-1}}}\|\w\|_{\V},
\end{align*}
for all $\u\in\V\cap\widetilde{\L}^{r+1}$, $\v\in\V\cap\widetilde{\L}^{\frac{2(r+1)}{r-1}}$ and $\w\in\V$, so that we get 
\begin{align}\label{2p9}
\|\B(\u,\v)\|_{\V'}\leq \|\u\|_{\widetilde{\L}^{r+1}}\|\v\|_{\widetilde{\L}^{\frac{2(r+1)}{r-1}}}.
\end{align}
Hence, the trilinear map $b : \V\times\V\times\V\to \R$ has a unique extension to a bounded trilinear map from $(\V\cap\widetilde{\L}^{r+1})\times(\V\cap\widetilde{\L}^{\frac{2(r+1)}{r-1}})\times\V$ to $\R$. It can also be seen that $\B$ maps $ \V\cap\widetilde{\L}^{r+1}$  into $\V'+\widetilde{\L}^{\frac{r+1}{r}}$ and using interpolation inequality (see \eqref{211}), we get 
\begin{align}\label{212}
\left|\langle \B(\u,\u),\v\rangle \right|=\left|b(\u,\v,\u)\right|\leq \|\u\|_{\widetilde{\L}^{r+1}}\|\u\|_{\widetilde{\L}^{\frac{2(r+1)}{r-1}}}\|\v\|_{\V}\leq\|\u\|_{\widetilde{\L}^{r+1}}^{\frac{r+1}{r-1}}\|\u\|_{\H}^{\frac{r-3}{r-1}},
\end{align}
for all $\v\in\V\cap\widetilde{\L}^{r+1}$. Thus, we have 
\begin{align}\label{2.9a}
\|\B(\u)\|_{\V'+\widetilde{\L}^{\frac{r+1}{r}}}\leq\|\u\|_{\widetilde{\L}^{r+1}}^{\frac{r+1}{r-1}}\|\u\|_{\H}^{\frac{r-3}{r-1}}.
\end{align}
Using \eqref{2p9}, for $\u,\v\in\V\cap\widetilde{\L}^{r+1}$, we also have 
\begin{align}\label{lip}
\|\B(\u)-\B(\v)\|_{\V'+\widetilde{\L}^{\frac{r+1}{r}}}&\leq \|\B(\u-\v,\u)\|_{\V'}+\|\B(\v,\u-\v)\|_{\V'}\nonumber\\&\leq \left(\|\u\|_{\widetilde{\L}^{\frac{2(r+1)}{r-1}}}+\|\v\|_{\widetilde{\L}^{\frac{2(r+1)}{r-1}}}\right)\|\u-\v\|_{\widetilde{\L}^{r+1}}\nonumber\\&\leq \left(\|\u\|_{\H}^{\frac{r-3}{r-1}}\|\u\|_{\widetilde{\L}^{r+1}}^{\frac{2}{r-1}}+\|\v\|_{\H}^{\frac{r-3}{r-1}}\|\v\|_{\widetilde{\L}^{r+1}}^{\frac{2}{r-1}}\right)\|\u-\v\|_{\widetilde{\L}^{r+1}},
\end{align}
for $r>3$, by using the interpolation inequality. For $r=3$, a calculation similar to \eqref{lip} yields 
\begin{align}
\|\B(\u)-\B(\v)\|_{\V'+\widetilde{\L}^{\frac{4}{3}}}&\leq \left(\|\u\|_{\widetilde{\L}^{4}}+\|\v\|_{\widetilde{\L}^{4}}\right)\|\u-\v\|_{\widetilde{\L}^{4}},
\end{align}
hence $\B(\cdot):\V\cap\widetilde{\L}^{4}\to\V'+\widetilde{\L}^{\frac{4}{3}}$ is a locally Lipschitz operator. 
\subsection{Nonlinear operator}
Let us now consider the operator $\mathcal{C}(\u):=\P_{\H}(|\u|^{r-1}\u)$. It is immediate that $\langle\mathcal{C}(\u),\u\rangle =\|\u\|_{\widetilde{\L}^{r+1}}^{r+1}$ and the map $\mathcal{C}(\cdot):\widetilde{\L}^{r+1}\to\widetilde{\L}^{\frac{r+1}{r}}$ is Gateaux differentiable with Gateaux derivative $\mathcal{C}'(\u)\v=r\mathrm{P}_{\H}(|\u|^{r-1}\v),$ for $\v\in\widetilde{\L}^{r+1}$. For  $\u,\v\in\widetilde{\L}^{r+1}$, using Taylor's formula, we have 
\begin{align}\label{213}
\langle \P_{\H}(|\u|^{r-1}\u)-\P_{\H}(|\v|^{r-1}\v),\w\rangle&\leq \|(|\u|^{r-1}\u)-(|\v|^{r-1}\v)\|_{\widetilde{\L}^{\frac{r+1}{r}}}\|\w\|_{\widetilde{\L}^{r+1}}\nonumber\\&\leq\sup_{0<\theta<1}r\|(\u-\v)|\theta\u+(1-\theta)\v|^{r-1}\|_{\widetilde{\L}^{{\frac{r+1}{r}}}}\|\w\|_{\widetilde{\L}^{r+1}}\nonumber\\&\leq r\sup_{0<\theta<1}\|\theta\u+(1-\theta)\v\|_{\widetilde{\L}^{r+1}}^{r-1}\|\u-\v\|_{\widetilde{\L}^{r+1}}\|\w\|_{\widetilde{\L}^{r+1}}\nonumber\\&\leq r\left(\|\u\|_{\widetilde{\L}^{r+1}}+\|\v\|_{\widetilde{\L}^{r+1}}\right)^{r-1}\|\u-\v\|_{\widetilde{\L}^{r+1}}\|\w\|_{\widetilde{\L}^{r+1}},
\end{align}
for all $\u,\v,\w\in\widetilde{\L}^{r+1}$. 
Thus the operator $\mathcal{C}(\cdot):\widetilde{\L}^{r+1}\to\widetilde{\L}^{\frac{r+1}{r}}$ is locally Lipschitz. Moreover, 	for any $r\in[1,\infty)$, we have 
\begin{align}\label{2pp11}
&\langle \mathrm{P}_{\H}(\u|\u|^{r-1})-\mathrm{P}_{\H}(\v|\v|^{r-1}),\u-\v\rangle\nonumber\\&=\int_{\mathcal{O}}\left(\u(x)|\u(x)|^{r-1}-\v(x)|\v(x)|^{r-1}\right)\cdot(\u(x)-\v(x))\d x\nonumber\\&=\int_{\mathcal{O}}\left(|\u(x)|^{r+1}-|\u(x)|^{r-1}\u(x)\cdot\v(x)-|\v(x)|^{r-1}\u(x)\cdot\v(x)+|\v(x)|^{r+1}\right)\d x\nonumber\\&\geq\int_{\mathcal{O}}\left(|\u(x)|^{r+1}-|\u(x)|^{r}|\v(x)|-|\v(x)|^{r}|\u(x)|+|\v(x)|^{r+1}\right)\d x\nonumber\\&=\int_{\mathcal{O}}\left(|\u(x)|^r-|\v(x)|^r\right)(|\u(x)|-|\v(x)|)\d x\geq 0. 
\end{align}
Furthermore, we find 
\begin{align}\label{224}
&\langle\mathrm{P}_{\H}(\u|\u|^{r-1})-\mathrm{P}_{\H}(\v|\v|^{r-1}),\u-\v\rangle\nonumber\\&=\langle|\u|^{r-1},|\u-\v|^2\rangle+\langle|\v|^{r-1},|\u-\v|^2\rangle+\langle\v|\u|^{r-1}-\u|\v|^{r-1},\u-\v\rangle\nonumber\\&=\||\u|^{\frac{r-1}{2}}(\u-\v)\|_{\H}^2+\||\v|^{\frac{r-1}{2}}(\u-\v)\|_{\H}^2\nonumber\\&\quad+\langle\u\cdot\v,|\u|^{r-1}+|\v|^{r-1}\rangle-\langle|\u|^2,|\v|^{r-1}\rangle-\langle|\v|^2,|\u|^{r-1}\rangle.
\end{align}
But, we know that 
\begin{align*}
&\langle\u\cdot\v,|\u|^{r-1}+|\v|^{r-1}\rangle-\langle|\u|^2,|\v|^{r-1}\rangle-\langle|\v|^2,|\u|^{r-1}\rangle\nonumber\\&=-\frac{1}{2}\||\u|^{\frac{r-1}{2}}(\u-\v)\|_{\H}^2-\frac{1}{2}\||\v|^{\frac{r-1}{2}}(\u-\v)\|_{\H}^2+\frac{1}{2}\langle\left(|\u|^{r-1}-|\v|^{r-1}\right),\left(|\u|^2-|\v|^2\right)\rangle \nonumber\\&\geq -\frac{1}{2}\||\u|^{\frac{r-1}{2}}(\u-\v)\|_{\H}^2-\frac{1}{2}\||\v|^{\frac{r-1}{2}}(\u-\v)\|_{\H}^2.
\end{align*}
From \eqref{224}, we finally have 
\begin{align}\label{2.23}
&\langle\mathrm{P}_{\H}(\u|\u|^{r-1})-\mathrm{P}_{\H}(\v|\v|^{r-1}),\u-\v\rangle\geq \frac{1}{2}\||\u|^{\frac{r-1}{2}}(\u-\v)\|_{\H}^2+\frac{1}{2}\||\v|^{\frac{r-1}{2}}(\u-\v)\|_{\H}^2\geq 0,
\end{align}
for $r\geq 1$. 
\subsection{Monotonicity}
Let us now show the monotonicity as well as the hemicontinuity properties of the linear and nonlinear operators, which plays a crucial role in this paper. 
\begin{definition}[\cite{VB}]
	Let $\X$ be a Banach space and let $\X^{'}$ be its topological dual.
	An operator $\G:\mathrm{D}\rightarrow
	\X^{'},$ $\mathrm{D}=\mathrm{D}(\G)\subset \X$ is said to be
	\emph{monotone} if
	$$\langle\G(x)-\G(y),x-y\rangle\geq
	0,\ \text{ for all } \ x,y\in \mathrm{D}.$$ 
	The operator $\G(\cdot)$ is said to be \emph{hemicontinuous}, if for all $x, y\in\X$ and $w\in\X',$ $$\lim_{\lambda\to 0}\langle\G(x+\lambda y),w\rangle=\langle\G(x),w\rangle.$$
	The operator $\G(\cdot)$ is called \emph{demicontinuous}, if for all $x\in\mathrm{D}$ and $y\in\X$, the functional $x \mapsto\langle \G(x), y\rangle$  is continuous, or in other words, $x_k\to x$ in $\X$ implies $\G(x_k)\xrightarrow{w}\G(x)$ in $\X'$. Clearly demicontinuity implies hemicontinuity. 
\end{definition}
\begin{theorem}\label{thm2.2}
	Let $\u,\v\in\V\cap\widetilde{\L}^{r+1}$, for $r>3$. Then,	for the operator $\G(\u)=\mu \A\u+\B(\u)+\beta\mathcal{C}(\u)$, we  have 
	\begin{align}\label{fe}
	\langle(\G(\u)-\G(\v),\u-\v\rangle+\eta\|\u-\v\|_{\H}^2\geq 0,
	\end{align}
	where \begin{align}\label{215}\eta=\frac{r-3}{2\mu(r-1)}\left(\frac{2}{\beta\mu (r-1)}\right)^{\frac{2}{r-3}}.\end{align} That is, the operator $\G+\eta\mathrm{I}$ is a monotone operator from $\V\cap\widetilde{\L}^{r+1}$ to $\V'+\widetilde{\L}^{\frac{r+1}{r}}$. 
\end{theorem}
\begin{proof}
	We estimate $	\langle\A\u-\A\v,\u-\v\rangle $ by	using an integration by parts as
	\begin{align}\label{ae}
	\langle\A\u-\A\v,\u-\v\rangle =\|\u-\v\|^2_{\V}.
	\end{align}
	From \eqref{2.23}, we easily have 
	\begin{align}\label{2.27}
	\beta	\langle\mathcal{C}(\u)-\mathcal{C}(\v),\u-\v\rangle \geq \frac{\beta}{2}\||\v|^{\frac{r-1}{2}}(\u-\v)\|_{\H}^2. 
	\end{align}
	Note that $\langle\B(\u,\u-\v),\u-\v\rangle=0$ and it implies that
	\begin{equation}\label{441}
	\begin{aligned}
	\langle \B(\u)-\B(\v),\u-\v\rangle &=\langle\B(\u,\u-\v),\u-\v\rangle +\langle \B(\u-\v,\v),\u-\v\rangle \nonumber\\&=\langle\B(\u-\v,\v),\u-\v\rangle=-\langle\B(\u-\v,\u-\v),\v\rangle.
	\end{aligned}
	\end{equation} 
Using H\"older's and Young's inequalities, we estimate $|\langle\B(\u-\v,\u-\v),\v\rangle|$ as  
\begin{align}\label{2p28}
|\langle\B(\u-\v,\u-\v),\v\rangle|&\leq\|\u-\v\|_{\V}\|\v(\u-\v)\|_{\H}\nonumber\\&\leq\frac{\mu }{2}\|\u-\v\|_{\V}^2+\frac{1}{2\mu }\|\v(\u-\v)\|_{\H}^2.
\end{align}
We take the term $\|\v(\u-\v)\|_{\H}^2$ from \eqref{2p28} and use H\"older's and Young's inequalities to estimate it as (see \cite{KWH} also)
\begin{align}\label{2.29}
&\int_{\mathcal{O}}|\v(x)|^2|\u(x)-\v(x)|^2\d x\nonumber\\&=\int_{\mathcal{O}}|\v(x)|^2|\u(x)-\v(x)|^{\frac{4}{r-1}}|\u(x)-\v(x)|^{\frac{2(r-3)}{r-1}}\d x\nonumber\\&\leq\left(\int_{\mathcal{O}}|\v(x)|^{r-1}|\u(x)-\v(x)|^2\d x\right)^{\frac{2}{r-1}}\left(\int_{\mathcal{O}}|\u(x)-\v(x)|^2\d x\right)^{\frac{r-3}{r-1}}\nonumber\\&\leq{\beta\mu }\left(\int_{\mathcal{O}}|\v(x)|^{r-1}|\u(x)-\v(x)|^2\d x\right)+\frac{r-3}{r-1}\left(\frac{2}{\beta\mu (r-1)}\right)^{\frac{2}{r-3}}\left(\int_{\mathcal{O}}|\u(x)-\v(x)|^2\d x\right),
\end{align}
for $r>3$. Using \eqref{2.29} in \eqref{2p28}, we find 
\begin{align}\label{2.30}
&|\langle\B(\u-\v,\u-\v),\v\rangle|\nonumber\\&\leq\frac{\mu }{2}\|\u-\v\|_{\V}^2+\frac{\beta}{2}\||\v|^{\frac{r-1}{2}}(\u-\v)\|_{\H}^2+\frac{r-3}{2\mu(r-1)}\left(\frac{2}{\beta\mu (r-1)}\right)^{\frac{2}{r-3}}\|\u-\v\|_{\H}^2.
\end{align}
Combining \eqref{ae}, \eqref{2.27} and \eqref{2.30}, we get 
\begin{align}
	\langle(\G(\u)-\G(\v),\u-\v\rangle+\frac{r-3}{2\mu(r-1)}\left(\frac{2}{\beta\mu (r-1)}\right)^{\frac{2}{(r-3)}}\|\u-\v\|_{\H}^2\geq\frac{\mu }{2}\|\u-\v\|_{\V}^2\geq 0,
\end{align}
for $r>3$ and the estimate \eqref{fe} follows.  
\end{proof} 
\begin{theorem}\label{thm2.3}
	For the critical case $r=3$ with $2\beta\mu \geq 1$, the operator $\G(\cdot):\V\cap\widetilde{\L}^{r+1}\to \V'+\widetilde{\L}^{\frac{r+1}{r}}$ is globally monotone, that is, for all $\u,\v\in\V$, we have 
\begin{align}\label{218}\langle\G(\u)-\G(\v),\u-\v\rangle\geq 0.\end{align}
	\end{theorem}
\begin{proof}
From \eqref{2.23}, we have 
	\begin{align}\label{231}
\beta\langle\mathcal{C}(\u)-\mathcal{C}(\v),\u-\v\rangle\geq\frac{\beta}{2}\|\v(\u-\v)\|_{\H}^2. 
\end{align}
We estimate $|\langle\B(\u-\v,\u-\v),\v\rangle|$ using H\"older's and Young's inequalities as 
\begin{align}\label{232}
|\langle\B(\u-\v,\u-\v),\v\rangle|\leq\|\v(\u-\v)\|_{\H}\|\u-\v\|_{\V} \leq\mu \|\u-\v\|_{\V}^2+\frac{1}{4\mu }\|\v(\u-\v)\|_{\H}^2.
\end{align}
Combining \eqref{ae}, \eqref{231} and \eqref{232}, we obtain 
\begin{align}
	\langle\G(\u)-\G(\v),\u-\v\rangle\geq\frac{1}{2}\left(\beta-\frac{1}{2\mu }\right)\|\v(\u-\v)\|_{\H}^2\geq 0,
\end{align}
provided $2\beta\mu \geq 1$. 
\end{proof}
\begin{remark}
1. 	As in \cite{YZ}, for $r\geq 3$, one can estimate $|\langle\B(\u-\v,\u-\v),\v\rangle|$ as 
	\begin{align}\label{2.26}
&	|\langle\B(\u-\v,\u-\v),\v\rangle|\nonumber\\&\leq \mu \|\u-\v\|_{\V}^2+\frac{1}{4\mu }\int_{\mathcal{O}}|\v(x)|^2|\u(x)-\v(x)|^2\d x\nonumber\\&= \mu \|\u-\v\|_{\V}^2+\frac{1}{4\mu }\int_{\mathcal{O}}|\u(x)-\v(x)|^2\left(|\v(x)|^{r-1}+1\right)\frac{|\v(x)|^2}{|\v(x)|^{r-1}+1}\d x\nonumber\\&\leq \mu \|\u-\v\|_{\V}^2+\frac{1}{4\mu }\int_{\mathcal{O}}|\v(x)|^{r-1}|\u(x)-\v(x)|^2\d x+\frac{1}{4\mu }\int_{\mathcal{O}}|\u(x)-\v(x)|^2\d x,
	\end{align}
	where we used the fact that $\left\|\frac{|\v|^2}{|\v|^{r-1}+1}\right\|_{\widetilde{\L}^{\infty}}<1$, for $r\geq 3$. The above estimate also yields the global monotonicity result given in \eqref{218}, provided $2\beta\mu \geq 1$. 
	
	2. For $n=2$ and $r= 3$, one can estimate $|\langle\B(\u-\v,\u-\v),\v\rangle|$ using H\"older's, Ladyzhenskaya and Young's inequalities  as 
	\begin{align}\label{2.21}
	|\langle\B(\u-\v,\u-\v),\v\rangle|&\leq \|\u-\v\|_{\wi\L^4}\|\u-\v\|_{\V}\|\v\|_{\wi\L^4}\nonumber\\&\leq 2^{1/4}\|\u-\v\|_{\H}^{1/2}\|\u-\v\|_{\V}^{3/2}\|\v\|_{\wi\L^4}\nonumber\\&\leq  \frac{\mu }{2}\|\u-\v\|_{\V}^2+\frac{27}{32\mu ^3}\|\v\|_{\widetilde{\L}^4}^4\|\u-\v\|_{\H}^2.
	\end{align}
	Combining \eqref{ae}, \eqref{2.27} and \eqref{2.21}, we obtain 
	\begin{align}\label{fe2}
	\langle(\G(\u)-\G(\v),\u-\v\rangle+ \frac{27}{32\mu ^3}N^4\|\u-\v\|_{\H}^2\geq 0,
	\end{align}
	for all $\v\in\widehat{\mathbb{B}}_N$, where $\widehat{\mathbb{B}}_N$ is an $\widetilde{\L}^4$-ball of radius $N$, that is,
	$
	\widehat{\mathbb{B}}_N:=\big\{\z\in\widetilde{\L}^4:\|\z\|_{\widetilde{\L}^4}\leq N\big\}.
	$ Thus, the operator $\G(\cdot)$ is locally monotone in this case (see \cite{MJSS,MTM}, etc). 
	\end{remark}

Let us now show that the operator $\G:\V\cap\widetilde{\L}^{r+1}\to \V'+\widetilde{\L}^{\frac{r+1}{r}}$ is hemicontinuous, which is useful in proving the global solvability of the system \eqref{1}. 
\begin{lemma}\label{lem2.8}
	The operator $\G:\V\cap\widetilde{\L}^{r+1}\to \V'+\widetilde{\L}^{\frac{r+1}{r}}$ is demicontinuous. 
\end{lemma}
\begin{proof}
	Let us take a sequence $\u^n\to \u$ in $\V\cap\widetilde{\L}^{r+1}$, that is, $\|\u^n-\u\|_{\wi\L^{r+1}}+\|\u^n-\u\|_{\V}\to 0$, as $n\to\infty$. For any $\v\in\V\cap\widetilde{\L}^{r+1}$, we consider 
	\begin{align}\label{214}
	\langle\G(\u^n)-\G(\u),\v\rangle &=\mu \langle \A(\u^n)-\A(\u),\v\rangle+\langle\B(\u^n)-\B(\u),\v\rangle-\beta\langle \mathcal{C}(\u^n)-\mathcal{C}(\u),\v\rangle.
	\end{align} 
	Let us take $\langle \A(\u^n)-\A(\u),\v\rangle$ from \eqref{214} and estimate it as 
	\begin{align}
	|\langle \A(\u^n)-\A(\u),\v\rangle|=|(\nabla(\u^n-\u),\nabla\v)|\leq\|\u^n-\u\|_{\V}\|\v\|_{\V}\to 0, \ \text{ as } \ n\to\infty, 
	\end{align}
	since $\u^n\to \u$ in $\V$. We estimate the term $\langle\B(\u^n)-\B(\u),\v\rangle$ from \eqref{214} using H\"older's inequality as 
	\begin{align}
	|\langle\B(\u^n)-\B(\u),\v\rangle|&=|\langle\B(\u^n,\u^n-\u),\v\rangle+\langle\B(\u^n-\u,\u),\v\rangle|\nonumber\\&
	\leq|\langle\B(\u^n,\v),\u^n-\u\rangle|+|\langle\B(\u^n-\u,\v),\u\rangle|\nonumber\\&\leq\left(\|\u^n\|_{\widetilde{\L}^{\frac{2(r+1)}{r-1}}}+\|\u\|_{\widetilde{\L}^{\frac{2(r+1)}{r-1}}}\right)\|\u^n-\u\|_{\widetilde{\L}^{r+1}}\|\v\|_{\V}\nonumber\\&\leq \left(\|\u^n\|_{\H}^{\frac{r-3}{r-1}}\|\u^n\|_{\widetilde{\L}^{r+1}}^{\frac{2}{r-1}}+\|\u\|_{\H}^{\frac{r-3}{r-1}}\|\u\|_{\widetilde{\L}^{r+1}}^{\frac{2}{r-1}}\right)\|\u^n-\u\|_{\widetilde{\L}^{r+1}}\|\v\|_{\V}\nonumber\\& \to 0, \ \text{ as } \ n\to\infty, 
	\end{align}
		since $\u^n\to\u$ in $\wi\L^{r+1}$ and $\u^n,\u\in\V\cap\wi\L^{r+1}$. Finally, we estimate the term $\langle \mathcal{C}(\u^n)-\mathcal{C}(\u),\v\rangle$ from \eqref{214} using Taylor's formula and H\"older's inequality as 
		\begin{align}
		|\langle \mathcal{C}(\u^n)-\mathcal{C}(\u),\v\rangle|&\leq \sup_{0<\theta<1}r\|(\u^n-\u)|\theta\u^n+(1-\theta)\u|^{r-1}\|_{\widetilde{\L}^{\frac{r+1}{r}}}\|\v\|_{\widetilde{\L}^{r+1}}\nonumber\\&\leq r\|\u^n-\u\|_{\widetilde{\L}^{r+1}}\left(\|\u^n\|_{\widetilde{\L}^{r+1}}+\|\u\|_{\widetilde{\L}^{r+1}}\right)^{r-1}\|\v\|_{\widetilde{\L}^{r+1}}\to 0, \ \text{ as } \ n\to\infty, 
		\end{align}
		since $\u^n\to\u$ in $\widetilde{\L}^{r+1}$ and $\u^n,\u\in\V\cap\widetilde{\L}^{r+1}$, for $r\geq 3$. From the above convergences, it is immediate that $\langle\G(\u^n)-\G(\u),\v\rangle \to 0$, for all $\v\in \V\cap\widetilde{\L}^{r+1}$.
Hence the operator $\G:\V\cap\widetilde{\L}^{r+1}\to \V'+\widetilde{\L}^{\frac{r+1}{r}}$ is demicontinuous, which implies that the operator $\G(\cdot)$ is also hemicontinuous. 
\end{proof}

\section{Stochastic Navier-Stokes-Brinkman-Forchheimer  equations}\label{sec4}\setcounter{equation}{0} In this section, we consider the following stochastic  convective Brinkman-Forchheimer  equations perturbed by multiplicative jump noise:
\begin{equation}\label{31}
\left\{
\begin{aligned}
\d\u(t)-\mu \Delta\u(t)&+(\u(t)\cdot\nabla)\u(t)+\beta|\u(t)|^{r-1}\u(t)+\nabla p(t)\\&=\int_{\Z}\gamma(t-,\u(t-),z)\widetilde{\pi}(\d t,\d z), \ \text{ in } \ \mathcal{O}\times(0,T), \\ \nabla\cdot\u(t)&=0, \ \text{ in } \ \mathcal{O}\times(0,T), \\
\u(t)&=\mathbf{0}\ \text{ on } \ \partial\mathcal{O}\times(0,T), \\
\u(0)&=\u_0 \ \text{ in } \ \mathcal{O},
\end{aligned}
\right.
\end{equation} 
where $\widetilde{\pi}(\cdot,\cdot)$ is the compensated Poisson random measure.

\subsection{Stochastic setting}
Let $(\Omega,\mathscr{F},\mathbb{P})$ be a complete probability space equipped with an increasing family of sub-sigma fields $\{\mathscr{F}_t\}_{0\leq t\leq T}$ of $\mathscr{F}$ satisfying  
\begin{enumerate}
	\item [(i)] $\mathscr{F}_0$ contains all elements $F\in\mathscr{F}$ with $\mathbb{P}(F)=0$,
	\item [(ii)] $\mathscr{F}_t=\mathscr{F}_{t+}=\bigcap\limits_{s>t}\mathscr{F}_s$, for $0\leq t\leq T$.
\end{enumerate}

Let $(\Z,\mathscr{B}(\Z))$ be a measurable space and let $\lambda$ be a $\sigma$-finite positive measure on it. Let $\pi: \Omega\times\mathscr{B}(\R^+)\times\mathscr{B}(\Z)\to\N\cup\{0\}$ be a time homogeneous Poisson random measure with intensity measure $\lambda$ defined over the probability space  $(\Omega,\mathscr{F},\{\mathscr{F}_t\}_{t\geq 0},\mathbb{P})$. The intensity measure $\lambda(\cdot)$ on $\Z$ satisfies the conditions $\lambda(\{0\})=0$   and 
\begin{equation*}
\int_{\Z} \left(1\wedge |z|^p\right)  \lambda(\d z)< +\infty,\ \ p\geq 2.
\end{equation*}
We denote $\widetilde{\pi}=\pi-\upgamma$ as the compensated Poisson random measure associated to $\lambda$, where the compensator $\upgamma$ is given by $\mathscr{B}(\R^+)\times\mathscr{B}(\Z)\ni (A,I)\mapsto\upgamma(A,I)=\lambda(A)\d(I)\in\R^+$, where $\d(\cdot)$ is the Lebesgue measure.

Let  $\gamma:[0,T]\times \H\times \Z \to \H$ be a measurable and $\mathscr{F}_t$-adapted process satisfying 
\[\E\left[\left\|\int_0^T\int_{\Z} \gamma(t-,\u(t-),z)\wi\pi(\d t,\d z)\right\|_{\H}^2\right]<+\infty,\] for all $\u\in\H$. The  integral   $\M(t):=\int_0^t\int_{\Z} \gamma(s-,\u(s-),z)\wi\pi(\d s,\d z)$ is an  $\H$-valued martingale and there exists an increasing c\`adl\`ag process so-called \emph{quadratic variation process} $[\M]_t$ and \emph{Meyer process} $\langle \M\rangle_t$ such that $[\M]_t-\langle \M\rangle_t$ is a local martingale (see, \cite{Me}).  Indeed, we have $\E\|\M(t)\|^2_{\H}=\E[\M]_t=\E\langle \M\rangle_t.$
Moreover, the following It\^o isometry holds:
$$\E\left[\left\|\int_0^T\int_{\Z} \gamma(t-,\u(t-),z)\wi\pi(\d t,\d z)\right\|_{\H}^2\right]=\E\left[\int_0^T\int_{\Z}   \|\gamma(t,\u(t),z)\|^2_{\H}\lambda(\d z)\d t \right],$$ for all $\u\in\H$.  For more details on jump processes, one may refer to \cite{Ap,PEP,VMBR}, etc.

Let us denote by $\D([0,T];\H)$, the set of all $\H$-valued
functions defined on $[0,T]$, which are right continuous and have
left limits (c\`{a}dl\`{a}g functions) for every $t\in[0,T]$. Also,
let
\begin{align}\label{33}
\mathrm{L}^2(\Omega;\mathrm{L}^2(0,T;\mathrm{L}^2_{\lambda}(\Z;\H))):=\mathrm{L}^{2}\left(\Omega\times(0,T]\times\Z,\mathscr{F}\times\mathscr{B}((0,T])\times\mathscr{B}(\Z),\mathbb{P}\otimes\d
t\otimes\lambda;\H\right),\end{align} be the space
of all $\mathscr{F}\times\mathscr{B}((0,T])\times\mathscr{B}(\Z)$ measurable
functions $\gamma:[0,T]\times\Omega\times\Z\to\H$ such that
$$\E
\left[\int_0^T\int_{\Z}\|\gamma(t,\cdot,z)\|_{\H}^{2}\lambda(\d
z)\d t\right]<+\infty.$$

\begin{hypothesis}\label{hyp}
	The noise coefficient $\gamma(\cdot,\cdot)$ satisfies: 
	\begin{itemize}
		\item [(H.1)] The function  $\gamma\in\mathrm{L}^2(\Omega;\mathrm{L}^2(0,T;\mathrm{L}^2_{\lambda}(\Z;\H)))$.
		\item[(H.2)]  (Growth condition)
		There exists a positive
		constant $K$ such that for all $t\in[0,T]$ and $\u\in\H$,
		\begin{equation*}
	\int_{\Z} \|\gamma(t,\u, z)\|^{2}_{\H}\lambda(\d z)	\leq K\left(1 +\|\u\|_{\H}^{2}\right),
		\end{equation*}
		
		\item[(H.3)]  (Lipschitz condition)
		There exists a positive constant $L$ such that for any $t\in[0,T]$ and all $\u_1,\u_2\in\H$,
		\begin{align*}
	\int_{\Z} \|\gamma(t,\u_1, z)-\gamma(t,\u_2, z)\|^2_{\H}\lambda(\d z)	\leq L\|\u_1 -	\u_2\|_{\H}^2,\end{align*} for all $\u_1, \u_2\in \H$.
		\item[(H.4)] We fix the measurable subsets $\Z_m$ of $\Z$ with
		$\Z_m\uparrow\Z$ and $\lambda(\Z_m)<+\infty$ such that
		$$\sup_{\|\u\|_{\H}\leq
			M}\int_{\Z^c_m}\|\gamma(t,\u,z)\|_{\H}^2\lambda(\d z)\to
		0,\textrm{ as } m\to\infty, \textrm{ for }M>0.$$
		\end{itemize}
\end{hypothesis}

\subsection{Abstract formulation of the stochastic system}\label{sec2.4}
On  taking orthogonal projection $\mathrm{P}_{\H}$ onto the first equation in \eqref{31}, we get 
\begin{equation}\label{32}
\left\{
\begin{aligned}
\d\u(t)+[\mu \A\u(t)+\B(\u(t))+\beta\mathcal{C}(\u(t))]\d t&=\int_{\Z}\gamma(t-,\u(t-),z)\widetilde{\pi}(\d t,\d z), \ t\in(0,T),\\
\u(0)&=\u_0,
\end{aligned}
\right.
\end{equation}
where $\u_0\in\mathrm{L}^2(\Omega;\H)$. Strictly speaking one should write  $\mathrm{P}_{\H}\gamma$ instead of $\gamma$. Let us now provide the definition of a unique global strong solution  to the system (\ref{32}).
\begin{definition}[Global strong solution]\label{defn3.2}
	Let $\u_0\in\mathrm{L}^2(\Omega;\H)$ be given. An $\H$-valued $(\mathscr{F}_t)_{t\geq 0}$-adapted stochastic process $\u(\cdot)$ is called a \emph{strong solution} to the system (\ref{32}) if the following conditions are satisfied: 
	\begin{enumerate}
		\item [(i)] the process $$\u\in\mathrm{L}^2(\Omega;\mathrm{L}^{\infty}(0,T;\H)\cap\mathrm{L}^2(0,T;\V))\cap\mathrm{L}^{r+1}(\Omega;\mathrm{L}^{r+1}(0,T;\widetilde{\L}^{r+1}))$$ and $\u(\cdot)$ has a $\V\cap\widetilde{\L}^{r+1}$-valued  modification, which is progressively measurable with c\`adl\`ag paths in $\H$ and $\u\in\D([0,T];\H)\cap\mathrm{L}^2(0,T;\V)\cap\mathrm{L}^{r+1}(0,T;\widetilde{\L}^{r+1})$, $\mathbb{P}$-a.s.,
		\item [(ii)] the following equality holds for every $t\in [0, T ]$, as an element of $\V'+\wi\L^{\frac{r+1}{r}},$ $\mathbb{P}$-a.s.
		\begin{align}\label{4.4}
		\u(t)&=\u_0-\int_0^t\left[\mu \A\u(s)+\B(\u(s))+\beta\mathcal{C}(\u(s))\right]\d s+\int_0^t\int_{\Z}\gamma(s-,\u(s-),z)\wi\pi(\d s,\d z).
		\end{align}
			\item [(iii)] the following It\^o formula holds true: 
				\begin{align}
			&	\|\u(t)\|_{\H}^2+2\mu \int_0^t\|\u(s)\|_{\V}^2\d s+2\beta\int_0^t\|\u(s)\|_{\widetilde{\L}^{r+1}}^{r+1}\d s\nonumber\\&=\|{\u_0}\|_{\H}^2+\int_0^t\|\gamma(s,\u(s),z)\|_{\H}^2\pi(\d s,\d z)+2\int_0^t\int_{\Z}(\gamma(s-,\u(s-),z),\u(s-))\wi\pi(\d s,\d z),
			\end{align}
			for all $t\in(0,T)$, $\mathbb{P}$-a.s. 
	\end{enumerate}
\end{definition}
An alternative version of condition (\ref{4.4}) is to require that for any  $\v\in\V\cap\widetilde{\L}^{r+1}$:
\begin{align}\label{4.5}
(\u(t),\v)&=(\u_0,\v)-\int_0^t\langle\mu \A\u(s)+\B(\u(s))+\beta\mathcal{C}(\u(s)),\v\rangle\d s\no\\&\quad+\int_0^t\int_{\Z}(\gamma(s-,\u(s-),z),\v)\wi\pi(\d s,\d z),\ \mathbb{P}\text{-a.s.}
\end{align}	
\begin{definition}
	A strong solution $\u(\cdot)$ to (\ref{32}) is called a
	\emph{pathwise  unique strong solution} if
	$\widetilde{\u}(\cdot)$ is an another strong
	solution, then $$\mathbb{P}\Big\{\omega\in\Omega:\u(t)=\widetilde{\u}(t),\ \text{ for all }\ t\in[0,T]\Big\}=1.$$ 
\end{definition}

\subsection{Energy estimates} In this subsection, we formulate a finite dimensional system and establish some a-priori energy estimates. Let $\{w_1,\ldots,w_n,\ldots\}$ be a complete orthonormal system in $\H$ belonging to $\V$ and let $\H_n$ be the $\mathrm{span}\{w_1,\ldots,w_n\}$. Let $\mathrm{P}_n$ denotes the orthogonal projection of $\V'$ to $\H_n$, that is, $\mathrm{P}_n\x=\sum_{i=1}^n\langle \x,w_i\rangle w_i$. Since every element $\x\in\H$ induces a functional $\x^*\in\H$  by the formula $\langle \x^*,\y\rangle =(\x,\y)$, $\y\in\V$, then $\mathrm{P}_n\big|_{\H}$, the orthogonal projection of $\H$ onto $\H_n$  is given by $\mathrm{P}_n\x=\sum_{i=1}^n\left(\x,w_i\right)w_i$. Hence in particular, $\mathrm{P}_n$ is the orthogonal projection from $\H$ onto $\text{span}\{w_1,\ldots,w_n\}$.  We define $\B^n(\u^n)=\mathrm{P}_n\B(\u^n)$, $\mathcal{C}^n(\u^n)=\mathrm{P}_n\mathcal{C}(\u^n)$  and $\gamma^n(\cdot,\u^n,\cdot)=\mathrm{P}_n\gamma(\cdot,\u^n,\cdot)$. We consider the following system of ODEs:
\begin{equation}\label{4.7}
\left\{
\begin{aligned}
\d\left(\u^n(t),\v\right)&=-\langle\mu \A\u^n(t)+\B^n(\u^n(t))+\beta\mathcal{C}^n(\u^n(t)),\v\rangle\d t\\&\quad+\int_{\Z_n}\left(\gamma^n(t-,\u_n(t-),z),\v\right)\widetilde{\uppi}(\d t,\d z),\\
\u^n(0)&=\u_0^n,
\end{aligned}
\right.
\end{equation}
with $\u_0^n=\mathrm{P}_n\u_0,$ for all $\v\in\H_n$. Since $\B^n(\cdot)$ and $\mathcal{C}^n(\cdot)$ are  locally Lipschitz (see \eqref{lip} and \eqref{213}), and  $\gamma^n(\cdot,\u^n,\cdot)$  is globally Lipschitz, the system (\ref{4.7}) has a unique $\H_n$-valued local strong solution $\u^n(\cdot)$ and $\u^n\in\mathrm{L}^2(\Omega;\mathrm{L}^{\infty}(0,T^*;\H_n))$ with $\mathscr{F}_t$-adapted c\`adl\`ag sample paths.  Now we discuss about the a-priori energy estimates satisfied by the solution to the system \eqref{4.7}. Note that the energy estimates established in the next proposition is true for $r\geq 1$. 
\begin{proposition}[Energy estimates]\label{prop1}
	Let $\u^n(\cdot)$ be the unique solution of the system of stochastic
	ODE's (\ref{4.7}) with $\u_0\in\mathrm{L}^{2}(\Omega;\H).$ Then, we have 
	\begin{align}\label{ener1}
	&\E\left[\sup_{t\in[0,T]}\|\u^n(t)\|_{\H}^2+4\mu \int_0^{T}\|\u^n(t)\|_{\V}^2\d t+4\beta\int_0^T\|\u^n(t)\|_{\widetilde{\L}^{r+1}}^{r+1}\d t\right]\nonumber\\&\leq \left(2\E\left[\|\u_0\|_{\H}^2\right]+14KT\right)e^{28KT}. 
\end{align}
\end{proposition}
\begin{proof}
	\noindent\textbf{Step (1):} Let us first define a sequence of stopping times $\tau_N^n$ by
	\begin{align}\label{stopm}
	\tau_N^n:=\inf_{t\geq 0}\left\{t:\|\u^n(t)\|_{\H}\geq N\right\},
	\end{align}
	for $N\in\mathbb{N}$. Applying the finite dimensional It\^{o} formula to the process
	$\|\u^n(\cdot)\|_{\H}^2$, we obtain 
	\begin{align}\label{4.9}
	\|\u^n(\t)\|_{\H}^2&=
	\|\u^n(0)\|_{\H}^2-2\int_0^{\t}\langle\mu \A\u^n(s)+\B^n(\u^n(s))+\beta\mathcal{C}^n(\u^n(s)),\u^n(s)\rangle\d s \nonumber\\&\quad+\int_0^{\t}\int_{\Z_n}\|\gamma^n(s,\u^n(s),z)\|_{\H}^2\pi(\d
	s,\d z)
	\nonumber\\&\quad+2\int_0^{\t}\int_{\Z_n}\left(\gamma^n(s-,\u^n(s-),z),\u^n(s-)\right)\widetilde{\pi}(\d
	s,\d z).
	\end{align}
Here we used $\langle\B^n(\u^n),\u^n\rangle=\langle\B(\u^n),\u^n\rangle=0$. 
	Note that$\|\u^n(0)\|_{\H}\leq \|\u_0\|_{\H}$ and the term
 $$\int_0^{\t}\int_{\Z_n}2\left(\gamma^n(s-,\u^n(s-),z),\u^n(s-)\right)\widetilde{\pi}(\d
	s,\d z)$$ 
	is a local  martingale with zero expectation. Moreover, we know that (\cite{UMMT1}) \begin{align}\label{311}\E\underbrace{\left[\int_0^{\t}\int_{\Z_n}\|\gamma^n(s,\u^n(s),z)\|_{\H}^2\pi(\d
		s,\d z)\right]}_{[\M]_{\t}}=\E\underbrace{\left[\int_0^{\t}\int_{\Z_n}\|\gamma^n(s,\u^n(s),z)\|_{\H}^2\lambda(\d z)\d s\right]}_{\langle\M\rangle_{\t}},\end{align} where $[\M]_t$ and $\langle\M\rangle_t$, respectively are the quadratic variation process and Meyer process of $\M_t=\int_0^t\int_{\Z}\gamma^n(s-,\u^n(s-),z)\widetilde{\pi}(\d z,\d s)$. Thus, taking expectation in \eqref{4.9}, we get
	\begin{align}\label{4.13}
	&\E\left[\|\u^n(\t)\|_{\H}^2+2\mu \int_0^{\t}\|\u^n(s)\|_{\V}^2\d s+2\beta\int_0^{\t}\|\u^n(s)\|_{\widetilde{\L}^{r+1}}^{r+1}\d s\right]\nonumber\\&\leq
\E\left[	\|\u^n(0)\|_{\H}^2\right]+\E\left[\int_0^{\t}\int_{\Z}\|\gamma^n(s,\u^n(s),z)\|_{\H}^2\lambda(\d z)\d
	s \right]\nonumber\\&\leq \E\left[	\|\u_0\|_{\H}^2\right]+K\E\left[\int_0^{\t}\left(1+\|\u^n(s))\|^2_{\H}\right]\d
	s \right],
	\end{align}
	where we used  the Hypothesis \ref{hyp}  (H.2) and (H.4). Applying Gronwall's inequality in (\ref{4.13}), we get 
	\begin{align}\label{2.7}
	\E\left[\|\u^n(\t)\|_{\H}^2\right]\leq
	\left(\E\left[\|\u_0\|_{\H}^2\right]+KT\right)e^{KT},
	\end{align}
	for all $t\in[0,T]$.  Note that for the indicator function $\chi$, 
	$$\E\left[{\chi}_{\{\tau_N^n<t\}}\right]=\mathbb{P}\Big\{\omega\in\Omega:\tau_N^n(\omega)<t\Big\},$$
	and using \eqref{stopm}, we obtain 
	\begin{align}\label{sq11}
	\E\left[\|\u^n(\t)\|_{\H}^2\right]&=
	\E\left[\|\u^n(\tau_N^n)\|_{\H}^2{\chi}_{\{\tau_N^n<t\}}\right]+\E\left[\|\u^n(t)\|_{\H}^2
	{\chi}_{\{\tau_N^n\geq t\}}\right]\nonumber\\&\geq
	\E\left[\|\u^n(\tau_N^n)\|_{\H}^2{\chi}_{\{\tau_N^n<t\}}\right]\geq
	N^2\mathbb{P}\Big\{\omega\in\Omega:\tau_N^n<t\Big\}.
	\end{align}
	Using the energy estimate (\ref{2.7}), we find 
	\begin{align}\label{sq12}
	\mathbb{P}\Big\{\omega\in\Omega:\tau_N^n<t\Big\}&\leq
	\frac{1}{N^2}\E\left[\|\u^n(\t)\|_{\H}^2\right]\leq
	\frac{1}{N^2}
	\left(\E\left[\|\u_0\|_{\H}^2\right]+KT\right)e^{KT}.
	\end{align}
	Hence, we have
	\begin{align}\label{sq13}
	\lim_{N\to\infty}\mathbb{P}\Big\{\omega\in\Omega:\tau_N^n<t\Big\}=0, \ \textrm{
		for all }\ t\in [0,T],
	\end{align}
	and $\t\to t$ as $N\to\infty$. 
	Taking limit $N\to\infty$ in
	(\ref{2.7}) and using the \emph{monotone convergence theorem}, we get 
	\begin{align}\label{4.16}
	\E\left[\|\u^n(t)\|_{\H}^2\right]\leq
	\left(\E\left[\|\u_0\|_{\H}^2\right]+KT\right)e^{KT},
	\end{align}
	for $0\leq t\leq T$. Substituting (\ref{4.16}) in (\ref{4.13}), we arrive at
	\begin{align}\label{4.16a}
	\E\left[\|\u^n(t)\|_{\H}^2+2\mu \int_0^{t}\|\u^n(s)\|_{\V}^2\d s+2\beta\int_0^t\|\u^n(s)\|_{\widetilde{\L}^{r+1}}^{r+1}\d s\right]\leq
	\left(\E\left[\|\u_0\|_{\H}^2\right]+KT\right)e^{2KT},
	\end{align}
	for all $t\in[0,T]$. 
	\vskip 0.2cm
	\noindent\textbf{Step (2):} Let us now prove \eqref{ener1}.  Taking supremum from $0$
	to $\T$ before taking expectation in (\ref{4.9}), we obtain
	\begin{align}\label{4.17}
	&\E\left[\sup_{t\in[0,\T]}\|\u^n(t)\|_{\H}^2+2\mu \int_0^{\T}\|\u^n(t)\|_{\V}^2\d t+2\beta\int_0^{\T}\|\u^n(t)\|_{\widetilde{\L}^{r+1}}^{r+1}\d t\right]\nonumber\\&\leq
	\E\left[\|\u_0\|_{\H}^2\right]
	+\E\left[\int_0^{\T}\int_{\Z}\|\gamma^n(t,\u^n(t),z)\|_{\H}^2\lambda(\d z)\d t\right]  \nonumber\\&\quad +2\E\left[\sup_{t\in[0,\T]}\left|\int_0^t\int_{\Z_n}\left(\gamma^n(s-,\u^n(s-),z),\u^n(s-)\right)\widetilde{\pi}(\d
		s,\d z)\right|\right].
	\end{align}
	Now we take the final term from the right hand side of the inequality \eqref{4.17} and use Burkholder-Davis-Gundy (see Theorem 1, \cite{BD} for the Burkholder-Davis-Gundy inequality and Theorem 1.1, \cite{DLB} for the best constant), H\"{o}lder  and Young's inequalities to deduce that
	\begin{align}\label{4.19}
&2\E\left[\sup_{t\in[0,\T]}\left|\int_0^t\int_{\Z_n}\left(\gamma^n(s-,\u^n(s-),z),\u^n(s-)\right)\widetilde{\pi}(\d
s,\d z)\right|\right]\nonumber\\&\leq 2\sqrt{3}\E\left[\int_0^{\T}\int_{\Z_n}\|\gamma^n(t,\u^n(t),z)\|_{\H}^2\|\u^n(t)\|_{\H}^2\lambda(\d z)\d t\right]^{1/2}\nonumber\\&\leq 2\sqrt{3}\E\left[\sup_{t\in[0,\T]}\|\u^n(t)\|_{\H}\left(\int_0^{\T}\int_{\Z_n}\|\gamma^n(t,\u^n(t),z)\|_{\H}^2\lambda(\d z)\d t\right)^{1/2}\right]\nonumber\\&\leq \frac{1}{2} \E\left[\sup_{t\in[0,\T]}\|\u^n(t)\|_{\H}^2\right]+6\E\left[\int_0^{\T}\int_{\Z_n}\|\gamma^n(\u^n(t),z)\|_{\H}^2\lambda(\d	z)\d t\right].
	\end{align}
	Substituting \eqref{4.19}  in (\ref{4.17}), we find
	\begin{align}\label{4.20}
	&\E\left[\sup_{t\in[0,\T]}\|\u^n(t)\|_{\H}^2+4\mu \int_0^{\T}\|\u^n(t)\|_{\V}^2\d t+4\beta\int_0^T\|\u^n(t)\|_{\widetilde{\L}^{r+1}}^{r+1}\d t\right]\nonumber\\&\leq 2\E\left[\|\u_0\|_{\H}^2\right]+14 \E\left[\int_0^{\T}\int_{\Z}\|\gamma^n(t,\u^n(t),z)\|_{\H}^2\lambda(\d z)\d t\right] \nonumber\\&\leq 2\E\left[\|\u_0\|_{\H}^2\right]+14K\int_0^{\T}\E\left[1+\|\u^n(t)\|_{\H}^2\right]\d t,
	\end{align}
	where  we used the Hypothesis \ref{hyp} (H.2). Applying Gronwall's inequality in (\ref{4.20}), we obtain 
	\begin{align}\label{4.21}
	\E\left[\sup_{t\in[0,\T]}\|\u^n(t)\|_{\H}^2\right]\leq\left(2\|\u_0\|_{\H}^2+14KT\right)e^{14KT}.
	\end{align}
	Passing $N\to\infty$, using the monotone convergence theorem and then substituting (\ref{4.21}) in (\ref{4.20}), we finally obtain (\ref{ener1}). 
\end{proof}

\begin{lemma}\label{lem3.6}
	For $r>3$ and any functions $$\u,\v\in\mathrm{L}^2(\Omega;\mathrm{L}^{\infty}(0,T;\H)\cap\mathrm{L}^2(0,T;\V))\cap\mathrm{L}^{r+1}(\Omega;\mathrm{L}^{r+1}(0,T;\widetilde{\L}^{r+1})),$$  we have
	\begin{align}\label{3.11y}
	&\int_0^Te^{-\eta t}\Big[\mu \langle \A(\u(t)-\v(t)),\u(t)-\v(t)\rangle +\langle \B(\u(t))-\B(\v(t)),\u(t)-\v(t)\rangle \nonumber\\&\quad+\beta\langle\mathcal{C}(\u(t))-\mathcal{C}(\v(t)),\u(t)-\v(t)\rangle \Big]\d t +\left(\eta+\frac{L}{2}\right)\int_0^Te^{-\eta t}\|\u(t)-\v(t)\|_{\H}^2\d t\nonumber\\&\geq \frac{1}{2}\int_0^Te^{-\eta t}\int_{\Z}\|\gamma(t,\u(t), z)-\gamma(t,\v(t), z)\|^2_{\H}\lambda(\d z)\d t,
	\end{align}
	where $\eta$ is defined in \eqref{215}. 
\end{lemma}
\begin{proof}
	Multiplying (\ref{fe}) with $e^{-\eta t}$, integrating over time $t\in(0,T)$, and using  the Hypothesis \ref{hyp} (H.3), we obtain (\ref{3.11y}). 
\end{proof}

\subsection{Existence and uniqueness of strong solution} Let us now show that the system (\ref{32}) has a unique global strong solution in the sense of Definition \ref{defn3.2}. We exploit the monotonicity property (see (\ref{3.11y})) and a stochastic generalization of the Minty-Browder technique to obtain such a result. This method is widely used to prove the existence of global  strong solutions to the stochastic partial differential equations.  The local monotonicity property of the linear and nonlinear operators and a stochastic generalization of the Minty-Browder technique has been used to obtain global solvability results of various mathematical physics models perturbed by Gaussian or L\'evy noise, see for instance  \cite{MJSS,SSSP,ICAM,UMMT,ZBEH,MoSS2,MTM}, etc and references therein.
\begin{theorem}\label{exis}
	Let $\u_0\in \mathrm{L}^2(\Omega;\H)$, for $r> 3$ be given.  Then there exists a \emph{pathwise unique  strong solution}	$\u(\cdot)$ to the system (\ref{32}) such that \begin{align*}\u&\in\mathrm{L}^2(\Omega;\mathrm{L}^{\infty}(0,T;\H)\cap\mathrm{L}^2(0,T;\V))\cap\mathrm{L}^{r+1}(\Omega;\mathrm{L}^{r+1}(0,T;\widetilde{\L}^{r+1})),\end{align*} with $\mathbb{P}$-a.s., c\`adl\`ag trajectories in $\H$ and $\u\in\D([0,T];\H)\cap\mathrm{L}^2(0,T;\V)\cap\mathrm{L}^{r+1}(0,T;\widetilde{\L}^{r+1})$, $\mathbb{P}$-a.s.
\end{theorem}
\begin{proof}
	We prove the global solvability of the system (\ref{32}) in the following steps.
	
	\vskip 0.2cm
	\noindent\textbf{Step (1):} \emph{Finite-dimensional (Galerkin) approximation of the system (\ref{32}):} Let us first consider the following Galerkin approximated It\^{o} stochastic differential equation satisfied by $\{\u^n(\cdot)\}$:
	\begin{equation}\label{4.37}
	\left\{
	\begin{aligned}
	\d\u^n(t)&=-\G(\u^n(t))\d
	t+\int_{\Z_n}\left(\gamma^n(t-,\u^n(t-),z),\v\right)\widetilde{\pi}(\d t,\d z),\\
	\u^n(0)&=\u_0^n,
	\end{aligned}
	\right.
	\end{equation}
	where
	$\G(\u^n(\cdot))=\mu \A\u^n(\cdot)+\B^n\left(\u^n(\cdot)\right)+\beta\mathcal{C}^n(\u^n(\cdot))$. Let us first apply finite dimensional  It\^o's formula to the process $e^{-2\eta t}\|\u^n(\cdot)\|_{\H}^2$ to get 
	\begin{align}\label{4.38}
	e^{-2\eta t}\|\u^n(t)\|_{\H}^2&=\|\u^n(0)\|_{\H}^2-\int_0^te^{-2\eta s}\langle 2\G(\u^n(s))+2\eta \u^n(s),\u^n(s)\rangle\d
	s \nonumber\\&\quad+\int_0^{\t}\int_{\Z_n}e^{-2\eta s}\|\gamma^n(s,\u^n(s),z)\|_{\H}^2\pi(\d
	s,\d z)
	\nonumber\\&\quad+2\int_0^{\t}\int_{\Z_n}e^{-2\eta s}\left(\gamma^n(s-,\u^n(s-),z),\u^n(s-)\right)\widetilde{\pi}(\d
	s,\d z),
		\end{align}
	for all $t\in[0,T]$. Note that the  final term from the right hand side of the equality (\ref{4.38}) is a martingale and the fifth term satisfies \eqref{311}. Taking expectation, we get 
	\begin{align}\label{4.39}
	\E\left[e^{-2\eta t}\|\u^n(t)\|_{\H}^2\right]&=\E\left[\|\u^n(0)\|_{\H}^2\right]-\E\left[\int_0^te^{-2\eta s}\langle 2\G(\u^n(s))+2\eta\u^n(s),\u^n(s)\rangle\d
	s\right]\nonumber\\&\quad+\E\left[\int_0^t
	e^{-2\eta s}\int_{\Z_n}\|\gamma^n(s,\u^n(s),z)\|_{\H}^2\lambda(\d z)\d s\right],
	\end{align}
	for all $t\in[0,T]$.
	
	\vskip 0.2cm
	\noindent\textbf{Step (2):} \emph{Weak convergence of the sequences $\u^n(\cdot)$, $\G(\u^n(\cdot))$,		and $\gamma^n(\cdot,\cdot,\cdot)$.} Our aim is extract subsequences from the uniformly bounded (independent of $n$) energy estimate \eqref{ener1} in Proposition \ref{prop1}.	We know that 	$\mathrm{L}^2\left(\Omega;\mathrm{L}^{\infty}(0,T;\H)\right)$ is the dual of 	$\mathrm{L}^{2}\left(\Omega;\mathrm{L}^1(0,T;\H)\right)$ and  the space $\mathrm{L}^{2}\left(\Omega;\mathrm{L}^1(0,T;\H)\right)$ is separable. Moreover,  the spaces $\mathrm{L}^2(\Omega;\mathrm{L}^2(0,T;\V))$ and $\mathrm{L}^{r+1}(\Omega;\mathrm{L}^{r+1}(0,T;\widetilde{\L}^{r+1}))$ are  reflexive ($\X''=\X$).  Thus, we are in a position to apply the Banach-Alaoglu theorem. From the energy estimate \eqref{ener1} given in Proposition \ref{prop1}, we know that the sequence $\{\u^n(\cdot)\}$ is bounded independent of $n$ in the spaces $\mathrm{L}^2\left(\Omega;\mathrm{L}^{\infty}(0,T;\H)\right)$, $\mathrm{L}^2(\Omega;\mathrm{L}^2(0,T;\V))$ and $\mathrm{L}^{r+1}(\Omega;\mathrm{L}^{r+1}(0,T;\widetilde{\L}^{r+1}))$. Then applying the	Banach-Alaoglu theorem, we can extract a subsequence	$\{\u^{n_k}\}$ of $\{\u^n\}$ such that	(for simplicity, we denote the index $n_k$ by $n$):
	\begin{equation}\label{4.40}
	\left\{
	\begin{aligned}
	\u^n&\xrightarrow{w^{*}} \u\textrm{ in
	}\mathrm{L}^2(\Omega;\mathrm{L}^{\infty}(0,T ;\H)),\\ 
	\u^n&\xrightarrow{w} \u\textrm{ in
	}\mathrm{L}^2(\Omega;\mathrm{L}^{2}(0,T ;\V)),\\  
\u^n&\xrightarrow{w} \u\textrm{ in
}\mathrm{L}^{r+1}(\Omega;\mathrm{L}^{r+1}(0,T ;\widetilde{\L}^{r+1})),\\
	\u^n(T)&\xrightarrow{w}\xi \in\mathrm{L}^2(\Omega;\H),\\
	\G(\u^n)&\xrightarrow{w} \G_0\textrm{ in
	}\mathrm{L}^2(\Omega;\mathrm{L}^2(0,T ;\V'))+\mathrm{L}^{\frac{r+1}{r}}(\Omega;\mathrm{L}^{\frac{r+1}{r}}(0,T;\widetilde{\L}^{\frac{r+1}{r}})).
	\end{aligned}
	\right.
	\end{equation}
Using H\"older's inequality, interpolation inequality (see \eqref{211})  and Proposition \ref{prop1}, we justify the final convergence  in (\ref{4.5})  in the following way: 
	\begin{align}\label{4.41}
	&\E\left[\left|\int_0^T\langle\G(\u^n(t)),\v(t)\rangle\d t\right]\right|\nonumber\\&\leq \mu\E\left[ \left|\int_0^T(\nabla\u^n(t),\nabla\v(t))\d t\right|\right]+\E\left[\left|\int_0^T\langle \B(\u^n(t),\v(t)),\u^n(t)\rangle\d t\right|\right]\nonumber\\&\quad+\beta\E\left[\left|\int_0^T\langle|\u^n(t)|^{r-1}\u^n(t),\v(t)\rangle\d t\right|\right]\nonumber\\&\leq\mu\E\left[ \int_0^T\|\nabla\u^n(t)\|_{\H}\|\nabla\v(t)\|_{\H}\d t\right]+\E\left[\int_0^T\|\u^n(t)\|_{\widetilde{\L}^{r+1}}\|\u^n(t)\|_{\widetilde{\L}^{\frac{2(r+1)}{r-1}}}\|\v(t)\|_{\V}\d t\right]\nonumber\\&\quad+\beta\E\left[\int_0^T\|\u^n(t)\|_{\widetilde{\L}^{r+1}}^{r}\|\v(t)\|_{\widetilde{\L}^{r+1}}\d t\right]\nonumber\\&\leq \mu \mathbb{E}\left[\left(\int_0^T\|\u^n(t)\|_{\V}^2\d t\right)^{1/2}\left(\int_0^T\|\v(t)\|_{\V}^2\d t\right)^{1/2}\right]\nonumber\\&\quad+\E\left[\left(\int_0^T\|\u^n(t)\|_{\widetilde{\L}^{r+1}}^{r+1}\d t\right)^{\frac{1}{r-1}}\left(\int_0^T\|\u^n(t)\|_{\H}^2\d t\right)^{\frac{r-3}{2(r-1)}}\left(\int_0^T\|\v(t)\|_{\V}^2\d t\right)^{1/2}\right]\nonumber\\&\quad+\beta\E\left[\left(\int_0^T\|\u^n(t)\|_{\widetilde{\L}^{r+1}}^{r+1}\d t\right)^{\frac{r}{r+1}}\left(\int_0^T\|\v(t)\|_{\widetilde{\L}^{r+1}}^{r+1}\d t\right)^{\frac{1}{r+1}}\right]\nonumber\\&\leq\mu \left\{\E\left(\int_0^T\|\u^n(t)\|_{\V}^2\d t\right)\right\}^{1/2}\left\{\E\left(\int_0^T\|\v(t)\|_{\V}^2\d t\right)\right\}^{1/2}\nonumber\\&\quad+T^{\frac{r-3}{2(r-1)}}\left\{\E\left(\int_0^T\|\u^n(t)\|_{\widetilde{\L}^{r+1}}^{r+1}\d t\right)\right\}^{\frac{1}{r-1}}\left\{\E\left(\sup_{t\in[0,T]}\|\u^n(t)\|_{\H}^2\right)\right\}^{\frac{r-3}{2(r-1)}}\left\{\E\left(\int_0^T\|\v(t)\|_{\V}^2\d t\right)\right\}^{\frac{1}{2}}\nonumber\\&\quad+\beta\left\{\E\left(\int_0^T\|\u^n(t)\|_{\widetilde{\L}^{r+1}}^{r+1}\d t\right)\right\}^{\frac{r}{r+1}}\left\{\E\left(\int_0^T\|\v(t)\|_{\widetilde{\L}^{r+1}}^{r+1}\d t\right)\right\}^{\frac{1}{r+1}}\nonumber\\&\leq  C(\E\left[\|\u_0\|_{\H}^2\right],\mu ,T,\beta,K)\left[\left\{\E\left(\int_0^T\|\v(t)\|_{\V}^2\d t\right)\right\}^{1/2}+\E\left\{\left(\int_0^T\|\v(t)\|_{\widetilde{\L}^{r+1}}^{r+1}\d t\right)\right\}^{\frac{1}{r+1}}\right],
	\end{align}
	for all $\v\in\mathrm{L}^2(\Omega;\mathrm{L}^2(0,T;\V))\cap\mathrm{L}^{r+1}(\Omega;\mathrm{L}^{r+1}(0,T;\widetilde{\L}^{r+1}))$. 
	Using the Hypothesis \ref{hyp} (H.2) and energy	estimates  in Proposition \ref{prop1},	we also have
	\begin{align}\label{4.42}
\E\left[\int_{\Z_n}\|\gamma^n(s,\u^n(s),z)\|_{\H}^2\lambda(\d z)\d	t\right]&\leq K\E\left[\int_0^T\left(1+\|\u^n(t)\|_{\H}^2\right)\d t\right]\nonumber\\&\leq	KT \left(1+\left(2\E\left[\|\u_0\|_{\H}^2\right]+24KT\right)e^{48KT}\right)<+\infty.
	\end{align}
	Note that the right hand side of the estimate \eqref{4.42} is independent of $n$ and thus we can extract a subsequences	 $\{\gamma^{n_k}(\cdot,\u^{n_k},\cdot)\}$ of $\{\gamma^n(\cdot,\u^n,\cdot)\}$ such that (relabeled as  $\{\gamma^n(\cdot,\u^n,\cdot)\}$)
	\begin{eqnarray}\label{4.43z}
\gamma^n(\cdot,\u^n,\cdot)&\xrightarrow{w} \gamma(\cdot,\cdot)\ \textrm{ in
}\ \mathrm{L}^2(\Omega;\mathrm{L}^2(0,T;\mathrm{L}^2_{\lambda}(\Z;\H))).
	\end{eqnarray}
	
		\vskip 0.2cm
	\noindent\textbf{Step (3):} \emph{It\^o stochastic differential satisfied by $\u(\cdot)$.} 	 Due to technical reasons, we extend the time	interval from $[0,T]$ to an open interval $(-\e, T+\e)$ with	$\e>0$, and set the terms in the equation (\ref{4.37}) equal to	zero outside the interval $[0,T]$. Let $\phi\in\mathrm{H}^1(-\mu,T+\mu)$ be such that $\phi(0)=1$. For $\v_m\in\V\cap\widetilde{\L}^{r+1},$ we define $\v_m(t)=\phi(t)\v_m$.	Let us apply finite dimensional It\^{o}'s formula to the process	$(\u^n(t),\v_m(t))$ to get 
	\begin{align}\label{sq63}
	(\u^n(T),\v_m(T ))&=(\u^n(0),\v_m)+\int_0^{T	}(\u^n(t),\dot{\v}_m(t))\d	t-\int_0^{T	}(\G(\u^n(t)),\v_m(t))\d t \nonumber\\&\quad+\int_0^T\int_{\Z_n}\left(\gamma^n(t-,\u^n(t-),z),\v_m(t)\right)\widetilde{\pi}(\d t,\d z),
	\end{align}
where $\dot{\v}_m(t)=\frac{\d\phi(t)}{\d t}\v_m$. One can take the term by term limit $n\to\infty$ in (\ref{sq63}) by	making use of the weak convergences given in (\ref{4.40}) and (\ref{4.43z}). 

	We consider the stochastic integral present in the final term	from the right hand side of the equality (\ref{sq63}) with $m$	fixed. Let $\mathscr{P}_{T}$ denotes the class of predictable	processes with values in $\mathrm{L}^2(\Omega;\mathrm{L}^2(0,T;\mathrm{L}^2_{\lambda}(\Z;\H)))$ (see (\ref{33}) for	definition and Chapter 3, \cite{VMBR}) associated with the inner	product	$$(\gamma,\xi)_{\mathscr{P}_{T}}=\E\left[\int_0^{t}\int_{\Z}(\gamma(s,z),\xi(s,z))\lambda(\d z)\d s\right],\ \textrm{ for		all }\ \gamma,\xi\in \mathscr{P}_{T}\ \text{ and }\ t\in[0,T].$$Moreover, we have 
	\begin{align}\label{3.33}
	&	\left|\E\left[\int_0^t\int_{\Z_n}(\gamma^n(s,\u^n(s),z),\v_m(s))\lambda(\d z)\d s-\int_0^t\int_{\Z}(\gamma(s,z),\v_m(s))\lambda(\d z)\d s\right]\right|\nonumber\\&\leq	\left|\E\left[\int_0^t\int_{\Z}(\gamma^n(s,\u^n(s),z)-\gamma(s,z),\v_m(s))\lambda(\d z)\d s\right]\right|\nonumber\\&\quad +	\left|\E\left[\int_0^t\int_{\Z\backslash\Z_n}(\gamma(s,z),\v_m(s))\lambda(\d z)\d s\right]\right|.
	\end{align}
	 The weak convergence of
	$\gamma^n(\cdot,\u^n,\cdot)\xrightarrow{w}\gamma(\cdot,\cdot)$ in
	$\mathrm{L}^2(\Omega;\mathrm{L}^2(0,T;\mathrm{L}^2_{\lambda}(\Z;\H)))$ (see \eqref{4.43z}) implies that
	$(\gamma^n(t,\u^n(t),z),\xi)_{\mathscr{P}_{T}}\to(\gamma(t,z),\xi)_{\mathscr{P}_{T}}$,
	for all $\xi\in\mathscr{P}_{T}$ and $t\in[0,T]$, as $n\to\infty$. 
	 In particular, for $\xi=\v_m(\cdot)$, we find that the first term in the right hand side of the inequality \eqref{3.33} converges to zero as $n\to\infty$. Using H\"older's inequality and Hypothesis \ref{hyp} (H.4), we estimate
	 \begin{align*}
	 &\left|\E\left[\int_0^t\int_{\Z\backslash\Z_n}(\gamma(s,z),\v_m(s))\lambda(\d z)\d s\right]\right|\nonumber\\&\leq\E\left[\int_0^t\int_{\Z\backslash\Z_n}\|\gamma(s,z)\|_{\H}\|\v_m(s)\|\lambda(\d z)\d s\right]\nonumber\\&\leq\sup_{t\in[0,T]}|\phi(t)|\|\v_m\|_{\H}\E\left[\int_0^t[\lambda(\Z\backslash\Z_n)]^{1/2}\left(\int_{\Z\backslash\Z_n}\|\gamma(s,z)\|_{\H}^2\lambda(\d z)\right)^{1/2}\d s\right]\nonumber\\&\leq\sup_{t\in[0,T]}|\phi(t)|\|\v_m\|_{\H}T^{1/2}\lambda(\Z\backslash\Z_n)]^{1/2}\E\left[\int_0^T\int_{\Z}\|\gamma(t,z)\|_{\H}^2\lambda(\d z)\d t\right]^{1/2}\to 0,
	 \end{align*}
	 as $n\to\infty$. Using the fact that the expectation of quadratic variation process and Meyer process as same, we get 
	\begin{align}\label{332}
&	\left|\E\left[\int_0^t\int_{\Z_n}(\gamma^n(s-,\u^n(s-),z),\v_m(s-)){\pi}(\d	s,\d	z)-\int_0^t\int_{\Z}(\gamma(s-,z),\v_m(s-)){\pi}(\d	s,\d z)\right]\right|\nonumber\\&=		\left|\E\left[\int_0^t\int_{\Z_n}(\gamma^n(s,\u^n(s),z),\v_m(s))\lambda(\d z)\d s-\int_0^t\int_{\Z}(\gamma(s,z),\v_m(s))\lambda(\d z)\d s\right]\right|\to 0, 
	\end{align}
as $n\to\infty$, where we used \eqref{3.33}. Let us now define the map
$\Gamma:\mathscr{P}_{T}\to\mathrm{L}^2(\Omega;\mathrm{L}^2(0,T))$ by
$$\Gamma(\gamma)=\int_0^t\int_{\Z}(\gamma(s-,\omega,z),\v_m(s-))\widetilde{\pi}(\d
s,\d z),$$ for all $t\in[0,T]$. It can be easily seen that the map
$\Gamma$ is linear and continuous. 	Thus, as $n\to\infty$, we have
	\begin{align*}\Gamma(\gamma^n(\cdot,\u^n(\cdot),\cdot))&=\int_0^t\int_{\Z}(\gamma^n(s-,\u^n(s-),z),\v_m(s-))\widetilde{\pi}(\d	s,\d	z)\nonumber\\&\to\int_0^t\int_{\Z}(\gamma(s-,z),\v_m(s-))\widetilde{\pi}(\d	s,\d z)\textrm{ as }n\to\infty,\end{align*} for all $t\in[0,T]$ and	for each fixed $m$. Using this convergence and calculation similar to \eqref{3.33}  and \eqref{332} yields 
		\begin{align*}\int_0^t\int_{\Z_n}(\gamma^n(s-,\u^n(s-),z),\v_m(s-))\widetilde{\pi}(\d	s,\d	z)\to\int_0^t\int_{\Z}(\gamma(s-,z),\v_m(s-))\widetilde{\pi}(\d	s,\d z)\textrm{ as }n\to\infty,\end{align*} for all $t\in[0,T]$ and	for each fixed $m$. 
	
	Passing to limits term wise in the equation (\ref{sq63}), we get
	\begin{align}\label{sq64}
	(\xi,\v_m)\phi(T)&=(\u_0,\v_m)+\int_0^{T
	}(\u(t),\dot{\v}_k)\d
	t-\int_0^{T
	}\phi(t)\langle\G_0(t),\v_m\rangle\d t \nonumber\\&\quad+\int_0^T\phi(t)\int_{\Z}\left(\gamma(t-,z),\v_m\right)\widetilde{\pi}(\d t,\d z).
	\end{align}
	Let us now choose a subsequence $\{\phi_k\}\in\mathrm{H}^1(-\mu,T+\mu)$	with $\phi_k(0)=1$, for $k\in\mathbb{N}$, such that	$\phi_k\to\displaystyle{\chi}_t$ and the time derivative of $\phi_k$	converges to $\delta_t$, where $\displaystyle{\chi}_t(s)=1$, for	$s\leq t$ and $0$ otherwise, and $\delta_t(s)=\delta(t-s)$ is the	Dirac $\delta$-distribution. Using $\phi_k$ in place of $\phi$ in	(\ref{sq64}) and then letting $k\to\infty$, we obtain
	\begin{align}\label{sq65}
	(\u(t),\v_m)&=(\u_0,\v_m)-\int_0^{t}\langle\G_0(s),\v_m\rangle\d s +\int_0^t\int_{\Z}\left(\gamma(s-,z),\v_m\right)\widetilde{\pi}(\d s,\d z),
	\end{align}
	for all $0<t< T$ with $(\u(T),\v_m)=(\xi,\v_m)$ and for any $\v_m\in\V\cap\widetilde{\L}^{r+1}$. It should be noted that $\mathcal{V}\subset\V\cap\widetilde{\L}^{r+1}\subset\H$ and $\mathcal{V}$ is dense in $\H$. Therefore, $\V\cap\widetilde{\L}^{r+1}$ is dense in $\H$ and the above equation holds for any $\v\in\V\cap\widetilde{\L}^{r+1}$.   Thus, we	have
	\begin{align}\label{sq66}
	(\u(t),\v)&=(\u_0,\v)-\int_0^t\langle\G_0(s),\v\rangle\d s +\int_0^t\int_{\Z}\left(\gamma(s-,z),\v\right)\widetilde{\pi}(\d s,\d z),\ \mathbb{P}\text{-a.s.,}
	\end{align}
for all $0<t<T$	with $(\u(T),\v)=(\xi,\v)$, for all $\v\in\V\cap\widetilde{\L}^{r+1}$. Hence, $\u(\cdot)$ satisfies the following stochastic differential: 
	\begin{equation}\label{4.44}
	\left\{
	\begin{aligned}
	\d\u(t)&=-\G_0(t)\d
	t+\int_{\Z}\gamma(t-,z)\widetilde{\pi}(\d t,\d z),\\
	\u(0)&=\u_0,
	\end{aligned}
	\right.
	\end{equation}
	for $\u_0\in\mathrm{L}^2(\Omega;\H)$. 
	
		\vskip 0.2cm
	\noindent\textbf{Step (4):} \emph{Energy equality satisfied by $\u(\cdot)$.} 
Let us now establish the energy equality (It\^o's formula) satisfied by $\u(\cdot)$, which is the crucial step in proving the Theorem. It should be noted that such an energy equality is not immediate due to the final convergence in \eqref{4.40} and we cannot apply the infinite dimensional It\^o formula available in the literature for semimartingales (see Theorem 1, \cite{GK1}, Theorem 6.1, \cite{Me}). We follow the approximations given in \cite{CLF} to obtain such an energy equality. In \cite{CLF}, the authors established an approximation of $\u(\cdot)$ in bounded domains such that the approximations are bounded and converge in both Sobolev and Lebesgue spaces simultaneously (one can see \cite{KWH} for such an approximation of $\L^p$-space valued functions using truncated Fourier expansions in periodic domains). We approximate $\u(t),$ for each $t\in(0,T)$ and $\mathbb{P}$-a.s., by using the finite-dimensional space spanned by the first $n$ eigenfunctions of the Stokes operator as (Theorem 4.3, \cite{CLF})
\begin{align}\label{3.32}\u_n(t):=\mathrm{P}_{1/n}\u(t)=\sum_{\lambda_j<n^2}e^{-\lambda_j/n}\langle\u(t),w_j\rangle w_j.\end{align}
For notational convenience, we use the approximations given in \eqref{3.32} as $\u_n(\cdot)$  and the Galerkin approximations  in steps (1)-(3) as $\u^n(\cdot)$. Note first that 
\begin{align}\label{3.36}
\|\u_n\|_{\H}^2=\|\mathrm{P}_{1/n}\u\|_{\H}^2=\sum_{\lambda_j<n^2}e^{-2\lambda_j/n}|\langle\u,w_j\rangle|^2\leq\sum_{j=1}^{\infty}|\langle\u,w_j\rangle|^2=\|\u\|_{\H}^2<+\infty,
\end{align}
for all $\u\in\H$. Moreover, we have 
\begin{align}\label{3.37}
\|(\mathrm{I}-\mathrm{P}_{1/n})\u\|_{\H}^2&=\|\u\|_{\H}^2-2\langle\u,\mathrm{P}_{1/n}\u\rangle+\|\mathrm{P}_{1/n}\u\|_{\H}^2\nonumber\\&=\sum_{j=1}^{\infty}|\langle\u,w_j\rangle|^2-2\sum_{\lambda_j<n^2}e^{-\lambda_j/n}|\langle\u,w_j\rangle|^2+\sum_{\lambda_j<n^2}e^{-2\lambda_j/n}|\langle\u,w_j\rangle|^2\nonumber\\&=\sum_{\lambda_j<n^2}(1-e^{-\lambda_j/n})^2|\langle\u,w_j\rangle|^2+\sum_{\lambda_j\geq n^2}|\langle\u,w_j\rangle|^2,
\end{align}
for all $\u\in\H$. It should be noted that the final term in the right hand side of the equality \eqref{3.37} tends zero as $n\to\infty$, since the series $\sum_{j=1}^{\infty}|\langle\u,w_j\rangle|^2$ is convergent. The first term on the right hand side of the equality can be made bounded above  by $$\sum_{j=1}^{\infty}(1-e^{-\lambda_j/n})^2|\langle\u,w_j\rangle|^2\leq 4\sum_{j=1}^{\infty}|\langle\u,w_j\rangle|^2=4\|\u\|_{\H}^2<+\infty.$$ Using the dominated convergence theorem, one can interchange the limit and sum, and hence we obtain 
$$\lim_{n\to\infty}\sum_{j=1}^{\infty}(1-e^{-\lambda_j/n})^2|\langle\u,w_j\rangle|^2=\sum_{j=1}^{\infty}\lim_{n\to\infty}(1-e^{-\lambda_j/n})^2|\langle\u,w_j\rangle|^2=0.$$
  Hence $\|(\mathrm{I}-\mathrm{P}_{1/n})\u\|_{\H}\to 0$ as $n\to\infty$, which implies  $\|\mathrm{I}-\mathrm{P}_{1/n}\|_{\mathcal{L}(\H)}\to 0$ as $n\to\infty$. Moreover, for $\u\in\V'$, we have 
  \begin{align}\label{3.34}
\|(\mathrm{I}-\mathrm{P}_{1/n})\u\|_{\V'}^2&=\|\A^{-1/2}(\mathrm{I}-\mathrm{P}_{1/n})\u\|_{\H}^2=\|(\mathrm{I}-\mathrm{P}_{1/n})\A^{-1/2}\u\|_{\H}^2\nonumber\\&\leq\|\mathrm{I}-\mathrm{P}_{1/n}\|_{\mathcal{L}(\H)}^2\|\u\|_{\V'}^2\to 0\ \text{ as }\ n\to\infty. 
  \end{align}
Let us now discuss about the properties of the approximation given in \eqref{3.32}. The authors in \cite{CLF}  showed that such an approximation  satisfies: 
\begin{enumerate}
	\item [(1)] $\u_n(t)\to\u(t)$ in $\H_0^1$ with $\|\u_n(t)\|_{\H^1}\leq C\|\u(t)\|_{\H^1}$, $\mathbb{P}$-a.s. and a.e. $t\in[0,T]$, 
	\item [(2)] $\u_n(t)\to\u(t)$ in $\L^{p}(\mathcal{O})$ with $\|\u_n(t)\|_{{\L}^{p}}\leq C\|\u(t)\|_{{\L}^{p}}$, for any $p\in(1,\infty)$, $\mathbb{P}$-a.s. and a.e. $t\in[0,T]$, 
	\item [(3)] $\u_n(t)$ is divergence free and zero on $\partial\mathcal{O}$, $\mathbb{P}$-a.s. and a.e. $t\in[0,T]$.
\end{enumerate}
In (1) and (2), $C$ is an absolute constant.   It should be noted that for $n\leq 4$, $\D(\A)\subset\H^2(\mathcal{O})\subset\L^p(\mathcal{O}),$ for all $p\in(1,\infty)$ (cf. \cite{CLF}). Since $w_j$'s are the eigenfunctions of the Stokes' operator $\A$, we get $w_j\in\D(\A)\subset\V$ and $w_j\in\D(\A)\subset\widetilde{\L}^{r+1}$.  Taking $\v=w_j$ in \eqref{sq66}, multiplying by $e^{-\lambda_j/n}w_j$ and then summing over all $j$ such that $\lambda_j<n^2$, we see that $\u_n(\cdot)$ satisfies the following It\^o stochastic differential: 
\begin{align}\label{333} 
\u_n(t)={\u_0}_n-\int_0^t{\mathrm{G}_0}_n(s)\d s+\int_0^t\int_{\Z}\gamma_n(s-,z)\wi\pi(\d s, \d z),
\end{align}
where ${\u_0}_n=\mathrm{P}_{1/n}\u_0$, ${\mathrm{G}_0}_n=\mathrm{P}_{1/n}{\mathrm{G}_0}$, and $\gamma_n=\mathrm{P}_{1/n}\gamma$. It is clear that the equation \eqref{333} has a unique solution $\u_n(\cdot)$ (see \cite{DaZ,VMBR}).  We apply It\^o's formula to the process $\|\u_n(\cdot)\|_{\H}^2$ to find 
\begin{align}\label{334}
\|\u_n(t)\|_{\H}^2&=\|{\u_0}_n\|_{\H}^2-2\int_0^t({\mathrm{G}_0}_n(s),\u_n(s))\d s+\int_0^{t}\int_{\Z}\|\gamma_n(s,z)\|_{\H}^2\pi(\d s,\d z) \nonumber\\&\quad+2\int_0^{t}\int_{\Z}\left(\gamma_n(s-,z),\u^n(s-)\right)\widetilde{\pi}(\d s,\d z),
\end{align}
for all $t\in(0,T)$. 

Using the convergence given in \eqref{3.37}, we get 
\begin{align}\label{3.40}
\E\left[\sup_{t\in[0,T]}\|\u(t)-\u_n(t)\|_{\H}^2\right]&\leq \|\mathrm{I}-\mathrm{P}_{1/n}\|_{\mathcal{L}(\H)}^2\E\left[\sup_{t\in[0,T]}\|\u(t)\|_{\H}^2\right]\to 0,
\end{align}
as $n\to\infty$, since $\u\in\mathrm{L}^2(\Omega;\mathrm{L}^{\infty}(0,T;\H))$. Thus, it is immediate that
\begin{align}\label{3.41}
\E\left[\|\u_n(t)\|_{\H}^2\right]\to \E\left[\|\u(t)\|_{\H}^2\right],\ \text{ as } \ n\to\infty,
\end{align}
for a. e. $t\in[0,T]$. A calculation similar to \eqref{3.40} yields 
\begin{align}\label{3.42}
\E\left[\|{\u_0}_n\|_{\H}^2\right]\to\E\left[\|{\u_0}\|_{\H}^2\right]\ \text{ as } \ n\to\infty. 
\end{align}
We also need the fact \begin{align}\label{335}\|\u_n-\u\|_{\mathrm{L}^{r+1}(\Omega;\mathrm{L}^{r+1}(0,T;\widetilde{\L}^{r+1}))}, \ \text{ as }\ n\to\infty,\end{align} which  follows from (2). Since $\u\in\mathrm{L}^{r+1}(\Omega;\mathrm{L}^{r+1}(0,T;\widetilde{\L}^{r+1}))$ and the fact that $\|\u^n(t,\omega)-\u(t,\omega)\|_{\widetilde{\L}^{r+1}}\to 0$, for a.e. $t\in[0,T]$ and $\mathbb{P}$-a.s., one can obtain the above convergence by an application of the dominated convergence theorem (with the dominating function $(1+C)\|\u(t,\omega)\|_{\widetilde{\L}^{r+1}}$). 

Since $\G_0\in\mathrm{L}^2(\Omega;\mathrm{L}^{2}(0,T;\V'))+\mathrm{L}^{\frac{r+1}{r}}(\Omega;\mathrm{L}^{\frac{r+1}{r}}(0,T;{\wi\L}^{\frac{r+1}{r}}))$, we can write down $\G_0$ as $\G_0=\G_0^1+\G_0^2$, where $\G_0^1\in\mathrm{L}^2(\Omega;\mathrm{L}^{2}(0,T;\V'))$ and $\G_0^2\in\mathrm{L}^{\frac{r+1}{r}}(\Omega;\mathrm{L}^{\frac{r+1}{r}}(0,T;{\wi\L}^{\frac{r+1}{r}}))$. Note that $1<\frac{r+1}{r}<\frac{4}{3}$. We use the approximation
\begin{align*}
{\G_{0}^1}_n(t):=\mathrm{P}_{1/n}\G_0^1(t)=\sum_{\lambda_j<n^2}e^{-\lambda_j/n}\langle\G_0^1(t),w_j\rangle w_j,
\end{align*}
and by using \eqref{3.34}, we get 
\begin{align}
\E\left[\int_0^T\|{\G_{0}^1}_n(t)-{\G_{0}^1}(t)\|_{\V'}^2\d t\right]\to 0,\ \text{ as } \ n\to\infty. 
\end{align}
Similarly, for $\G_0^2$, we use the approximation 
\begin{align*}
{\G_{0}^2}_n(t):=\mathrm{P}_{1/n}\G_0^2(t)=\sum_{\lambda_j<n^2}e^{-\lambda_j/n}\langle\G_0^2(t),w_j\rangle w_j,
\end{align*}
and by using (2), we get 
\begin{align}
\E\left[\int_0^T\|{\G_{0}^2}_n(t)-{\G_{0}^2}(t)\|_{\wi\L^{\frac{r+1}{r}}}^{{\frac{r+1}{r}}}\d t\right]\to 0,\ \text{ as } \ n\to\infty. 
\end{align}
By defining ${\G_0}_n={\G_{0}^1}_n+{\G_{0}^2}_n$, one can easily see that 
 \begin{align}\label{336}\|{\G_0}_n-\G_0\|_{\mathrm{L}^2(\Omega;\mathrm{L}^2(0,T;\V'))+\mathrm{L}^{\frac{r+1}{r}}(\Omega;\mathrm{L}^{\frac{r+1}{r}}(0,T;{\wi\L}^{\frac{r+1}{r}}))}\to 0, \ \text{ as }\ n\to\infty.\end{align}
Let us now consider 
\begin{align}\label{3.39}
&\E\left[\left|\int_0^t({\mathrm{G}_0}_n(s),\u_n(s))\d s-\int_0^t\langle{\mathrm{G}_0}(s),\u(s)\rangle\d s\right|\right]\nonumber\\&\leq \E\left[\left|\int_0^t\langle{\mathrm{G}_0}_n(s)-{\mathrm{G}_0}(s),\u_n(s)\rangle\d s\right|\right]+\E\left[\left|\int_0^t\langle{\mathrm{G}_0}(s),\u_n(s)-\u(s)\rangle\d s\right|\right]\nonumber\\&\leq \E\left[\left|\int_0^t\langle{\mathrm{G}_0^1}_n(s)-\mathrm{G}^1_0(s),\u_n(s)\rangle\d s\right|\right]+\E\left[\left|\int_0^t\langle{\mathrm{G}_0^2}_n(s)-\mathrm{G}^2_0(s),\u_n(s)\rangle\d s\right|\right]\nonumber\\&\quad+\E\left[\left|\int_0^t\langle{\mathrm{G}_0^1}(s),\u_n(s)-\u(s)\rangle\d s\right|\right]+\E\left[\left|\int_0^t\langle{\mathrm{G}_0^2}(s),\u_n(s)-\u(s)\rangle\d s\right|\right]\nonumber\\&\leq \E\left[\int_0^t\|{\mathrm{G}_0^1}_n(s)-{\mathrm{G}_0^1}(s)\|_{\V'}\|\u_n(s)\|_{\V}\d s\right]+\E\left[\int_0^t\|{\mathrm{G}_0^2}_n(s)-{\mathrm{G}_0^2}(s)\|_{\widetilde{\L}^{\frac{r+1}{r}}}\|\u_n(s)\|_{\wi\L^{r+1}}\d s\right]\nonumber\\&\quad+\E\left[\int_0^t\|\G_0^1(s)\|_{\V'}\|\u_n(s)-\u(s)\|_{\V}\d s\right]+\E\left[\int_0^t\|\G_0^2(s)\|_{\widetilde{\L}^{\frac{r+1}{r}}}\|\u_n(s)-\u(s)\|_{\wi\L^{r+1}}\d s\right]\nonumber\\&\leq C\left[\E\left(\int_0^t\|{\mathrm{G}_0^1}_n(s)-{\mathrm{G}_0^1}(s)\|_{\V'}^{2}\d s\right)\right]^{\frac{r}{r+1}}\left[\E\left(\int_0^t\|\u(s)\|_{\V}^{2}\d s\right)\right]^{\frac{1}{r+1}} \nonumber\\&\quad+C\left[\E\left(\int_0^t\|{\mathrm{G}_0^2}_n(s)-{\mathrm{G}_0^2}(s)\|_{\widetilde{\L}^{\frac{r+1}{r}}}^{\frac{r+1}{r}}\d s\right)\right]^{\frac{r}{r+1}}\left[\E\left(\int_0^t\|\u(s)\|_{\widetilde{\L}^{r+1}}^{r+1}\d s\right)\right]^{\frac{1}{r+1}}\nonumber\\&\quad+\left[\E\left(\int_0^t\|\G_0^1(s)\|_{\V'}^{2}\d s\right)\right]^{\frac{r}{r+1}}\left[\E\left(\int_0^t\|\u_n(s)-\u(s)\|_{\V}^{2}\d s\right)\right]^{\frac{1}{r+1}}\nonumber\\&\quad+\left[\E\left(\int_0^t\|\G_0^2(s)\|_{\widetilde{\L}^{\frac{r+1}{r}}}^{\frac{r+1}{r}}\d s\right)\right]^{\frac{r}{r+1}}\left[\E\left(\int_0^t\|\u_n(s)-\u(s)\|_{\widetilde{\L}^{r+1}}^{r+1}\d s\right)\right]^{\frac{1}{r+1}}\to 0,
\end{align}
as $ n\to\infty$, for all $t\in(0,T)$, where we used \eqref{335} and \eqref{336} (as one can show the above convergence by taking supremum over time $0\leq t\leq T$). Next, we establish the convergence of the stochastic integral. Using \eqref{3.36} and \eqref{3.37}, we get 
\begin{align}\label{349}
\|\gamma_n(t,z)\|_{\H}\leq \|\gamma(t,z)\|_{\H} \ \text{ and } \ \|\gamma_n(t,z)-\gamma(t,z)\|_{\H}\to 0, \ \text{ as } \ n\to\infty, 
\end{align}
$\mathbb{P}$-a.s., for a.e. $t\in[0,T]$ and a.a. $z\in\Z$. 
Let us now consider
\begin{align}\label{350}
&\E\left[\left|\int_0^{t}\int_{\Z}\|\gamma_n(s,z)\|_{\H}^2\pi(\d s,\d z)-\int_0^{t}\int_{\Z}\|\gamma(s,z)\|_{\H}^2\pi(\d s,\d z)\right|\right]\nonumber\\& \leq\E\left[\int_0^{t}\int_{\Z}\left|\|\gamma_n(s,z)\|_{\H}^2-\|\gamma(s,z)\|_{\H}^2\right|\pi(\d s,\d z)\right]\nonumber\\& =\E\left[\int_0^{t}\int_{\Z}\left|(\|\gamma_n(s,z)\|_{\H}-\|\gamma(s,z)\|_{\H})(\|\gamma_n(s,z)\|_{\H}+\|\gamma(s,z)\|_{\H})\right|\lambda(\d z)\d s\right]\nonumber\\&\leq\left\{\E\left[\int_0^t\int_{\Z}\|\gamma_n(s,z)-\gamma(s,z)\|_{\H}^2\lambda(\d z)\d s\right]\right\}^{1/2}\nonumber\\&\quad\times\left\{\E\left[\int_0^t\int_{\Z}\left(\|\gamma_n(s,z)\|_{\H}+\|\gamma(s,z)\|_{\H}\right)^2\lambda(\d z)\d s\right]\right\}^{1/2}\nonumber\\&\leq 2\left\{\E\left[\int_0^t\int_{\Z}\|\gamma_n(s,z)-\gamma(s,z)\|_{\H}^2\lambda(\d z)\d s\right]\right\}^{1/2}\left\{\E\left[\int_0^t\int_{\Z}\|\gamma(s,z)\|_{\H}^2\lambda(\d z)\d s\right]\right\}^{1/2}\nonumber\\& \to 0, \ \text{ as }\ n\to\infty,
\end{align}
by using an application of Lebesgue's dominated convergence theorem. Finally, we consider 
\begin{align}\label{351}
&\E\left[\left|\int_0^{t}\int_{\Z}\left(\gamma_n(s-,z),\u^n(s-)\right)\widetilde{\pi}(\d s,\d z)-\int_0^{t}\int_{\Z}\left(\gamma(s-,z),\u(s-)\right)\widetilde{\pi}(\d s,\d z)\right|\right] \nonumber\\&\leq \E\left[\left|\int_0^{t}\int_{\Z}\left(\gamma_n(s-,z)-\gamma(s-,z),\u^n(s-)\right)\widetilde{\pi}(\d s,\d z)\right|\right]\nonumber\\&\quad+\E\left[\left|\int_0^{t}\int_{\Z}\left(\gamma(s-,z),\u_n(s-)-\u(s-)\right)\widetilde{\pi}(\d s,\d z)\right|\right].
\end{align}
Applying Burkholder-Davis-Gundy and H\"older's inequalities, we find 
\begin{align}\label{352}
&\E\left[\sup_{t\in[0,T]}\left|\int_0^{t}\int_{\Z}\left(\gamma_n(s-,z)-\gamma(s-,z),\u^n(s-)\right)\widetilde{\pi}(\d s,\d z)\right|\right]\nonumber\\&\leq\sqrt{3}\E\left[\int_0^T\int_{\Z}\|\gamma_n(s,z)-\gamma(s,z)\|_{\H}^2\|\u^n(s)\|_{\H}^2\lambda(\d z)\d s\right]^{1/2}\nonumber\\&\leq\sqrt{3}\E\left[\sup_{t\in[0,T]}\|\u_n(s)\|_{\H}\left(\int_0^T\int_{\Z}\|\gamma_n(s,z)-\gamma(s,z)\|_{\H}^2\lambda(\d z)\d s\right)^{1/2}\right]\nonumber\\&\leq\sqrt{3}\left\{\E\left[\sup_{t\in[0,T]}\|\u(s)\|_{\H}^2\right]\right\}^{1/2}\left\{\E\left[\int_0^T\int_{\Z}\|\gamma_n(s,z)-\gamma(s,z)\|_{\H}^2\lambda(\d z)\d s\right]\right\}^{1/2}\nonumber\\&\to 0\ \text{ as } \ n\to\infty,
\end{align}
using \eqref{3.36}, \eqref{349} and Lebesgue's dominated convergence theorem. Once again an application of the Burkholder-Davis-Gundy inequality yields 
\begin{align}\label{353}
&\E\left[\sup_{t\in[0,T]}\left|\int_0^{t}\int_{\Z}\left(\gamma(s-,z),\u_n(s-)-\u(s-)\right)\widetilde{\pi}(\d s,\d z)\right|\right]\nonumber\\&\leq\sqrt{3}\E\left[\int_0^T\int_{\Z}\|\gamma(s,z)\|_{\H}^2\|\u_n(s)-\u(s)\|_{\H}^2\lambda(\d z)\d s\right]^{1/2}\nonumber\\&\leq\sqrt{3}\E\left[\sup_{t\in[0,T]}\|\u_n(s)-\u(s)\|_{\H}\left(\int_0^T\int_{\Z}\|\gamma(s,z)\|_{\H}^2\lambda(\d z)\d s\right)^{1/2}\right]\nonumber\\&\leq\sqrt{3}\left\{\E\left[\sup_{t\in[0,T]}\|\u_n(s)-\u(s)\|_{\H}^2\right]\right\}^{1/2}\left\{\E\left[\int_0^T\int_{\Z}\|\gamma(s,z)\|_{\H}^2\lambda(\d z)\d s\right]\right\}^{1/2}\nonumber\\&\to 0\ \text{ as } \ n\to\infty,
\end{align}
using \eqref{3.40}. Combining \eqref{352} and \eqref{353}, we obtain that the right hand side of \eqref{351} tends to zero as $n\to\infty$. 
Using the convergences given in  \eqref{3.41}, \eqref{3.42}, \eqref{3.39}, \eqref{350} and \eqref{351}, along a subsequence one can pass to limit in \eqref{334} to get the energy equality: 
\begin{align}\label{345}
\|\u(t)\|_{\H}^2&=\|{\u_0}\|_{\H}^2-2\int_0^t\langle{\mathrm{G}_0}(s),\u(s)\rangle\d s+\int_0^{t}\int_{\Z}\|\gamma(s,z)\|_{\H}^2\pi(\d s,\d z) \nonumber\\&\quad+2\int_0^{t}\int_{\Z}\left(\gamma(s-,z),\u(s-)\right)\widetilde{\pi}(\d s,\d z),
\end{align}
for a. e. $t\in(0,T)$, $\mathbb{P}$-a.s.

Let $\eta(t)$ be an even, positive, smooth function with compact support contained in the interval $(-1, 1)$, such that $\int_{-\infty}^{\infty}\eta(s)\d s=1$. Let us denote by $\eta^h$, a family of mollifiers related to the function $\eta$ as $$\eta^h(s):=h^{-1}\eta(s/h), \ \text{ for } \ h>0.$$ In particular, we get $\int_0^h\eta^h(s)\d s=\frac{1}{2}$. For any function $\v\in\mathrm{L}^p(0, T; \X)$, $\mathbb{P}$-a.s., where $\X$ is a Banach space, for $p\in[1,\infty)$, we define its mollification in time  $\v^h(\cdot)$ as 
\begin{align*}
\v^h(s):=(\v*\eta^h)(s)=\int_0^T\v(\tau)\eta^h(s-\tau)\d\tau, \ \text{ for }\ h\in(0,T). 
\end{align*}
From Lemma 2.5, \cite{GGP}, we know that this mollification has the following properties.  For any $\v\in\mathrm{L}^p(0, T; \X)$, $\v^h\in\C^k([0,T);\X),$ $\mathbb{P}$-a.s. for all $k\geq 0$ and  \begin{align}\label{354}\lim_{h\to 0}\|\v^h-\v\|_{\mathrm{L}^p(0,T;\X)}=0, \ \mathbb{P}\text{-a.s}.\end{align} Moreover, $\|\v^h(t)-\v(t)\|_{\X}\to 0$, as $h\to 0$, for a.e. $t\in[0,T]$, $\mathbb{P}$-a.s. For some time $t_1>0$, we set 
\begin{align*}
\u^h(t)=\int_0^{t_1}\eta^h(t-s)\u(s)\d s=:(\eta^h*\u)(t),
\end{align*}
with the parameter $h$ satisfying  $0<h<T-t_1$ and $h<t_1,$ where $\eta_h$ is the even mollifier given above. Note that $\u^h(\cdot)$ satisfies the following It\^o stochastic differential: 
\begin{align}
\u^h(t)=\u^h(0)+\int_0^t(\dot{\eta}^h*\u)(s)\d s.
\end{align}
Applying It\^o's product formula to the process $(\u^h(\cdot),\u(\cdot))$, we obtain 
\begin{align}\label{356}
(\u^h(t),\u(t))&=(\u(0),\u^h(0))-\int_0^t\langle\u^h(s),\G_0(s)\rangle\d s+\int_0^t\int_{\Z}(\u^h(s-),\gamma(s-,z))\wi\pi(\d s,\d z)\nonumber\\&\quad+\int_0^t(\u(s),(\dot{\eta}^h*\u)(s))\d s+[\u^h,\u]_t,
\end{align}
where $[\u^h,\u]_t$ is the quadratic variation between the processes $\u^h(\cdot)$ and $\u(\cdot)$. Using stochastic Fubini's theorem (\cite[Lemma A.1.1]{VMBR}), we find 
\begin{align*}
[\u^h,\u]_t&=\left[\int_0^{t_1}\eta^h(t-s)\left(\int_0^s\int_{\Z}\gamma(\tau-,z)\widetilde{\pi}(\d \tau,\d z)\right)\d s,\int_0^t\int_{\Z}\gamma(\tau-,z)\widetilde{\pi}(\d\tau,\d z)\right]_t\nonumber\\&=\left[\int_0^{t_1}\int_{\Z}\left(\int_{\tau}^{t_1}\eta^h(t-s)\d s\right)\gamma(\tau-,z)\widetilde{\pi}(\d \tau,\d z),\int_0^t\int_{\Z}\gamma(\tau-,z)\widetilde{\pi}(\d\tau,\d z)\right]_t, 
\end{align*}
and hence 
\begin{align}
&[\u^h,\u]_{t_1}=\int_0^{t_1}\int_{\Z}\left(\int_{\tau}^{t_1}\eta^h(t_1-s)\d s\right)\|\gamma(\tau,z)\|_{\H}^2{\pi}(\d \tau,\d z).
\end{align}
Since the function $\eta^h$ is even in $(-h, h)$, we obtain $\dot{\eta}^h(r)=-\dot{\eta}^h(-r)$. Changing the order of integration, we get  (see \cite{KWH}) 
\begin{align}
\int_0^{t_1}(\u(s),(\dot{\eta}^h*\u)(s))\d s&=\int_0^{t_1}\int_0^{t_1}\dot{\eta}^h(s-\tau)(\u(s),\u(\tau))\d s\d \tau\nonumber\\&=-\int_0^{t_1}\int_0^{t_1}\dot{\eta}^h(\tau-s)(\u(s),\u(\tau))\d s\d \tau\nonumber\\&= -\int_0^{t_1}\int_0^{t_1}\dot{\eta}^h(\tau-s)(\u(s),\u(\tau))\d\tau\d s \nonumber\\&= -\int_0^{t_1}\int_0^{t_1}\dot{\eta}^h(s-\tau)(\u(\tau),\u(s))\d s\d \tau=0. 
\end{align}
Thus, from \eqref{356}, it is immediate that 
\begin{align}\label{359}
(\u(t_1),\u^h(t_1))&=(\u(0),\u^h(0))-\int_0^{t_1}\langle\u^h(s),\G_0(s)\rangle\d s+\int_0^{t_1}\int_{\Z}(\u^h(s-),\gamma(s-,z))\wi\pi(\d s,\d z)\nonumber\\&\quad+\int_0^{t_1}\int_{\Z}\left(\int_{\tau}^{t_1}\eta^h(t_1-s)\d s\right)\|\gamma(\tau,z)\|_{\H}^2{\pi}(\d \tau,\d z).
\end{align}
Next, we let $h\to 0$ in \eqref{3.39}, by considering the points  in $(0,T)$, where the jump occurs separately.  Let us first assume that $t_1$ is not a point in $(0,T)$, where the jump occurs. 
 For $\G_0=\G_0^1+\G_0^2$, where $\G_0^1\in\mathrm{L}^2(\Omega;\mathrm{L}^{2}(0,T;\V'))$ and $\G_0^2\in\mathrm{L}^{\frac{r+1}{r}}(\Omega;\mathrm{L}^{\frac{r+1}{r}}(0,T;{\wi\L}^{\frac{r+1}{r}}))$,  we consider 
\begin{align}\label{3p63}
&\E\left[ \left|\int_0^{t_1}\langle\G_0(s),\u^h(s)\rangle\d s-\int_0^{t_1}\langle\G_0(s),\u(s)\rangle\d s\right|\right]\nonumber\\&\leq\E\left[\int_0^{t_1}|\langle\G_0^1(s),\u^h(s)-\u(s)\rangle|\d s\right]+\E\left[\int_0^{t_1}|\langle\G_0^2(s),\u^h(s)-\u(s)\rangle|\d s\right]\nonumber\\&\leq \E\left[\int_0^{t_1}\|\G_0^1(s)\|_{\V'}\|\u^h(s)-\u(s)\|_{\V}\d s\right]+\E\left[\int_0^{t_1}\|\G_0^2(s)\|_{\wi\L^{\frac{r+1}{r}}}\|\u^h(s)-\u(s)\|_{\L^{r+1}}\d s\right]\nonumber\\&\leq\left\{\E\left[\int_0^{t_1}\|\G_0^1(s)\|_{\V'}^2\d s\right]\right\}^{1/2}\left\{\E\left[\int_0^{t_1}\|\u^h(s)-\u(s)\|_{\V}^2\d s\right]\right\}^{1/2}\nonumber\\&\quad+ \left\{\E\left[\int_0^{t_1}\|\G_0^2(s)\|_{\wi\L^{\frac{r+1}{r}}}^{\frac{r+1}{r}}\d s\right]\right\}^{\frac{r}{r+1}}\left\{\E\left[\int_0^{t_1}\|\u^h(s)-\u(s)\|_{\wi\L^{r+1}}^{r+1}\d s\right]\right\}^{\frac{1}{r+1}}\to 0\ \text{ as }\ h\to 0. 
\end{align}
Thus, along a subsequence, we obtain 
\begin{align}\label{360}
\lim_{h\to 0} \int_0^{t_1}\langle\G_0(s),\u^h(s)\rangle\d s= \int_0^{t_1}\langle\G_0(s),\u(s)\rangle\d s, \ \mathbb{P}\text{-a.s.}
\end{align}
Using Burkholder-Davis-Gundy and H\"older's inequalities, we have 
\begin{align}\label{3.65}
&\E\left[\left|\int_0^{t_1}\int_{\Z}(\u^h(s-),\gamma(s-,z))\wi\pi(\d s,\d z)-\int_0^{t_1}\int_{\Z}(\u(s-),\gamma(s-,z))\wi\pi(\d s,\d z)\right|\right]\nonumber\\&\leq\E\left[\sup_{t_1\in[0,T]}\left|\int_0^{t_1}\int_{\Z}(\u^h(s-)-\u(s-),\gamma(s-,z))\wi\pi(\d s,\d z)\right|\right]\nonumber\\&\leq\sqrt{3}\E\left[\int_0^T\int_{\Z}\|\u^h(s)-\u(s)\|_{\H}^2\|\gamma(s,z)\|_{\H}^2\lambda(\d z)\d s\right]^{1/2}\nonumber\\&\leq\sqrt{3}\E\left[\sup_{s\in[0,T]}\|\u^h(s)-\u(s)\|_{\H}\left(\int_0^T\int_{\Z}\|\gamma(s,z)\|_{\H}^2\lambda(\d z)\d s\right)^{1/2}\right]\nonumber\\&\leq\sqrt{3}\left\{\E\left[\sup_{s\in[0,T]}\|\u^h(s)-\u(s)\|_{\H}^2\right]\right\}^{1/2}\left\{\E\left[\int_0^T\int_{\Z}\|\gamma(s,z)\|_{\H}^2\lambda(\d z)\d s\right]\right\}^{1/2}\to 0,
\end{align}
as $h\to 0$. Thus, along subsequence, we get 
\begin{align}\label{3.66}
\lim_{h\to 0}\int_0^{t_1}\int_{\Z}(\u^h(s-),\gamma(s-,z))\wi\pi(\d s,\d z)=\int_0^{t_1}\int_{\Z}(\u(s-),\gamma(s-,z))\wi\pi(\d s,\d z), \ \mathbb{P}\text{-a.s.}
\end{align}
Finally, using the fact that $\int_0^h\eta^h(s)\d s=\frac{1}{2}$, we estimate 
\begin{align}\label{3.64}
&\E\left[\int_0^{t_1}\int_{\Z}\left(\int_{\tau}^{t_1}\eta^h(t_1-s)\d s\right)\|\gamma(\tau,z)\|_{\H}^2{\pi}(\d \tau,\d z) \right]\nonumber\\&=\E\left[\int_0^{t_1}\eta^h(t_1-s)\int_{0}^{s}\int_{\Z}\|\gamma(\tau,z)\|_{\H}^2{\pi}(\d \tau,\d z)\d s \right]\nonumber\\&=\E\left[\int_0^{t_1}\eta^h(s)\int_{0}^{t_1-s}\int_{\Z}\|\gamma(\tau,z)\|_{\H}^2{\pi}(\d \tau,\d z)\d s \right]\nonumber\\&= \E\left[\int_0^{h}\eta^h(s)\left(\int_{0}^{t_1}\int_{\Z}\|\gamma(\tau,z)\|_{\H}^2{\pi}(\d \tau,\d z)-\int_{t_1-s}^{t_1}\int_{\Z}\|\gamma(\tau,z)\|_{\H}^2{\pi}(\d \tau,\d z)\right)\d s \right]\nonumber\\&=\frac{1}{2}\E\left[\int_{0}^{t_1}\int_{\Z}\|\gamma(\tau,z)\|_{\H}^2{\pi}(\d \tau,\d z)\right]- \E\left[\int_0^{h}\eta^h(s)\int_{t_1-s}^{t_1}\int_{\Z}\|\gamma(\tau,z)\|_{\H}^2{\pi}(\d \tau,\d z)\d s \right]\nonumber\\&\to \frac{1}{2}\E\left[\int_{0}^{t_1}\int_{\Z}\|\gamma(\tau,z)\|_{\H}^2{\pi}(\d \tau,\d z)\right], \ \text{ as } \ h\to 0,
\end{align}
since jumps are not occurring at $t_1$. Thus, along a subsequence, we further have 
\begin{align}\label{3p64}
\lim_{h\to 0}\int_0^{t_1}\int_{\Z}\left(\int_{\tau}^{t_1}\eta^h(t_1-s)\d s\right)\|\gamma(\tau,z)\|_{\H}^2{\pi}(\d \tau,\d z) =\frac{1}{2}\int_0^{t_1}\int_{\Z}\|\gamma(\tau,z)\|_{\H}^2{\pi}(\d \tau,\d z), \ \mathbb{P}\text{-a.s.}
\end{align}
Using the convergences \eqref{360}-\eqref{3p64} in \eqref{359}, along a subsequence, we get  
\begin{align}
&\int_0^{t_1}\langle\u(s),\G_0(s)\rangle\d s-\int_0^{t_1}\int_{\Z}(\u(s-),\gamma(s-,z))\wi\pi(\d s,\d z)-\frac{1}{2}\int_0^{t_1}\int_{\Z}\|\gamma(s,z)\|_{\H}^2{\pi}(\d s,\d z)\nonumber\\&=-\lim_{h\to 0}\langle\u(t_1),\u^h(t_1)\rangle +\lim_{h\to 0}\langle\u(0),\u^h(0)\rangle,\ \mathbb{P}\text{-a.s.}
\end{align}
Using the $\mathrm{L}^2$-weak continuity (form right) of $\u(\cdot)$ around zero and the fact that $\int_0^h\eta^h(s)\d s=\frac{1}{2}$, we find 
\begin{align}\label{3.70}
(\u(0),\u^h(0))&=\int_0^{t_1}\eta^h(-s)(\u(0),\u(s))\d s=\int_0^{t_1}\eta^h(s)(\u(0),\u(0)+\u(s)-\u(0))\d s\nonumber\\&=\frac{1}{2}\|\u(0)\|_{\H}^2+\int_0^h\eta^h(s)(\u(0),\u(s)-\u(0))\d s \to\frac{1}{2}\|\u(0)\|_{\H}^2\ \text{ as } \ h\to 0,
\end{align}
$\mathbb{P}\text{-a.s.}$ 
Since $t_1$ is not a point, where the jump occurs, using the fact that $\u(\cdot)$ is $\mathrm{L}^2$-weakly continuous in time, we get 
\begin{align}\label{3.71}
(\u(t_1),\u^h(t_1))&=\int_0^{t_1}\eta^h(s)(\u(t_1),\u(t_1-s))\d s\nonumber\\&=\frac{1}{2}\|\u(t_1)\|_{\H}^2+\int_0^h\eta^h(s)(\u(t_1),\u(t_1-s)-\u(t_1))\d s\to \frac{1}{2}\|\u(t_1)\|_{\H}^2,
\end{align}
as $h\to 0$, $\mathbb{P}$-a.s. 
Combining the above convergences, we finally obtain the energy equality 
\begin{align}\label{3.72}
\|\u(t_1)\|_{\H}^2&=\|\u(0)\|_{\H}^2-2\int_0^{t_1}\langle\G_0(s),\u(s)\rangle\d s+2\int_0^{t_1}\int_{\Z}(\u(s-),\gamma(s-,z)\wi\pi(\d s,\d z)\nonumber\\&\quad+\int_0^{t_1}\int_{\Z}\|\gamma(s,z)\|_{\H}^2{\pi}(\d s,\d z),
\end{align}
for all $t_1\in(0,T)$, where the jump does not occur. 

Let us now take $t_1\in(0,T)$ as a point where a jump occurs. Let $\wi{t}_1$ be the point in $(0,T)$, where the previous jump occurs (take $\wi t_1=0,$ if the first jump occurs at $t_1$).  Let $\u(t_1-)$ denotes the left limit of $\u(\cdot)$ at the point $t_1$. From \eqref{359}, we have 
\begin{align}\label{373}
(\u(t_1),\u^h(t_1))&= (\u(0),\u^h(0))-\int_0^{t_1}\langle\u^h(s),\G_0(s)\rangle\d s+\int_0^{\wi t_1}\int_{\Z}(\u^h(s-),\gamma(s-,z))\wi\pi(\d s,\d z)\nonumber\\&\quad+(\u^h(t_1-),\u(t_1)-\u(t_1-))-\int_{\wi t_1}^{t_1}\int_{\Z}(\u^h(s),\gamma(s,z))\lambda(\d z)\d s\nonumber\\&\quad+\int_0^{t_1}\eta^h(t_1-s)\int_{0}^{s}\int_{\Z}\|\gamma(\tau,z)\|_{\H}^2{\pi}(\d \tau,\d z)\d s,
\end{align}
where we have used the fact that $\u(t_1)-\u(t_1-)=\gamma(t_1,\u(t_1)-\u(t_1-))\chi_{\u(t_1)-\u(t_1-)\in\Z}$ (see \cite[Chapter 4]{Ap}). Note that the convergences  \eqref{360} and \eqref{3.70} hold true in this case also. Once again using he fact that $\int_0^h\eta^h(s)\d s=\frac{1}{2} $ and the $\mathrm{L}^2$-weak continuity at $t_1-$, we find 
\begin{align}\label{374}
&\left|\int_0^h\eta^h(s)(\u(t_1),\u(t_1)-\u(t_1-s))\d s-\frac{1}{2}(\u(t_1),\u(t_1)-\u(t_1-))\right|\nonumber\\&=\left|\int_0^h\eta^h(s)\left[(\u(t_1),\u(t_1-)-\u(t_1-s))\right]\d s\right| \ \text{ as } \ h\to 0,\  \mathbb{P}\text{-a.s.}
\end{align}
Thus, the convergence given in \eqref{3.71} becomes 
\begin{align}\label{3.75}
(\u(t_1),\u^h(t_1))\to \frac{1}{2}\|\u(t_1)\|_{\H}^2-\frac{1}{2}(\u(t_1),\u(t_1)-\u(t_1-)), \ \text{ as } \ h\to 0,\  \mathbb{P}\text{-a.s.}
\end{align}
The convergence given in \eqref{3.65} implies 
\begin{align}
\lim_{h\to 0}\int_0^{\wi t_1}\int_{\Z}(\u^h(s-),\gamma(s-,z))\wi\pi(\d s,\d z)=\int_0^{\wi t_1}\int_{\Z}(\u(s-),\gamma(s-,z))\wi\pi(\d s,\d z), \ \mathbb{P}\text{-a.s.}
\end{align}
Let us now discuss about the convergence of $(\u^h(t_1-),\u(t_1)-\u(t_1-))$. A calculation similar to \eqref{374} gives
\begin{align}
&(\u^h(t_1-),\u(t_1)-\u(t_1-))\nonumber\\&=\int_0^{t_1}\eta^h(s)(\u((t_1-)-s),\u(t_1)-\u(t_1-))\d s\nonumber\\&= \frac{1}{2}(\u(t_1),\u(t_1)-\u(t_1-))+\int_0^{h}\eta^h(s)(\u((t_1-)-s)-\u(t_1),\u(t_1)-\u(t_1-))\d s\nonumber\\&\to \frac{1}{2}(\u(t_1),\u(t_1)-\u(t_1-))-\frac{1}{2}\|\u(t_1)-\u(t_1-)\|_{\H}^2, \ \text{ as } \ h\to 0,\  \mathbb{P}\text{-a.s.}
\end{align}
As there are no jumps in $(\wi t_1,t_1)$, we consider 
\begin{align}
&\E\left[\left|\int_{\wi t_1}^{t_1}\int_{\Z}(\u^h(s),\gamma(s,z))\lambda(\d z)\d s-\int_{\wi t_1}^{t_1}\int_{\Z}(\u(s),\gamma(s,z))\lambda(\d z)\d s\right|\right]\nonumber\\&=\E\left[\left|\int_{\wi t_1+}^{t_1-}\int_{\Z}(\u^h(s)-\u(s),\gamma(s,z))\lambda(\d z)\d s\right|\right]\nonumber\\&=\E\left[\left|-\int_{\wi t_1+}^{t_1-}\int_{\Z}(\u^h(s)-\u(s),\gamma(s,z))\wi\pi(\d s,\d z)\right|\right]\nonumber\\&\leq\sqrt{3}\E\left[\int_{0}^{T}\int_{\Z}\|\u^h(s)-\u(s)\|_{\H}^2\|\gamma(s,z)\|_{\H}^2\lambda(\d z)\d s\right]^{1/2}\nonumber\\&\leq\sqrt{3}\left\{\E\left[\sup_{s\in[0,T]}\|\u^h(s)-\u(s)\|_{\H}^2\right]\right\}^{1/2}\left\{\E\left[\int_0^T\int_{\Z}\|\gamma(s,z)\|_{\H}^2\lambda(\d z)\d s\right]\right\}^{1/2}\to 0,
\end{align}
as $h\to 0$, where we used Burkholder-Davis-Gundy inequality. Thus, along a subsequence, we have the following convergence: 
\begin{align}
\lim_{h\to 0}\int_{\wi t_1}^{t_1}\int_{\Z}(\u^h(s),\gamma(s,z))\lambda(\d z)\d s=\int_{\wi t_1}^{t_1}\int_{\Z}(\u(s),\gamma(s,z))\lambda(\d z)\d s, \ \mathbb{P}\text{-a.s.}
\end{align}
Since a jump occurs at the point $t_1$ and the jumps are isolated, a calculation similar to \eqref{3.64} yields 
\begin{align}
&\E\left[\int_0^{t_1}\eta^h(t_1-s)\int_{0}^{s}\int_{\Z}\|\gamma(\tau,z)\|_{\H}^2{\pi}(\d \tau,\d z)\d s \right]\nonumber\\&=\frac{1}{2}\E\left[\int_{0}^{t_1}\int_{\Z}\|\gamma(\tau,z)\|_{\H}^2{\pi}(\d \tau,\d z)\right]- \E\left[\int_0^{h}\eta^h(s)\int_{t_1-s}^{t_1}\int_{\Z}\|\gamma(\tau,z)\|_{\H}^2{\pi}(\d \tau,\d z)\d s \right]\nonumber\\&\to \frac{1}{2}\E\left[\int_{0}^{t_1}\int_{\Z}\|\gamma(\tau,z)\|_{\H}^2{\pi}(\d \tau,\d z)\right]-\frac{1}{2}\E\left[\|\u(t_1)-\u(t_1-)\|_{\H}^2\right], \ \text{ as } \ h\to 0,
\end{align}
Thus, along a subsequence, we have 
\begin{align}\label{3.79}
&\int_0^{t_1}\eta^h(t_1-s)\int_{0}^{s}\int_{\Z}\|\gamma(\tau,z)\|_{\H}^2{\pi}(\d \tau,\d z)\d s\nonumber\\&\to  \frac{1}{2}\int_{0}^{t_1}\int_{\Z}\|\gamma(\tau,z)\|_{\H}^2{\pi}(\d \tau,\d z)-\frac{1}{2}\|\u(t_1)-\u(t_1-)\|_{\H}^2, \ \text{ as } \ h\to 0,\ \mathbb{P}\text{-a.s.}
\end{align}
Combining the convergences \eqref{3.75}-\eqref{3.79}, substituting it in \eqref{373} and then taking limit along a subsequence as $h\to 0$, we find 
\begin{align}
\frac{1}{2}\|\u(t_1)\|_{\H}^2&=\frac{1}{2}\|\u(0)\|_{\H}^2-\int_0^{t_1}\langle\G_0(s),\u(s)\rangle\d s+\frac{1}{2}\int_{0}^{t_1}\int_{\Z}\|\gamma(\tau,z)\|_{\H}^2{\pi}(\d \tau,\d z) \nonumber\\&\quad+\int_0^{\wi t_1}\int_{\Z}(\u(s-),\gamma(s-,z)\wi\pi(\d s,\d z)+(\u(t_1-),\u(t_1)-\u(t_1-))\nonumber\\&\quad-\int_{\wi t_1}^{t_1}\int_{\Z}(\u(s),\gamma(s,z))\lambda(\d z)\d s,\nonumber\\&= \frac{1}{2}\|\u(0)\|_{\H}^2-\int_0^{t_1}\langle\G_0(s),\u(s)\rangle\d s+\int_0^{ t_1}\int_{\Z}(\u(s-),\gamma(s-,z)\wi\pi(\d s,\d z)  \nonumber\\&\quad+\frac{1}{2}\int_{0}^{t_1}\int_{\Z}\|\gamma(\tau,z)\|_{\H}^2{\pi}(\d \tau,\d z),\ \mathbb{P}\text{-a.s.},
\end{align}
and the It\^o formula \eqref{3.72} holds true for all $t_1\in(0,T)$.

Taking expectation and noting the fact that the final term in the right hand side of the equality \eqref{345} is a martingale, we find 
\begin{align}
\E\left[\|\u(t)\|_{\H}^2\right]&=\E\left[\|{\u_0}\|_{\H}^2\right]-2\E\left[\int_0^t\langle{\mathrm{G}_0}(s),\u(s)\rangle\d s\right]+\E\left[\int_0^{t}\int_{\Z}\|\gamma(s,z)\|_{\H}^2\lambda(\d z)\d s\right].
\end{align}
Thus an application of It\^o's formula to the process $e^{-2\eta t}\|\u(\cdot)\|_{\H}^2$ yields 
	\begin{align}\label{4.45}
	\E\left[e^{-2\eta t}\|\u(t)\|_{\H}^2\right]&=\E\left[\|\u_0\|_{\H}^2\right]-\E\left[\int_0^te^{-2\eta s}\langle 2\G_0(s)+2\eta \u(s),\u(s)\rangle\d
	s\right]\nonumber\\&\quad+\E\left[\int_0^t
	e^{-2\eta s}\int_{\Z}\|\gamma(s,z)\|_{\H}^2\lambda(\d z)\d s\right],
	\end{align}
	for all $t\in[0,T]$. Finally, we note that the initial value
	$\u^n(0)$ converges to $\u_0$ strongly in $\mathrm{L}^2(\Omega;\H)$, that is,
	\begin{align}\label{4.46}
	\lim_{n\to\infty}\E\left[\|\u^n(0)-\u_0\|_{\H}^2\right]=0.
	\end{align}

	\noindent\textbf{Step (5):} \emph{Minty-Browder technique and global strong solution.}  Now, we are ready to prove the strong solution to the system \eqref{32}. It is now left to show that $$\G(\u(\cdot))=\G_0(\cdot)\ \text{ and }\ 	\gamma(\cdot,\u(\cdot),\cdot)=\gamma(\cdot,\cdot).$$ In order to achieve this aim, we make use of the Lemma \ref{lem3.6}.  For	$\v\in\mathrm{L}^2(\Omega;\mathrm{L}^{\infty}(0,T;\H_m)),$ with $m<n$, using the local monotonicity result (see (\ref{3.11y})),	we get 
	\begin{align}\label{4.48}
	&\E\bigg[\int_0^{T}e^{-2\eta t}(2\langle\G(\v(t))-\G(\u^n(t)),\v(t)-\u^n(t)\rangle
	+2\eta \left(\v(t)-\u^n(t),\v(t)-\u^n(t)\right))\d
	t\bigg]\nonumber\\&\geq \E\left[\int_0^{T}e^{-2\eta t}\int_{\Z_n}\|\gamma^n(t,
	\v(t),z) - \gamma^n(t,\u^n(t),z)\|^2_{\H}\lambda(\d z)\d
	t\right].
	\end{align}
Rearranging the terms in \eqref{4.48} and then using energy	equality (\ref{4.39}), we obtain 
	\begin{align}\label{4.49}
	&\E\left[\int_0^{T}e^{-2\eta t}\langle 2\G(\v(t))+2\eta \v(t),\v(t)-\u^n(t)\rangle\d
	t\right]\nonumber\\&\quad-\E\left[\int_0^{T}e^{-2\eta t}\int_{\Z_n}\|\gamma^n(t,
	\v(t),z)\|^2_{\H}
\lambda(\d z)	\d
	t\right]\nonumber\\&\quad+2\E\left[\int_0^{T}e^{-2\eta t}\int_{\Z_n}\left(\gamma^n(t,
	\v(t),z),	\gamma^n(t,\u^n(t),z)\right)\lambda(\d z)\d
	t\right]\nonumber\\&\geq
	\E\left[\int_0^{T}e^{-2\eta t}\langle 2\G(\u^n(t))+2\eta \u^n(t),\v(t)\rangle\d
	t\right]\nonumber\\&\quad-\E\left[\int_0^{T}e^{-2\eta t}\langle 2\G(\u^n(t))+2\eta \u^n(t),\u^n(t)\rangle\d
	t\right]\nonumber\\&\quad+\E\left[\int_0^{T}e^{-2\eta t}\int_{\Z_n}\|
	\gamma^n(t,\u^n(t),z)\|^2_{\H}\lambda(\d z)\d
	t\right]\nonumber\\&=\E\left[\int_0^{T}e^{-2\eta t}\langle 2\G(\u^n(t))+2\eta \u^n(t),\v(t)\rangle\d
	t\right]
	+\E\left[e^{-2\eta T}\|\u^n(T)\|_{\H}^2-\|\u^n(0)\|_{\H}^2\right].
	\end{align}
	Let us now discuss the convergence of the terms involving noise coefficient. Note that \begin{align}\label{eqn49z}
	&\E\Bigg[\int_0^{T}e^{-2\eta t}\int_{\Z_n}(2\left(\gamma^n(t, \v(t),z),	\gamma^n(t,\u^n(t),z)\right)-\|\gamma^n(t,	\v(t),z)\|^2_{\H})\lambda(\d z)\d	t\Bigg]\nonumber\\& = \E\left[\int_0^{T}e^{-2\eta t}\int_{\Z_n}2\left(\gamma(t,	\v(t),z),	\gamma^n(t,\u^n(t),z)\right)\lambda(\d z)\d	t\right]\nonumber\\&\quad+	\E\left[\int_0^{T}e^{-2\eta t}\int_{\Z_n}2\left(\gamma^n(t, \v(t),z)-\gamma(t,	\v(t),z),	\gamma^n(t,\u^n(t),z)\right)\lambda(\d z)\d	t\right]\nonumber\\&\quad	-\E\left[\int_0^{T}e^{-2\eta t}\int_{\Z_n}\|\gamma^n(t,	\v(t),z)\|^2_{\H}\lambda(\d z)\d t\right]\nonumber\\&\leq	\E\left[\int_0^{T}e^{-2\eta t}\int_{\Z_n}2\left(\gamma(t, \v(t),z),	\gamma^n(t,\u^n(t),z)\right)\lambda(\d z)\d	t\right]\nonumber\\&\quad+2C	\left(\E\left[\int_0^{T}e^{-4\eta t}\int_{\Z}\left\|\gamma^n(t,	\v(t),z)-\gamma(t, \v(t),z)\right\|^2_{\H}\lambda(\d z)\d	t\right]\right)^{1/2}\nonumber\\&\quad	-\E\left[\int_0^{T}e^{-2\eta t}\int_{\Z_n}\|\gamma^n(t,	\v(t),z)\|^2_{\H}\lambda(\d z)\d t\right],
	\end{align}
	where $C=	{\left(\E\left[\int_0^{T}e^{-4\eta t}\int_{\Z}\left\|\gamma^n(t,	\u^n(t),z)\right\|^2\lambda(\d z)\d	t\right]\right)^{1/2}}$. Then, applying the weak convergence of $\{\gamma^n(\cdot,\u^n(\cdot),\cdot):n\in\mathbb{N}\}$  given in 	(\ref{4.43z})	to the first term and using the Lebesgue dominated convergence theorem   (an argument similar to \eqref{3.33}) to	the second and final terms on the right hand side of the inequality	(\ref{eqn49z}), we deduce that 
	\begin{align}\label{4.50}
	&\E\Bigg[\int_0^{T}e^{-2\eta t}\int_{\Z_n}(2\left(\gamma^n(t, \v(t),z),
	\gamma^n(t,\u^n(t),z)\right)-\|\gamma^n(t,
	\v(t),z)\|^2_{\H})\lambda(\d z)\d
	t\Bigg]\nonumber\\& \to
	\E\left[\int_0^{T}e^{-2\eta t}\int_{\Z}(2(\gamma(t, \v(t),z),
	\gamma(t,z))-\|\gamma(t,	\v(t),z)\|^2_{\H})\d t\right],
	\end{align}
	as $n\to\infty$. Taking liminf on both sides of (\ref{4.49}),
	and using (\ref{4.50}), we obtain
	\begin{align}\label{4.51}
	&\E\left[\int_0^{T}e^{-2\eta t}\langle 2\G(\v(t))+2\eta \v(t),\v(t)-\u(t)\rangle\d
	t\right]\nonumber\\&\quad-\E\left[\int_0^{T}e^{-2\eta t}\int_{\Z}\|\gamma(t,
	\v(t),z)\|^2_{\H} \lambda(\d z)\d
	t\right] +2\E\left[\int_0^{T}e^{-2\eta t}\int_{\Z}\left(\gamma(t,
	\v(t),z), \gamma(t,z)\right)\lambda(\d z)\d
	t\right]\nonumber\\&\geq\E\left[\int_0^{T}e^{-2\eta t}\langle 2\G_0(t)+2\eta \u(t),\v(t)\rangle\d
	t\right]
	+\liminf_{n\to\infty}\E\left[e^{-2\eta T}\|\u^n(T)\|_{\H}^2-\|\u^n(0)\|_{\H}^2\right].
	\end{align}
	Making use of the lower semicontinuity property of the $\H$-norm and the strong convergence given in
	(\ref{4.46}), the second term on the right hand side of the
	inequality (\ref{4.51}) satisfies:
	\begin{align}\label{4.52}
	&\liminf_{n\to\infty}\E\left[e^{-2\eta T}\|\u^n(T)\|_{\H}^2-\|\u^n(0)\|_{\H}^2\right]\geq
	\E\left[e^{-2\eta T}\|\u(T)\|^2_{\H}-\|\u_0\|^2_{\H}\right].
	\end{align}
	We use of the energy equality (\ref{4.45}) and (\ref{4.52}) in	(\ref{4.51}) to obtain 
	\begin{align}\label{4.53}
	&\E\left[\int_0^{T}e^{-2\eta t}\langle 2\G(\v(t))+2\eta \v(t),\v(t)-\u(t)\rangle\d
	t\right]\nonumber\\&\geq\E\left[\int_0^{T}e^{-2\eta t}\int_{\Z}\|\gamma(t,
	\v(t),z)\|^2_{\H}\lambda(\d z) \d
	t\right]-2\E\left[\int_0^{T}e^{-2\eta t}\int_{\Z}\left(\gamma(t,
	\v(t),z), \gamma(t,z)\right)\lambda(\d z)\d
	t\right]\nonumber\\&\quad+\E\left[\int_0^{T}e^{-2\eta t}\int_{\Z}\|\gamma(t,z)\|^2_{\H}\lambda(\d z)\d t\right]
	+\E\left[\int_0^{T}e^{-2\eta t}\langle2\G_0(t)+\eta \u(t),\v(t)-\u(t)\rangle\d
	t\right].
	\end{align}
	Rearranging the terms in (\ref{4.53}), we obtain
	\begin{align}\label{4.54}
	&\E\left[\int_0^{T}e^{-2\eta t}\langle 2\G(\v(t))-2\G_0(t)+2\eta (\v(t)-\u(t)),\v(t)-\u(t)\rangle\d
	t\right]\nonumber\\&\geq
	\E\Bigg[\int_0^{T}e^{-2\eta t}\int_{\Z}\|\gamma(t,
	\v(t),z)-\gamma(t,z)\|^2_{\H}\lambda(\d z)
	\d t\Bigg]\geq 0.
	\end{align}
	Note that the estimate (\ref{4.54}) holds true for any	$\v\in\mathrm{L}^2(\Omega;\mathrm{L}^{\infty}(0,T;\H_m))$ and for any
	$m\in\mathbb{N}$, since the estimate is independent of
	$m$ and $n$. Using a density argument, the
	inequality (\ref{4.54}) remains true for any
\begin{align*}\v&\in\mathrm{L}^2(\Omega;\mathrm{L}^{\infty}(0,T;\H)\cap\mathrm{L}^2(0,T;\V))\cap\mathrm{L}^{r+1}(\Omega;\mathrm{L}^{r+1}(0,T;\widetilde{\L}^{r+1}))=:\mathcal{G}.\end{align*} 
	Indeed, for any $\v\in\mathcal{G}$,
	there exists a strongly convergent subsequence
	$\v_m\in\mathcal{G}$, which satisfies the
	inequality (\ref{4.54}).
	Taking $\v(\cdot)=\u(\cdot)$ in (\ref{4.54}) immediately gives
	$\gamma(\cdot,\v(\cdot),\cdot)=\gamma(\cdot,\cdot)$. Next, we take
	$\v(\cdot)=\u(\cdot)+\lambda\w(\cdot)$, $\lambda>0$, where
	$\w\in\mathcal{G},$ and substitute for
	$\v$ in (\ref{4.54}) to find
	\begin{align}\label{4.55}
	\E\left[\int_0^{T}e^{-2\eta t}\langle \G(\u(t)+\lambda\w(t))-\G_0(t)+\eta \lambda\w(t),\lambda\w(t)\rangle\d
	t\right]\geq 0.
	\end{align}
	Dividing the above inequality by $\lambda$,  using the hemicontinuity property of
	$\G(\cdot)$ (see Lemma \ref{lem2.8}), and passing $\lambda\to 0$, we obtain
	\begin{align}\label{4.56}
	\E\left[\int_0^{T}e^{-2\eta t}\langle\G(\u(t))-\G_0(t),\w(t)\rangle\d
	t\right]\geq 0,
	\end{align}
	since the final term in (\ref{4.55}) tends to $0$ as $\lambda\to0$.
	Thus from (\ref{4.56}), we get  $\G(\u(t))=\G_0(t)$ and hence $\u(\cdot)$ is a strong solution of the
	system (\ref{32}) and
	$\u\in\mathcal{G}$. From \eqref{345}, it is immediate that $\u(\cdot)$ satisfy the following energy equality (It\^o's formula): 
	\begin{align}\label{3.63}
&	\|\u(t)\|_{\H}^2+2\mu \int_0^t\|\u(s)\|_{\V}^2\d s+2\beta\int_0^t\|\u(s)\|_{\widetilde{\L}^{r+1}}^{r+1}\d s\nonumber\\&=\|{\u_0}\|_{\H}^2+\int_0^t\|\gamma(s,\u(s),z)\|_{\H}^2\pi(\d s,\d z)+2\int_0^t\int_{\Z}(\gamma(s-,\u(s-),z),\u(s-))\wi\pi(\d s,\d z),
	\end{align}
for all $t\in(0,T)$, $\mathbb{P}$-a.s. 	Moreover, the following energy estimate holds true: 
		\begin{align}
	&\E\left[\sup_{t\in[0,T]}\|\u(t)\|_{\H}^2+4\mu \int_0^{T}\|\u(t)\|_{\V}^2\d t+4\beta\int_0^T\|\u(t)\|_{\widetilde{\L}^{r+1}}^{r+1}\d t\right]\nonumber\\&\leq \left(2\E\left[\|\u_0\|_{\H}^2\right]+14KT\right)e^{28KT}. \label{eng1}
	\end{align}
	Furthermore, since $\u(\cdot)$ satisfies the energy equality \eqref{3.63}, one can show  that the $\mathscr{F}_t$-adapted paths of $\u(\cdot)$ are c\`adl\`ag with trajectories in  $\D([0,T];\H)$, $\mathbb{P}$-a.s. (see \cite{GK1,Me,GK2}, etc).

	\vskip 0.2cm 
	\noindent\textbf{Step (6):} \emph{Uniqueness.} Finally, we show that the strong solution established in step (5) is unique. Let $\u_1(\cdot)$ and $\u_2(\cdot)$ be two strong solutions of the system (\ref{32}). For $N>0$, let us define 
	\begin{align*}
	{\tau_N^1}&=\inf_{0\leq t\leq T}\Big\{t:\|\u_1(t)\|_{\H}\geq N\Big\},\ \tau_N^2=\inf_{0\leq t\leq T}\Big\{t:\|\u_2(t)\|_{\H}\geq N\Big\}\text{ and }\tau_N:=\tau_N^1\wedge\tau_N^2.
	\end{align*}
	Using the energy estimate \eqref{eng1}, it can  be shown in a similar way as in step (1), Proposition \ref{prop1} that $\tau_N\to T$ as $N\to\infty$, $\mathbb{P}$-a.s. Let us define $\w(\cdot):=\u_1(\cdot)-\u_2(\cdot)$ and $\widetilde{\gamma}(\cdot,\w(\cdot),\cdot):=\gamma(\cdot,\u_1(\cdot),\cdot)-\gamma(\cdot,\u_2(\cdot),\cdot)$. Then, $\w(\cdot)$ satisfies the following system:
	\begin{equation}
	\left\{
	\begin{aligned}
	\d\w(t)&=-\left[\mu \A\w(t)+\B(\u_1(t))-\B(\u_2(t))+\beta(\mathcal{C}(\u_1(t))-\mathcal{C}(\u_2(t)))\right]\d t\\&\quad+\int_{\Z}\widetilde{\gamma}(t,\w(t),z)\wi\pi(\d t,\d z),\\
	\w(0)&=\w_0.
	\end{aligned}
	\right.
	\end{equation}
Then, in a similar way as in \eqref{3.63}, one can show that $\w(\cdot)$ satisfies the following energy equality: 
	\begin{align}\label{4.59}
	&\|\w(\s)\|_{\H}^2+2\mu \int_0^{\s}\|\w(s)\|_{\V}^2\d s\nonumber\\&=\|\w(0)\|_{\H}^2 -2\int_0^{\s}\langle\B(\u_1(s))-\B(\u_2(s)),\w(s)\rangle\d s\nonumber\\&\quad-2\int_0^{\s}\langle\mathcal{C}(\u_1(s))-\mathcal{C}(\u_2(s)),\u_1(s)-\u_2(s)\rangle\d s+\int_0^{\s}\int_{\Z}\|\wi\gamma(s,\w(s),z)\|_{\H}^2\pi(\d s,\d z)\nonumber\\&\quad+2\int_0^{\s}\int_{\Z}(\widetilde{\gamma}(s-,\w(s-),z),\w(s-))\wi\pi(\d s,\d z).
	\end{align}
	From (\ref{2.30}), we obtain 
	\begin{align*}
	|\langle \B(\u_1)-\B(\u_2),\w\rangle|\leq\frac{\mu }{2}\|\w\|_{\V}^2+\frac{\beta}{2}\||\u_2|^{\frac{r-1}{2}}\w\|_{\H}^2+\eta\|\w\|_{\H}^2.
	\end{align*}
	and from \eqref{2.27}, we get 
	\begin{align*}
\beta	\langle\mathcal{C}(\u_1)-\mathcal{C}(\u_2),\w\rangle\geq \frac{\beta}{2}\||\u_2|^{\frac{r-1}{2}}\w\|_{\H}^2.
	\end{align*}
	Thus, using the above two estimates in   (\ref{4.59}), we infer that 
	\begin{align}\label{4.60}
	&\|\w(\s)\|_{\H}^2+\mu \int_0^{\s}\|\w(s)\|_{\V}^2\d s\nonumber\\&\leq\|\w(0)\|_{\H}^2 +2\eta\int_0^{\s}\|\w(s)\|_{\H}^2\d s +\int_0^{\s}\int_{\Z}\|\wi\gamma(s,\w(s),z)\|_{\H}^2\pi(\d s,\d z)\nonumber\\&\quad+2\int_0^{\s}\int_{\Z}(\widetilde{\gamma}(s-,\w(s-),z),\w(s-))\wi\pi(\d s,\d z).
	\end{align}
	It should be noted that the final term in the right hand side of the inequality (\ref{4.60}) is a local martingale. Taking expectation in (\ref{4.60}), and  then using the Hypothesis \ref{hyp} (H.2), we obtain  
	\begin{align}\label{4.62}
	&\E\left[\|\w(\s)\|_{\H}^2+\mu \int_0^{\s}\|\w(s)\|_{\V}^2\d s\right]\nonumber\\& \leq \E\left[\|\w(0)\|_{\H}^2\right]+ 2\eta\E\left[\int_0^{\s}\|\w(s)\|_{\H}^2\d s\right]+\E\left[\int_0^{\s}\int_{\Z}\|\wi\gamma(s,\w(s),z)\|_{\H}^2\lambda(\d z)\d s\right]\nonumber\\&\leq  \E\left[\|\w(0)\|_{\H}^2\right]+(L+2\eta)\E\left[\int_0^{\s}\|\w(s)\|_{\H}^2\d s\right].
	\end{align}
	Applying  Gronwall's inequality in (\ref{4.62}),  we arrive at 
	\begin{align}\label{4.63}
	&\E\left[\|\w(\s)\|_{\H}^2\right]\leq \E\left[\|\w_0\|_{\H}^2\right]e^{(L+2\eta)T}.
	\end{align}
	Thus the initial data  $\u_1(0)=\u_2(0)=\u_0$ leads to $\w(\s)=0$, $\mathbb{P}$-a.s. But using the fact that $\tau_N\to T$, $\mathbb{P}$-a.s., implies $\w(t)=0$ and hence $\u_1(t) = \u_2(t)$, $\mathbb{P}$-a.s., for all $t \in[0, T ]$,  and hence the uniqueness follows.
\end{proof}

\begin{remark}
	Recently authors in \cite{GK2} (Theorem 1) obtained It\^o's formula (semimartingales) for processes taking values in intersection of finitely many Banach spaces. But it seems to the author that this result may not applicable in our context for establishing the energy equality \eqref{3.63}, as our operator $\B(\cdot):\V\cap\widetilde{\L}^{r+1}\to\V'+\widetilde{\L}^{\frac{r+1}{r}}$ (see \eqref{212}) and one can show the local integrability in the sum of Banach spaces only. 
\end{remark}
\begin{theorem}\label{exis1}
For $r=3$ and $2\beta\mu \geq 1$,	let $\u_0\in \mathrm{L}^2(\Omega;\H)$ be given.  Then there exists a \emph{pathwise unique strong solution}
	$\u(\cdot)$ to the system (\ref{32}) such that \begin{align*}\u&\in\mathrm{L}^2(\Omega;\mathrm{L}^{\infty}(0,T;\H)\cap\mathrm{L}^2(0,T;\V))\cap\mathrm{L}^{4}(\Omega;\mathrm{L}^{4}(0,T;\widetilde{\L}^{4})),\end{align*} with $\mathbb{P}$-a.s., c\`adl\`ag paths in $\H$.
\end{theorem}
\begin{proof}
	A proof of the Theorem \ref{exis1} follows similarly as in the Theorem \ref{exis}, by using the global monotonicity result \eqref{218} and the fact that 
	\begin{align}
	&\int_0^T\langle\mathrm{G}(\u(t))-\mathrm{G}(\v(t)),\u(t)-\v(t)\rangle +\frac{L}{2}\int_0^T\|\u(t)-\v(t)\|_{\H}^2\d t\nonumber\\&\geq \frac{1}{2}\int_0^T\int_{\Z}\|\gamma(t, \u(t),z) - \gamma(t,	\v(t),z)\|^2_{\H}\lambda(\d z)\d t,
	\end{align}
	for $2\beta\mu \geq 1$. Uniqueness also follows easily by using the estimate \eqref{232}.
\end{proof}

\begin{remark}
	If the  domain is a $n$-dimensional torus, then one can approximate functions in $\L^p$-spaces using the truncated Fourier expansions in the following way (see Theorem 1.6, \cite{JCR4} and  Theorem 5.2, \cite{KWH}). Let $\mathcal{O}=[0,2\pi]^n$ and $\mathcal{Q}_k:=[-k,k]^n\cap\mathbb{Z}^n$. For every $\w\in\L^1(\mathcal{O})$ and every $k\in\mathbb{N}$, we define $R_k(\w):=\sum_{m\in\mathcal{Q}_k}\widehat{\w}_je^{\iota m\cdot x},$ where the Fourier coefficients $\widehat{\w}$ are given by $\widehat{\w}:=\frac{1}{|\mathcal{O}|}\int_{\mathcal{O}}\w(x)e^{-\iota m\cdot x}\d x$. Then, for every $1<p<\infty$, there exists a constant $C_p$, independent of $k$, such that $$\|R_k\w\|_{\L^p(\mathcal{O})}\leq C_p\|\w\|_{\L^p(\mathcal{O})}, \ \text{ for all } \ \w\in\L^p(\mathcal{O})$$ and $$\|R_k\w-\w\|_{\L^p(\mathcal{O})}\to 0\ \text{ as }\ k\to\infty.$$
\end{remark}
\subsection{Regularity of strong solution}
In order to get the regularity results of the strong solution to \eqref{32}, we restrict ourselves to periodic domains with $\int_{\mathcal{O}}\u(x)\d x=0$. The main difficulty in working with bounded domains is that $\mathrm{P}_{\H}(|\u|^{r-1}\u)$ need not be zero on the boundary, and $\mathrm{P}_{\H}$ and $-\Delta$ are not necessarily commuting (see \cite{JCR4}). Thus applying It\^o's formula for the process $\|\A^{1/2}\u(\cdot)\|_{\H}^2$ may not work. Moreover, $\Delta\u\big|_{\partial\mathcal{O}}\neq 0$ in general and the term with pressure will not disappear (see \cite{KT2}) and hence applying It\^o's formula for the stochastic process $\|\nabla\u(\cdot)\|_{\H}^2$ (for $\u(\cdot)$ appearing in \eqref{31}) also may not work. In a periodic domain or in whole space $\R^n,$ the operators $\mathrm{P}_{\H}$ and $-\Delta$ commute and we have the following result (see Lemma 2.1, \cite{KWH}): 
\begin{align}\label{370}
0\leq\int_{\mathcal{O}}|\nabla\u(x)|^2|\u(x)|^{r-1}\d x\leq\int_{\mathcal{O}}|\u(x)|^{r-1}\u(x)\cdot\A\u(x)\d x\leq r\int_{\mathcal{O}}|\nabla\u(x)|^2|\u(x)|^{r-1}\d x.
\end{align}
Using Lemma 2.2,  \cite{KWH}, we further have 
\begin{align}\label{371}
\|\u\|_{\widetilde{\L}^{3(r+1)}}^{r+1}\leq C\int_{\mathcal{O}}|\nabla\u(x)|^2|\u(x)|^{r-1}\d x, \ \text{ for }\ r\geq 1.
\end{align}
Note that the estimate \eqref{370} is true even in bounded domains (with Dirichlet boundary conditions) if one replaces $\A\u$ with $-\Delta\u$ and \eqref{371} holds true in bounded domains as well as in the whole space $\R^n$. We assume that the noise coefficient satisfies the following: 
\begin{hypothesis}\label{hyp2}
	There exist a positive
	constant $\widetilde{K}$ such that for all $t\in[0,T]$ and $\u\in\V$, 
	\begin{equation*}
\int_{\Z}	\|\A^{1/2}\gamma(t, \u,z)\|^{2}_{\H} \lambda(\d z)
	\leq \widetilde{K}\left(1 +\|\u\|_{\V}^{2}\right).
	\end{equation*}
\end{hypothesis}
\begin{theorem}\label{thm3.10}
	Let $\u_0\in \mathrm{L}^2(\Omega;\V)$ be given.  Then, for $r\geq 3$, the pathwise unique strong solution	$\u(\cdot)$ to the system (\ref{32}) satisfies the following regularity:  \begin{align*}\u&\in\mathrm{L}^2(\Omega;\mathrm{L}^{\infty}(0,T;\V)\cap\mathrm{L}^2(0,T;\D(\A)))\cap\mathrm{L}^{r+1}(\Omega;\mathrm{L}^{r+1}(0,T;\widetilde{\L}^{3(r+1)})),\end{align*} with $\mathbb{P}$-a.s., c\`adl\`ag paths in $\V$.
\end{theorem}
\begin{proof}
Let us take $\u_0\in\mathrm{L}^2(\Omega;\V),$  to obtain the further regularity results of the strong solution to \eqref{32} with $r\geq 3$. We provide a sketch of the proof only. Applying the infinite dimensional It\^o formula (\cite{GK1,Me,GK2}) to the process $\|\A^{1/2}\u(\cdot)\|_{\H}^2$, we get 
	\begin{align}\label{362}
&	\|\u(t)\|_{\V}^2+2\mu \int_0^t\|\A\u(s)\|_{\H}^2\d s\nonumber\\&=\|\u_0\|_{\V}^2-2\int_0^t(\B(\u(s)),\A\u(s))\d s	-2\beta\int_0^t(\mathcal{C}(\u(s)),\A\u(s))\d s\nonumber\\&\quad+\int_0^{t}\int_{\Z}\|\gamma(s,\u(s),z)\|^2_{\H}\pi(\d s,\d z) +2\int_0^{t}\int_{\Z}\left(\A^{1/2}\gamma(s-,\u(s-),z),\A^{1/2}\u(s-)\right)\wi\pi(\d s,\d z).
	\end{align}
	We estimate $|(\B(\u(s)),\A\u(s))|$ using H\"older's, and Young's inequalities as 
	\begin{align}\label{3.73}
	|(\B(\u),\A\u)|&\leq\||\u||\nabla\u|\|_{\H}\|\A\u\|_{\H}\leq\frac{\mu}{2}\|\A\u\|_{\H}^2+\frac{1}{2\mu }\||\u||\nabla\u|\|_{\H}^2. 
	\end{align}
	For $r>3$, we  estimate the final term from \eqref{3.73} using H\"older's and Young's inequalities as 
	\begin{align*}
	&	\int_{\mathcal{O}}|\u(x)|^2|\nabla\u(x)|^2\d x\nonumber\\&=\int_{\mathcal{O}}|\u(x)|^2|\nabla\u(x)|^{\frac{4}{r-1}}|\nabla\u(x)|^{\frac{2(r-3)}{r-1}}\d x\nonumber\\&\leq\left(\int_{\mathcal{O}}|\u(x)|^{r-1}|\nabla\u(x)|^2\d x\right)^{\frac{2}{r-1}}\left(\int_{\mathcal{O}}|\nabla\u(x)|^2\d x\right)^{\frac{r-3}{r-1}}\nonumber\\&\leq{\beta\mu }\left(\int_{\mathcal{O}}|\u(x)|^{r-1}|\nabla\u(x)|^2\d x\right)+\frac{r-3}{r-1}\left(\frac{2}{\beta\mu (r-1)}\right)^{\frac{2}{r-3}}\left(\int_{\mathcal{O}}|\nabla\u(x)|^2\d x\right).
	\end{align*}
Making use of the estimate \eqref{370} in \eqref{3.73}, taking supreumum over time from $0$ to $t$ and then taking  expectation, we have 
	\begin{align}\label{363}
&\E\left[\sup_{t\in[0,T]}	\|\u(t)\|_{\V}^2+\mu \int_0^T\|\A\u(t)\|_{\H}^2\d t+\beta\int_0^T\||\u(t)|^{\frac{r-1}{2}}|\nabla\u(t)|\|_{\H}^2\d s\right]\nonumber\\&\leq\E\left[\|\u_0\|_{\V}^2\right]+2\eta\E\left[\int_0^T\|\u(t)\|_{\V}^2\d t\right]+\E\left[\int_0^{T}\int_{\Z}\|\A^{1/2}\gamma(t,\u(t),z)\|^2_{\H}\lambda(\d z)\d
t\right] \nonumber\\&\quad+2\E\left[\sup_{t\in[0,T]}\left|\int_0^{t}\int_{\Z}\left(\A^{1/2}\gamma(s-,\u(s-),z),\A^{1/2}\u(s-)\right)\wi\pi(\d s,\d z)\right|\right],
\end{align}
where $\eta$ is defined in \eqref{215}.
Applying Burkholder-Davis-Gundy inequality to the final term appearing in the inequality \eqref{363}, we obtain 
\begin{align}\label{364}
&2\E\left[\sup_{t\in[0,T]}\left|\int_0^{t}\int_{\Z}\left(\A^{1/2}\gamma(s-,\u(s-),z),\A^{1/2}\u(s-)\right)\wi\pi(\d s,\d z)\right|\right]\nonumber\\&\leq 2\sqrt{3}\E\left[\int_0^T\int_{\Z}\|\A^{1/2}\gamma(t,\u(t),z)\|_{\H}^2\|\A^{1/2}\u(t)\|_{\H}^2\lambda(\d z)\d t\right]^{1/2}\nonumber\\&\leq 2\sqrt{3}\E\left[\sup_{t\in[0,T]}\|\A^{1/2}\u(t)\|_{\H}\left(\int_0^T\int_{\Z}\|\A^{1/2}\gamma(t,\u(t),z)\|_{\H}^2\lambda(\d z)\d t\right)^{1/2}\right]\nonumber\\&\leq\frac{1}{2}\E\left[\sup_{t\in[0,T]}\|\A^{1/2}\u(t)\|_{\H}^2\right]+6\E\left[\int_0^T\int_{\Z}\|\A^{1/2}\gamma(t,\u(t),z)\|_{\H}^2\lambda(\d z)\d t\right].
\end{align}
Substituting \eqref{364} in \eqref{363}, we find
\begin{align}\label{365}
&\E\left[\sup_{t\in[0,T]}	\|\u(t)\|_{\V}^2+2\mu \int_0^T\|\A\u(t)\|_{\H}^2\d t+2\beta\int_0^T\||\u(t)|^{\frac{r-1}{2}}|\nabla\u(t)|\|_{\H}^2\d t\right]\nonumber\\&\leq 2\E\left[\|\u_0\|_{\V}^2\right]+4\eta\E\left[\int_0^T\|\u(t)\|_{\V}^2\d t\right]+14\E\left[\int_0^{T}\int_{\Z}\|\A^{1/2}\gamma(t,\u(t),z)\|^2_{\H}\lambda(\d z)\d
t\right]\nonumber\\&\leq 2\E\left[\|\u_0\|_{\V}^2\right]+4\eta\E\left[\int_0^T\|\u(t)\|_{\V}^2\d t\right]+14\widetilde{K}\E\left[\int_0^{T}(1+\|\u(t)\|_{\V}^2)\d
t\right],
\end{align}
where we used Hypothesis \ref{hyp}. 
Applying Gronwall's inequality in \eqref{365}, we obtain 
\begin{align}\label{366}
&\E\left[\sup_{t\in[0,T]}	\|\u(t)\|_{\V}^2\right]\leq \left\{2\E\left[\|\u_0\|_{\V}^2\right]+14\widetilde{K}T\right\}e^{(4\eta+14\widetilde{K})T}.
\end{align}
Using \eqref{366} in \eqref{365}, we finally get 
\begin{align}\label{367}
&\E\left[\sup_{t\in[0,T]}	\|\u(t)\|_{\V}^2+2\mu \int_0^T\|\A\u(t)\|_{\H}^2\d t+2\beta\int_0^T\||\u(t)|^{\frac{r-1}{2}}|\nabla\u(t)|\|_{\H}^2\d t\right]\nonumber\\&\leq\left\{2\E\left[\|\u_0\|_{\V}^2\right]+14\widetilde{K}T\right\}e^{4(2\eta+7\widetilde{K})T},
\end{align}
which competes the proof for $r>3$.

For $r=3$, we estimate $|(\B(\u),\A\u)|$ as 
\begin{align}\label{2.91}
|(\B(\u),\A\u)|&\leq\|(\u\cdot\nabla)\u\|_{\H}\|\A\u\|_{\H}\leq\frac{1}{4\theta}\|\A\u\|_{\H}^2+\theta\||\u||\nabla\u|\|_{\H}^2.
\end{align}
A calculation similar to \eqref{365} gives 
\begin{align}\label{2.92}
&\E\left[\sup_{t\in[0,T]}	\|\u(t)\|_{\V}^2+2\left(\mu-\frac{1}{2\theta}\right) \int_0^T\|\A\u(t)\|_{\H}^2\d t+4(\beta-\theta)\int_0^T\||\u(t)||\nabla\u(t)|\|_{\H}^2\d t\right]\nonumber\\&\leq 2\E\left[\|\u_0\|_{\V}^2\right]+14\widetilde{K}\E\left[\int_0^{T}(1+\|\u(t)\|_{\V}^2)\d
t\right],
\end{align} 
For $2\beta\mu \geq 1$,  it is immediate that 
\begin{align}\label{2p94}
&\E\left[\sup_{t\in[0,T]}	\|\u(t)\|_{\V}^2+2\left(\mu-\frac{1}{2\theta}\right)\int_0^T\|\A\u(t)\|_{\H}^2\d t+4(\beta-\theta)\int_0^T\||\u(t)||\nabla\u(t)|\|_{\H}^2\d t\right]\nonumber\\&\leq\left\{2\E\left[\|\u_0\|_{\V}^2\right]+14\widetilde{K}T\right\}e^{28\widetilde{K}T}.
\end{align}
Hence $\u\in\mathrm{L}^2(\Omega;\mathrm{L}^{\infty}(0,T;\V)\cap\mathrm{L}^2(0,T;\D(\A)))$ and using the estimate \eqref{371}, we also get $\u\in\mathrm{L}^{r+1}(\Omega;\mathrm{L}^{r+1}(0,T;\widetilde{\L}^{3(r+1)}))$.  Moreover, the $\mathscr{F}_t$-adapted paths of $\u(\cdot)$ are c\`adl\`ag with trajectories in  $\D([0,T];\V)$, $\mathbb{P}$-a.s.
	\end{proof}

\section{Stationary solutions and stability}\label{se5}\setcounter{equation}{0}
In this section, we consider the stationary system (in the deterministic sense) corresponding to the convective Brinkman-Forchheimer  equations. We discuss about the existence and uniqueness of weak solutions to the steady state equations and examine the exponential stability as well as stabilization by pure jump noise results. 
\subsection{Existence and uniqueness of weak solutions to the stationary system} 
Let us consider the following stationary system: 
\begin{equation}\label{3pp1}
\left\{
\begin{aligned}
-\mu \Delta\u_{\infty}+(\u_{\infty}\cdot\nabla)\u_{\infty}+\beta|\u_{\infty}|^{r-1}\u_{\infty}+\nabla p_{\infty}&=\mathbf{f}, \ \text{ in } \ \mathcal{O}, \\ \nabla\cdot\u_{\infty}&=0, \ \text{ in } \ \mathcal{O}, \\
\u_{\infty}&=\mathbf{0}\ \text{ on } \ \partial\mathcal{O}.
\end{aligned}
\right.
\end{equation}
Taking the Helmholtz-Hodge orthogonal projection onto the system \eqref{3pp1}, we  write down the abstract formulation of the system \eqref{3pp1} as 
\begin{align}\label{3pp2}
\mu \A\u_{\infty}+\B(\u_{\infty})+\beta\mathcal{C}(\u_{\infty})=\f \ \text{ in }\ \V'.
\end{align}
The following theorem shows that there exists a unique weak solution of the system \eqref{3pp2} in  $\V\cap\widetilde{\L}^{r+1}$, for $r\geq 3$. 
Given any $\f\in\V'$, our problem is to find $\u_{\infty}\in\V\cap\widetilde{\L}^{r+1}$ such that 
\begin{align}\label{3pp3}
\mu (\nabla\u_{\infty},\nabla\v)+\langle\B(\u_{\infty}),\v\rangle+\beta\langle \mathcal{C}(\u_{\infty}),\v\rangle=\langle \f,\v\rangle, \ \text{ for all }\ \v\in\V\cap\widetilde{\L}^{r+1}, 
\end{align}
is satisfied. The next Theorem provides the existence and uniqueness of weak solutions of the system \eqref{3pp1}. 
\begin{theorem}[\cite{MTM4}]\label{thm6.1}
	For every $\f\in\V'$ and $r\geq 1$, there exists at least one weak solution of the system \eqref{3pp1}. For $r>3$, if 
	\begin{align}\label{8pp5}
	\mu>\frac{2\eta}{\uplambda_1},
	\end{align}
where $\eta$ is defined in \eqref{215}, 	then the solution of \eqref{3pp2} is unique. For $r=3$, the condition given in \eqref{8pp5} becomes $\mu\geq \frac{1}{2\beta}$. 
\end{theorem}

\subsection{Exponential stability in deterministic case} Let us now discuss about the exponential stability of the stationary solution obtained in Theorem \ref{thm6.1}. Let us first discuss about the deterministic case. 
\begin{definition}
	A weak solution\footnote{For $r\geq 3,$ the existence and uniqueness of weak solutions can be obtained from \cite{SNA,KT2,CLF}, etc} $\u(t)$ of the system  of the deterministic system:
	\begin{equation}\label{4p22}
	\left\{
	\begin{aligned}
	\frac{\d\u(t)}{\d t}+\mu \A\u(t)+\B(\u(t))+\beta\mathcal{C}(\u(t))&=\f,\\
	\u(0)&=\u_0,
	\end{aligned}
	\right.
	\end{equation}converges to $\u_{\infty}$ is \emph{exponentially stable in $\H$} if there exist a positive number $\kappa > 0$, such that
	\begin{align*}
	\|\u(t)-\u_{\infty}\|_{\H}\leq \|\u_0-\u_{\infty}\|_{\H}e^{-\kappa t},\ t\geq 0.
	\end{align*}
	In particular, if $\u_{\infty}$ is a stationary solution of system (\ref{3pp2}), then $\u_{\infty}$ is called \emph{exponentially stable in $\H$} provided that any weak solution to (\ref{4p22}) converges to $\u_{\infty}$ at the same exponential rate $\kappa> 0$.
\end{definition}
\begin{theorem}
	Let $\u_{\infty}$ be the unique solution of the system (\ref{3pp2}). If $\u(\cdot)$ is any weak  solution to the system \eqref{4p22}
	with $\u_0\in\H$ and $\f\in\V'$ arbitrary, then we have $\u_{\infty}$ is exponentially stable in $\H$ and $
	\u(t)\to \u_{\infty}\ \text{ in }\ \H\text{ as }\  t\to\infty,
	$
	for
	$\mu>\frac{2\eta}{\uplambda_1}$, for $r>3$ and $\mu\geq \frac{1}{2\beta},$ for $r=3$. 
\end{theorem}
\begin{proof} 
	Let us define $\w=\u-\u_{\infty}$, so that $\w$ satisfies the following system: 
	\begin{equation}\label{8.26}
	\left\{
	\begin{aligned}
	\frac{\d\w(t)}{\d t}+\mu  \A\w(t)+\beta (\B(\u(t))-\B(\u_{\infty}))+\beta (\mathcal{C}(\u(t))-\mathcal{C}(\u_{\infty}))&=0, \ t\in(0,T),\\
	\w(0)&=\u_0-\u_{\infty}.
	\end{aligned}
	\right.
	\end{equation}
	Taking inner product with $\w(\cdot)$ to the first equation in \eqref{8.26}, we find 
	\begin{align}\label{8.27}
	\frac{1}{2}\frac{\d}{\d t}\|\w(t)\|_{\H}^2+\mu \|\w(t)\|_{\V}^2&=-\beta\langle (\B(\u(t))-\B(\u_{\infty})),\w(t)\rangle -\beta\langle(\mathcal{C}(\u(t))-\mathcal{C}(\u_{\infty})),
	\w(t)\rangle\nonumber\\&\leq\frac{\mu }{2}\|\w(t)\|_{\V}^2+\eta\|\w(t)\|_{\H}^2,
	\end{align}
	for $r>3$,	where we used \eqref{2.27} and \eqref{2.30}. Thus, it is immediate that 
	\begin{align}
	&\frac{\d}{\d t}\|\w(t)\|_{\H}^2+(\uplambda_1\mu -2\eta)\|\w(t)\|_{\H}^2\leq 0. 
	\end{align}
	Thus, an application of variation of constants formula yields 
	\begin{align}
	\|\u(t)-\u_{\infty}\|_{\H}^2\leq e^{-\kappa t}\|\u_0-\u_{\infty}\|_{\H}^2, 
	\end{align}
	where $\kappa=(\uplambda_1\mu -2\eta)>0$, for $\mu>\frac{2\eta}{\uplambda_1}$ and the exponential stability of $\u_{\infty}$ follows. For $r=3$ and $\mu\geq\frac{1}{2\beta}$ one can use the estimates \eqref{231} and \eqref{232} to get the required result. 
\end{proof} 
\subsection{Exponential stability in stochastic case}
Now we discuss about the exponential stability results in  the stochastic case. 
\begin{definition}
	A strong solution $\u(t)$ of the system (\ref{32}) converges to $\u_{\infty}\in\H$ is \emph{exponentially stable in mean square} if there exist two positive numbers $a > 0$, such that
	\begin{align*}
	\E\left[\|\u(t)-\u_{\infty}\|_{\H}^2\right]\leq \E\left[\|\u_0-\u_{\infty}\|_{\H}^2\right]e^{-at},\ t\geq 0.
	\end{align*}
	In particular, if $\u_{\infty}$ is a stationary solution of system (\ref{3pp1}), then $\u_{\infty}$ is called \emph{exponentially stable in the mean square} provided that any strong solution to (\ref{32}) converges in $\mathrm{L}^2$ to $\u_{\infty}$ at the same exponential rate $a > 0$.
\end{definition}

\begin{definition}
	A strong solution $\u(t)$ of the system (\ref{32}) converges to $\u_{\infty} \in\H$ \emph{almost surely exponentially
		stable} if there exists $\upalpha >0$ such that
	$$\lim_{t\to+\infty}\frac{1}{t}\log\|\u(t)-\u_{\infty}\|_{\H}\leq -\upalpha,\ \mathbb{P}\text{-a.s.}$$
	In particular, if $\u_{\infty}$ is a stationary solution of system (\ref{3pp1}), then $\u_{\infty}$ is called \emph{almost surely exponentially stable} provided that any strong solution to (\ref{32}) converges in $\H$ to $\u_{\infty}$ with the same constant $\upalpha> 0$.
\end{definition}

Let us now show the exponential stability of the stationary solutions to the system \eqref{3pp1} in the mean square as well as almost sure sense. The authors in \cite{LHGH2} obtained similar results in the multiplicative Gaussian noise case, but with more regularity on the stationary solutions as well as the lower bound of $\mu$ depending on the stationary solutions. The following results are true for all $r\geq 3$, and one has to take $2\beta\mu\geq 1$, for $r=3$. In this section, we consider $\u(\cdot)$ as the solution to the system 
\begin{equation}\label{3.2}
\left\{
\begin{aligned}
\d\u(t)+\mu \A\u(t)+\B(\u(t))+\beta\mathcal{C}(\u(t))&=\f+\int_{\Z}\gamma(t-,\u(t-),z)\widetilde{\pi}(\d t,\d z), \ t\in(0,T),\\
\u(0)&=\u_0,
\end{aligned}
\right.
\end{equation}
for $\f\in\V'$ and $\u_0\in\mathrm{L}^2(\Omega;\H)$. 

\begin{theorem}\label{exp1}
	Let $\u_{\infty}$ be the unique stationary solution of (\ref{3pp1}) and  $\gamma(t,\u_{\infty},z)=0,$ for all $t\geq 0$ and $z\in\Z$. Suppose that the conditions in Hypothesis \ref{hyp} are satisfied, then for  $\theta=\mu\uplambda_1-(2\eta+L)>0,$ we have 
	\begin{align}\label{6.16}
&\E\left[\|\u(t)-\u_{\infty}\|_{\H}^2\right]\leq e^{-\theta t}\E\left[\|\u_0-\u_{\infty}\|_{\H}^2\right],
\end{align}
	provided 
	\begin{align}\label{6.3a}
	\mu>\frac{2\eta+L}{\uplambda_1},
	\end{align}
	where $L$ is the constant appearing in Hypothesis \ref{hyp} (H.3), $\eta$ is defined in \eqref{215} and $\uplambda_1$ is the Poincar\'e constant.
\end{theorem}
\begin{proof}
	Let us define $\w:=\u-\u_{\infty}$ and $\theta=\mu\uplambda_1-(2\eta+L)>0$. Then $\w(\cdot)$ satisfies the following It\^o stochastic differential: 
	\begin{align}
	\left\{
	\begin{aligned}
	\d\w(t)+\mu \A\w(t)&+(\B(\u(t))-\B(\u_{\infty}))+\beta(\mathcal{C}(\u(t))-\mathcal{C}(\u_{\infty}))\\&=\int_{\Z}(\gamma(t,\u(t),z)-\gamma(t,\u_{\infty},z))\wi\pi(\d s,\d z), \ t\in(0,T),\\
	\w(0)&=\u_0-\u_{\infty},
	\end{aligned}
	\right.
	\end{align}
	since $\gamma(t,\u_{\infty},z)=0$, for all $t\in(0,T)$ and $z\in\Z$. Then $\w(\cdot)$ satisfies the following energy equality: 
	\begin{align}
e^{\theta t}	\|\w(t)\|_{\H}^2&=\|\w_0\|_{\H}^2-2\int_0^{t}e^{\theta s}\langle\B(\u(s))-\B(\u_{\infty}(s)),\w(s)\rangle\d s+\theta\int_0^{t}e^{\theta s}\|\w(s)\|_{\H}^2\d s\nonumber\\&\quad-2\int_0^{t}e^{\theta s}\langle\mathcal{C}(\u(s))-\mathcal{C}(\u_{\infty}(s)),\w(s)\rangle\d s+\int_0^{t}e^{\theta s}\int_{\Z}\|\wi\gamma(s,\w(s),z)\|_{\H}^2\pi(\d s,\d z)\nonumber\\&\quad+2\int_0^{t}e^{\theta s}\int_{\Z}(\wi\gamma(s-,\w(s-),z),\w(s-))\wi\pi(\d s,\d z),
	\end{align}
where $\widetilde{\gamma}(\cdot,\w,\cdot)=\gamma(\cdot,\u(\cdot),\cdot)-\gamma(\cdot,\u_{\infty},\cdot)$. 	A calculation similar to \eqref{4.60} yields 
	\begin{align}\label{4.22}
e^{\theta t}\|\w(t)\|_{\H}^2&\leq\|\w_0\|_{\H}^2 +\left(\theta+2\eta-\mu\uplambda_1\right)\int_0^{t}e^{\theta s}\|\w(s)\|_{\H}^2\d s +\int_0^{t}e^{\theta s}\int_{\Z}\|\widetilde{\gamma}(s,\w(s),z)\|^2_{\H}\pi(\d s,\d z)\nonumber\\&\quad +2\int_0^{t}e^{\theta s}\int_{\Z}(\wi\gamma(s-,\w(s-),z),\w(s-))\wi\pi(\d s,\d z).
\end{align}
Taking expectation, and using Hypothesis \ref{hyp} (H.3) and the fact that the final term is a martingale, we find 
\begin{align}
	&e^{\theta t}\E\left[\|\w(t)\|_{\H}^2\right]\leq \E\left[\|\w_0\|_{\H}^2\right]+\left(\theta+\eta+L-\mu\uplambda_1\right)\int_0^{t}e^{\theta s}\E\left[\|\w(s)\|_{\H}^2\d s\right].
\end{align}
Since $\mu$ satisfies \eqref{6.3a} implies $\theta=\mu\uplambda_1-(2\eta+L)>0$ and an application of Gronwall's inequality yields  \eqref{6.16} is satisfied 	and hence $\u(t)$ converges to $\u_{\infty}$ exponentially in the mean square sense.
\end{proof}

\begin{theorem}\label{exp2}
	Let  all conditions given in Theorem \ref{exp1} are satisfied and 
	\begin{align}
	\mu>\frac{2\eta+3L}{\uplambda_1}.
	\end{align}
	 Then the strong solution $\u(\cdot)$ of the system (\ref{32}) converges to the stationary solution $\u_{\infty}$ of the system (\ref{3pp2}) almost surely exponentially stable.
\end{theorem}
\begin{proof}
	Let us take $n=1,2,\ldots$, $h>0$. Then the process $\|\u(\cdot)-\u_{\infty}\|_{\H}^2,$ for $t\geq nh$ satisfies:
	\begin{align}\label{6.17}
&	\|\u(t)-\u_{\infty}\|_{\H}^2+2\mu\int_{nh}^t\|\u(s)-\u_{\infty}\|_{\V}^2\d s\nonumber\\&=\|\u(nh)-\u_{\infty}\|_{\H}^2-2\int_{nh}^t\langle\B(\u(s))-\B(\u_{\infty}),\u(s)-\u_{\infty}\rangle\d s\nonumber\\&\quad-2\int_{nh}^t\langle\mathcal{C}(\u(s))-\mathcal{C}(\u_{\infty}),\u(s)-\u_{\infty}\rangle\d s\nonumber\\&\quad+2\int_{nh}^t\int_{\Z}\left(\gamma(s-,\u(s-),z)-\gamma(s-,\u_{\infty},z),\u(s-)-\u_{\infty}\right)\wi\pi(\d s,\d z)\nonumber\\&\quad +\int_{nh}^t\int_{\Z}\|\gamma(s,\u(s),z)-\gamma(s,\u_{\infty},z)\|_{\H}^2\pi(\d s,\d z).
	\end{align}
	Taking supremum from $nh$ to $(n+1)h$ and then taking expectation in (\ref{6.17}), we find 
	\begin{align}\label{6.18}
	&\E\left[\sup_{nh\leq t\leq  (n+1)h}\|\u(t)-\u_{\infty}\|_{\H}^2+\mu\int_{nh}^{(n+1)h}\|\u(s)-\u_{\infty}\|_{\V}^2\d s\right]\nonumber\\&\leq \E\left[\|\u(nh)-\u_{\infty}\|_{\H}^2\right] +2\eta\E\left[\int_{nh}^{(n+1)h}\|\u(s)-\u_{\infty}\|_{\V}^2\d s\right]\nonumber\\&\quad +\E\left[\int_{nh}^{(n+1)h}\int_{\Z}\|\gamma(s,\u(s),z)-\gamma(s,\u_{\infty},z)\|_{\H}^2\lambda(\d z)\d s\right]\nonumber\\&\quad +2\E\left[\sup_{nh\leq t\leq  (n+1)h}\left|\int_{nh}^t\int_{\Z}\left((\gamma(s-,\u(s-),z)-\gamma(s-,\u_{\infty},z)),\u(s-)-\u_{\infty}\right)\wi\pi(\d s,\d z)\right|\right],
	\end{align}
	where we used (\ref{2.27}) and (\ref{2.30}). We estimate the final term in the right hand side of the inequality \eqref{6.18} using Burkholder-Davis-Gundy, H\"older's and Young's inequalities as 
	\begin{align}\label{6.19}
	&2\E\left[\sup_{nh\leq t\leq  (n+1)h}\left|\int_{nh}^t\int_{\Z}\left((\gamma(s-,\u(s-),z)-\gamma(s-,\u_{\infty},z)),\u(s-)-\u_{\infty}\right)\wi\pi(\d s,\d z)\right|\right]\nonumber\\&\leq 2\sqrt{3}\E\left[\int_{nh}^{(n+1)h}\int_{\Z}\|\gamma(s,\u(s),z)-\gamma(s,\u_{\infty},z)\|_{\H}^2\|\u(s)-\u_{\infty}\|^2_{\H}\lambda(\d z)\d s\right]^{1/2}\nonumber\\&\leq 2\sqrt{3}\E\left[\sup_{nh\leq s\leq  (n+1)h}\|\u(s)-\u_{\infty}\|_{\H}\left(\int_{nh}^{(n+1)h}\int_{\Z}\|\gamma(s,\u(s),z)-\gamma(s,\u_{\infty},z)\|_{\H}^2\lambda(\d z)\d s\right)^{1/2}\right]\nonumber\\&\leq \frac{1}{2}\E\left[\sup_{nh\leq s\leq  (n+1)h}\|\u(s)-\u_{\infty}\|_{\H}^2\right]+6\E\left[\int_{nh}^{(n+1)h}\int_{\Z}\|\gamma(s,\u(s),z)-\gamma(s,\u_{\infty},z)\|_{\H}^2\lambda(\d z)\d s\right].
	\end{align}
	Substituting (\ref{6.19}) in (\ref{6.18}), and then using Hypothesis \ref{hyp} (see (H.2)), we get 
	\begin{align}\label{6.21}
	&\E\left[\sup_{nh\leq t\leq  (n+1)h}\|\u(t)-\u_{\infty}\|_{\H}^2\right]+\vartheta\E\left[\int_{nh}^{(n+1)h}\|\u(s)-\u_{\infty}\|_{\H}^2\d s\right]\nonumber\\&\leq 2\E\left[\|\u(nh)-\u_{\infty}\|_{\H}^2\right],
	\end{align}
	where 
	$$\vartheta=2\left(\mu\uplambda_1-(2\eta+6L)\right)>0.$$
	Let us now use (\ref{6.16}) in (\ref{6.21}) to obtain 
	\begin{align}
	\E\left[\sup_{nh\leq t\leq  (n+1)h}\|\u(t)-\u_{\infty}\|_{\H}^2\right]\leq 2\E\left[\|\u_0-\u_{\infty}\|_{\H}^2\right]e^{-\theta nh}.
	\end{align}
	Using Chebychev's inequality, for $\epsilon\in(0,\theta)$, we also have 
	\begin{align}
	&\mathbb{P}\left\{\omega\in\Omega:\sup_{nh\leq t\leq  (n+1)h}\|\u(t)-\u_{\infty}\|_{\H}>e^{-\frac{1}{2}(\theta-\epsilon)nh}\right\}\nonumber\\&\leq e^{(\theta-\epsilon)nh}\E\left[\sup_{nh\leq t\leq  (n+1)h}\|\u(t)-\u_{\infty}\|_{\H}^2\right]\leq 2\E\left[\|\u_0-\u_{\infty}\|_{\H}^2\right]e^{-\epsilon nh},
	\end{align}
	and 
	\begin{align}
	&\sum_{n=1}^{\infty}\mathbb{P}\left\{\omega\in\Omega:\sup_{nh\leq t\leq  (n+1)h}\|\u(t)-\u_{\infty}\|_{\H}>e^{-\frac{1}{2}(\theta-\epsilon)nh}\right\}\nonumber\\&\leq 2\E\left[\|\u_0-\u_{\infty}\|_{\H}^2\right]\frac{1}{e^{\epsilon h}-1}<+\infty.
	\end{align}
	Thus by using the Borel-Cantelli lemma, there is a finite integer $n_0(\omega)$ such that
	\begin{align}
	\sup_{nh\leq t\leq  (n+1)h}\|\u(t)-\u_{\infty}\|_{\H}\leq e^{-\frac{1}{2}(\theta-\epsilon)nh},\ \mathbb{P}\text{-a.s.,}
	\end{align}
	for all $n\geq n_0$, which completes the proof.
\end{proof}
\subsection{Stabilization by a multiplicative jump noise}
It is an interesting question  to ask about the exponential stability of  the stationary solution for small values of $\mu$. We  assume that (\cite{MTM})
\begin{align}\label{6.33a}
\gamma(\cdot,\u,z)=\gamma(z)(\u-\u_{\infty})
\end{align} 
such that 
\begin{align}\label{6.33b}
\int_{\Z} |\gamma(z)|^2\lambda(\d z)<+\infty\text{ and }\int_{\Z}\left(\gamma(z)-\log(1+\gamma(z))\right)\lambda(\d z)=\rho>0.
\end{align}
Remember that $\log(1+x)\leq x$, for $x>-1$. The following form of strong law of large numbers is needed to establish the stabilization result. 

\begin{lemma}[Strong law of large numbers, Theorem 1, \cite{RSL}]\label{slln}
	Let $\M=\M(t)_{t\geq t_0}$ be a local  c\`{a}dl\`{a}g martingale. If 
	\begin{align}
	\P\left\{\omega\in\Omega:\lim_{t\to\infty}\int_{t_0}^t\frac{\d\langle\M\rangle_s}{(1+s)^2}\right\}=1,\text{ then }\ \P\left\{\omega\in\Omega:\lim_{t\to\infty}\frac{\M(t)}{t}\right\}=1.
	\end{align}
\end{lemma}
We use the noise given in \eqref{6.33a} to obtain  the stabilization of the stochastic convective Brinkman-Forchheimer  equations. 
\begin{theorem}\label{thm4.7}
Let the equation \eqref{3.2} be perturbed by the jump noise given in  \eqref{6.33b} and noise coefficient \eqref{6.33a}. Then, there exists $\Omega_0\subset\Omega$ with $\mathbb{P}(\Omega_0 )=0$, such that for all $\omega\not\in\Omega_0$, there exists $T(\omega) > 0$ such that any strong solution $\u(t)$ to (\ref{3.2}) satisfies
\begin{align}
\|\u(t)-\u_{\infty}\|_{\H}^2\leq \|\u_0-\u_{\infty}\|_{\H}^2e^{-\zeta t}, \ \text{ for any }\ t\geq T(\omega),
\end{align}
where $\zeta=(\mu\uplambda_1-\eta+\rho)>0$. In particular, the exponential stability of sample paths  with probability one holds if $\zeta>0$.
\end{theorem}
\begin{proof}
	We know that the process $\u(\cdot)-\u_{\infty}$ satisfies the following energy equality:
	\begin{align}
	&\|\u(t)-\u_{\infty}\|_{\H}^2++2\mu\int_{0}^t\|\u(s)-\u_{\infty}\|_{\V}^2\d s\nonumber\\&=\|\u_0-\u_{\infty}\|_{\H}^2-2\int_{0}^t\langle\B(\u(s))-\B(\u_{\infty}),\u(s)-\u_{\infty}\rangle\d s\nonumber\\&\quad-2\int_{0}^t\langle\mathcal{C}(\u(s))-\mathcal{C}(\u_{\infty}),\u(s)-\u_{\infty}\rangle\d s+\int_{0}^t\int_{\Z}\|\gamma(z)(\u(s)-\u_{\infty})\|_{\H}^2\lambda(\d z)\d s\nonumber\\&\quad+\int_{0}^t\int_{\Z}\left[2\left(\gamma(z)(\u(s-)-\u_{\infty}),\u(s-)-\u_{\infty}\right)+\|\gamma(z)(\u(s-)-\u_{\infty})\|_{\H}^2\right]\wi\pi(\d s,\d z).
	\end{align}
	Applying It\^o's formula to the process $\log\|\u(\cdot)-\u_{\infty}\|_{\H}^2$, we find 
	\begin{align}\label{4.35}
&\log\|\u(t)-\u_{\infty}\|_{\H}^2\nonumber\\&=\log\|\u_0-\u_{\infty}\|_{\H}^2-2\mu\int_0^t\frac{\|\u(s)-\u_{\infty}\|_{\V}^2}{\|\u(s)-\u_{\infty}\|_{\H}^2}\d s\nonumber\\&\quad -2\int_{0}^t\frac{\langle\B(\u(s))-\B(\u_{\infty}),\u(s)-\u_{\infty}\rangle}{\|\u(s)-\u_{\infty}\|_{\H}^2}\d s-2\int_{0}^t\frac{\langle\mathcal{C}(\u(s))-\mathcal{C}(\u_{\infty}),\u(s)-\u_{\infty}\rangle}{\|\u(s)-\u_{\infty}\|_{\H}^2}\d s\nonumber\\&\quad+\int_0^t\int_{\Z}\frac{\|\gamma(z)(\u(s)-\u_{\infty})\|_{\H}^2}{\|\u(s)-\u_{\infty}\|_{\H}^2}\lambda(\d z)\d s\nonumber\\&\quad +\int_0^t\int_{\Z}\left(\log\|(\u(s-)-\u_{\infty})+\gamma(z)(\u(s-)-\u_{\infty})\|_{\H}^2-\log\|\u(s-)-\u_{\infty}\|_{\H}^2\right)\wi\pi(\d s,\d z)\nonumber\\&\quad +\int_0^t\int_{\Z}\bigg(\log\|(\u(s)-\u_{\infty})+\gamma(z)(\u(s)-\u_{\infty})\|_{\H}^2-\log\|\u(s)-\u_{\infty}\|_{\H}^2\nonumber\\&\qquad -\frac{2\left(\gamma(z)(\u(s)-\u_{\infty}),\u(s)-\u_{\infty}\right)+\|\gamma(z)(\u(s)-\u_{\infty})\|_{\H}^2}{\|\u(s)-\u_{\infty}\|_{\H}^2}\bigg)\lambda(\d z)\d s\nonumber\\&\leq\log\|\u_0-\u_{\infty}\|_{\H}^2+(-2\mu\uplambda_1+2\eta)t+2\int_0^t\int_{\Z}(\log(1+\gamma(z))-\gamma(z))\lambda(\d z)\d s\nonumber\\&\quad+2\int_0^t\int_{\Z}\log(1+\gamma(z))\wi\pi(\d s,\d z)\nonumber\\&\leq \log\|\u_0-\u_{\infty}\|_{\H}^2+2(-\mu\uplambda_1+\eta-\rho)t+2\int_0^t\int_{\Z}\log(1+\gamma(z))\wi\pi(\d s,\d z),
	\end{align}
where we used (\ref{2.27}), (\ref{2.30}) and \eqref{6.33b}.	Note that the term $$\M(t)=\int_0^t\int_{\Z}\log(1+\gamma(z))\wi\uppi(\d z,\d t),$$ is a real martingale. We know that 
\begin{align}
\langle\M\rangle_t&=\int_0^t\int_{\Z}\left[\log(1+\gamma(z))\right]^2\lambda(\d z)\d s\leq \int_{\Z}|\gamma(z)|^2\lambda(\d z)t,
\end{align}
and 
\begin{align}
\lim_{t\to\infty}\int_0^t\frac{\d\langle\M\rangle_s}{(1+s)^2}\leq \int_{\Z}|\gamma(z)|^2\lambda(\d z)<+\infty.
\end{align}
Thus	by means of the strong law of large numbers (see Lemma \ref{slln}), we have $$\lim_{t\to+\infty}\frac{\M(t)}{t}=0,\ \P\text{-a.s.}$$ One can assure the existence of  a set $\Omega_0\subset\Omega$ with $\mathbb{P}(\Omega_0)=0$, such
	that for every $\omega\not\in\Omega_0$, there exists $T(\omega)>0$ such that for all $t\geq T(\omega)$, we have
	$$\frac{2}{t}\int_0^t\int_{\Z}\log(1+\gamma(z))\wi\pi(\d s,\d z)\leq(\mu\uplambda_1-\eta+\rho).$$
	Hence from (\ref{4.35}), we finally have 
	\begin{align}
	\log\|\u(t)-\u_{\infty}\|_{\H}^2&\leq\log\|\u_0-\u_{\infty}\|_{\H}^2-(\mu\uplambda_1-\eta+\rho)t,
	\end{align}
for any $t\geq T(\omega)$,	which completes the proof.
\end{proof}

\section{Invariant Measures and Ergodicity}\label{se7}\setcounter{equation}{0} In this section, we discuss the existence and uniqueness of invariant measures and ergodicity results for the SCBF equations \eqref{32}. 
\subsection{Preliminaries}
Let us first provide the definitions of invariant measures,  ergodic, strongly mixing and exponentially mixing invariant measures. Let $\X$ be a Polish space (compete separable metric space).  
\begin{definition}
	A probability measure $\upeta$ on	$(\X,\mathscr{B}(\X))$ is called \emph{an invariant		measure or a stationary measure} for a given transition probability	function $\mathrm{P}(t,\x,\d \y),$ if it satisfies
	$$\upeta(\A)=\int_{\X}\mathrm{P}(t,\x,\A)\d\eta(\x),$$ for all $\A\in\mathscr{B}(\X)$ and
	$t>0$. Equivalently, if for all $\varphi\in \mathrm{C}_b(\X)$
	(the space of bounded continuous functions on $\X$), and all
	$t\geq 0$,
	$$\int_{\X}\varphi(\x)\d\upeta(\x)=\int_{\X}(\mathrm{P}_t\varphi)(\x)\d\upeta(\x),$$ where the Markov semigroup
	$(\mathrm{P}_t)_{t\geq 0}$ is defined by
	$$\mathrm{P}_t\varphi(\x)=\int_{\X}\varphi(\y)\mathrm{P}(t,\x,\d \y).$$
\end{definition}
\begin{definition}[Theorem 3.2.4, Theorem 3.4.2, \cite{GDJZ}, \cite{MSS}]
	Let $\upeta$ be an invariant measure for $\left(\mathrm{P}_t\right)_{t\geq 0}.$  We say that the measure $\upeta$ is an \emph{ergodic measure,}  if for all $\varphi \in \widetilde{\L}^2(\X;\upeta), $ we have  $$ \lim_{T\to +\infty}\frac{1}{T}\int_0^T (\mathrm{P}_t\varphi)(\x) \d t =\int_{\X}\varphi(\x) \d\upeta(\x) \ \text{ in } \ \widetilde{\L}^2(\X;\upeta).$$ The invariant measure $\upeta$ for $\left(\mathrm{P}_t\right)_{t\geq 0}$ is called \emph{strongly mixing} if  for all $\varphi \in \widetilde{\L}^2(\X;\upeta),$  we have $$\lim_{t\to+\infty}\mathrm{P}_t\varphi(\x) = \int_{\X}\varphi(\x) \d\upeta(\x)\ \text{ in }\ \widetilde{\L}^2(\X;\upeta).$$ The invariant measure $\upeta$ for $\left(\mathrm{P}_t\right)_{t\geq 0}$ is called \emph{exponentially mixing}, if there exists a constant $k>0$ and a positive function $\Psi(\cdot)$ such that for any bounded Lipschitz function $\varphi$, all $t>0$ and all $\x\in\X$, $$\left|\mathrm{P}_t\varphi(\x)-\int_{\X}\varphi(\x)\d\upeta(\x)\right|\leq \Psi(\x)e^{-k t}\|\varphi\|_{\mathrm{Lip}},$$  where $\|\cdot\|_{\mathrm{Lip}}$ is the Lipschitz constant. 
\end{definition}

	Clearly exponentially mixing implies strongly mixing. Theorem 3.2.6, \cite{GDJZ} states that if  $\upeta$ is the unique invariant measure for $(\mathrm{P}_t)_{t\geq 0}$, then  it is ergodic. The interested readers are referred to see \cite{GDJZ} for more details on the ergodicity for infinite dimensional systems and \cite{ADe} for ergodicity results for the stochastic Navier-Stokes equations. 

\subsection{Existence and uniqueness of invariant measures}
Let us now show that there exists a unique invariant measure for the Markovian transition probability associated to the  system (\ref{32}). Moreover, we establish that the invariant measure is ergodic and strongly mixing (in fact exponentially mixing). The authors in \cite{HGHL} also obtained similar results, but for the system \eqref{32} perturbed by additive L\'evy noise only. Let $\u(t,\u_0)$ denotes the unique strong solution of the system (\ref{32}) with the initial condition $\u_0\in\H$. Let $(\mathrm{P}_t)_{t\geq 0}$ be the \emph{Markovian transition semigroup} in the space $\C_b(\H)$ associated to the system (\ref{32}) defined by
\begin{align}\label{mar}
\mathrm{P}_t\varphi(\u_0)=\E\left[\varphi(\u(t,\u_0))\right]=\int_{\H}\varphi(\y)\mathrm{P}(t,\u_0,\d
\y)=\int_{\H}\varphi(\y)\upeta_{t,\u_0}(\d \y),\;\varphi\in \C_b(\H),
\end{align}
where $\mathrm{P}(t,\u_0,\d \y)$ is the transition probability of $\u(t,\u_0)$ and $\upeta_{t,\u_0}$ is the law of $\u(t,\u_0)$. The semigroup $(\mathrm{P}_t)_{t\geq 0}$ is Feller, since the solution to \eqref{32} depends continuously on the initial data (see \eqref{4.63}). From (\ref{mar}), we also have
\begin{align}\label{amr}
\mathrm{P}_t\varphi(\u_0)=\left<\varphi,\upeta_{t,\u_0}\right>=\left<\mathrm{P}_t\varphi,\upeta\right>,
\end{align}
where $\upeta$ is the law of the initial data $\u_0\in\H$. Thus from
(\ref{amr}), we have $\upeta_{t,\u_0}=\mathrm{P}_t^*\upeta$. We say that a
probability measure $\upeta$ on $\H$ is an \emph{invariant measure} if
\begin{align}
\mathrm{P}_t^*\upeta=\upeta,\textrm{ for all }\ t\geq 0.
\end{align}
That is, if a solution has law $\upeta$ at some time, then it has the same law for all later times. For such a solution, it can be shown by Markov property that for all $(t_1,\ldots,t_n)$ and $\tau>0$, $(\u(t_1+\tau,\u_0),\ldots,\u(t_n+\tau,\u_0))$ and $(\u(t_1,\u_0),\ldots,\u(t_n,\u_0))$ have the same law. Then, we say that the process $\u$ is \emph{stationary}. For more details, the interested readers are referred to see \cite{GDJZ,ADe}, etc.
\begin{theorem}\label{EIM}
	Let $\u_0\in\H$ be given. Then, for $\mu>\frac{K}{2\uplambda_1},$	there exists an invariant measure for the system (\ref{32}) with support in $\V$.
\end{theorem}
\begin{proof}
	Let us use the energy equality obtained in \eqref{3.63} to find 
	\begin{align}\label{7.20}
	&	\|\u(t)\|_{\H}^2+2\mu \int_0^t\|\u(s)\|_{\V}^2\d s+2\beta\int_0^t\|\u(s)\|_{\widetilde{\L}^{r+1}}^{r+1}\d s\nonumber\\&=\|{\u_0}\|_{\H}^2+\int_0^t\int_{\Z}\|\gamma(s,\u(s),z)\|_{\H}^2\pi(\d s,\d z)+2\int_0^t\int_{\Z}(\gamma(s-,\u(s-),z),\u(s-))\wi\pi(\d s,\d z).
	\end{align}
Taking expectation in (\ref{7.20}), using Hypothesis \ref{hyp} (H.2), Poincar\'e inequality and the fact that the final term is a martingale having zero expectation, we obtain  
	\begin{align}\label{5.4}
	&	\E\left\{	\|\u(t)\|_{\H}^2+\left(2\mu-\frac{K}{\uplambda_1}\right) \int_0^t\|\u(s)\|_{\V}^2\d s+2\beta\int_0^t\|\u(s)\|_{\widetilde{\L}^{r+1}}^{r+1}\d s\right\}\leq
	\E\left[\|\u_0\|_{\H}^2\right].
	\end{align}
	Thus, for $\mu>\frac{K}{2\uplambda_1}$, we have 
	\begin{align}\label{5.6}
	\frac{\left(2\mu-\frac{K}{\uplambda_1}\right)}{t}\E\left[\int_0^{t}\|\u(s)\|_{\V}^2\d s\right]\leq
	\frac{1}{T_0}\|\u_0\|_{\H}^2,\text{ for all }t>T_0.
	\end{align}
Using Markov's inequality, we get
	\begin{align}\label{5.7}
	\lim_{R\to\infty}\sup_{T>T_0}\left[\frac{1}{T}\int_0^T\mathbb{P}\Big\{\|\u(t)\|_{\V}>R\Big\}\d
	t\right]&\leq
	\lim_{R\to\infty}\sup_{T>T_0}\frac{1}{R^2}\E\left[\frac{1}{T}\int_0^T\|\u(t)\|_{\V}^2\d
	t\right]=0.
	\end{align}
	Hence along with the estimate in (\ref{5.7}),  using the compactness of $\V$ in $\H$, it is clear by a standard argument that the sequence of probability measures $$\upeta_{t,\u_0}(\cdot)=\frac{1}{t}\int_0^t\Pi_{s,\u_0}(\cdot)\d s,\ \text{ where }\ \Pi_{t,\u_0}(\Lambda)=\mathbb{P}\left(\u(t,\u_0)\in\Lambda\right), \ \Lambda\in\mathscr{B}(\H),$$ is tight, that is, for each $\updelta>0$, there is a compact subset $K\subset\H$  such that $\upeta_t(K^c)\leq \updelta$, for all $t>0$, and so by the Krylov-Bogoliubov theorem (or by a  result of Chow and Khasminskii see \cite{CHKH}) $\upeta_{t_n,\u_0}\to\upeta,$ weakly for $n\to\infty$, and $\upeta$ results to be an invariant measure for the transition semigroup $(\mathrm{P}_t)_{t\geq 0}$,  defined by 	$$\mathrm{P}_t\varphi(\u_0)=\E\left[\varphi(\u(t,\u_0))\right],$$ for all $\varphi\in\C_b(\H)$, where $\u(\cdot)$ is the unique strong solution of (\ref{32}) with the initial condition $\u_0\in\H$.
\end{proof}

Now we establish the uniqueness of invariant measure for the system (\ref{32}).  Similar results for 2D stochastic Navier-Stokes equations is established in \cite{ADe} and for the stochastic 2D Oldroyd models is obtained in \cite{MTM}. The following result provide the exponential stability results for the system \eqref{32}. 

\begin{theorem}\label{exps1}
	Let $\u(\cdot)$ and $\v(\cdot)$ be two solutions of the system (\ref{32}) with $r>3$ and the initial data $\u_0,\v_0\in\H$, respectively. Then, for the condition given in \eqref{6.3a}, we have 
	\begin{align}\label{513}
	\E\left[\|\u(t)-\v(t)\|_{\H}^2\right]\leq \|\u_0-\v_0\|_{\H}^2e^{-(\mu\uplambda_1-(2\eta+L))t},
	\end{align}
	where $\eta$ is defined in \eqref{215}. 
\end{theorem}
\begin{proof}
	Let us define $\w(t)=\u(t)-\v(t)$. Then, $\w(\cdot)$ satisfies the following energy equality: 
	\begin{align}\label{514}
	\|\w(t)\|_{\H}^2&=\|\w_0\|_{\H}^2-2\mu\int_0^t\|\w(s)\|_{\V}^2\d s-2\beta\int_0^t\langle
	\mathcal{C}(\u(s))-\mathcal{C}(\v(s)),\w(s)\rangle\d s\nonumber\\&\quad -2\int_0^t\left<\B(\u(s))-\B(\v(s)),\w(s)\right>\d s+\int_0^{t}\int_{\Z}\|\wi\gamma(s,\w(s),z)\|_{\H}^2\lambda(\d z)\d s\nonumber\\&\quad+2\int_0^{t}\int_{\Z}(\widetilde{\gamma}(s-,\w(s-),z),\w(s-))\wi\pi(\d s,\d z),
	\end{align}
	where $\widetilde{\gamma}(\cdot,\w(\cdot),\cdot)=\gamma(\cdot,\u(\cdot),\cdot)-\gamma(\cdot,\v(\cdot),\cdot)$. Taking expectation in \eqref{514} and then using the Poincar\'e inequality, Hypothesis (H.3), \eqref{2.27} and (\ref{2.30}), one can easily see that 
	\begin{align}\label{515}
	\E\left[	\|\w(t)\|_{\H}^2\right]\leq\|\w_0\|_{\H}^2-\mu\uplambda_1\int_0^t\E\left[\|\w(s)\|_{\H}^2\right]\d s+(2\eta+L)\int_0^t\E\left[\|\w(s)\|_{\H}^2\right]\d s,
	\end{align}
	where $\eta$ is defined in \eqref{215}. 	Thus, an application of the Gronwall's inequality yields 
	\begin{align}
	\E\left[	\|\w(t)\|_{\H}^2\right]\leq \|\w_0\|_{\H}^2e^{-(\mu\uplambda_1-(2\eta+L))t},
	\end{align}
	and for $\mu>\frac{2\eta+L}{\uplambda_1}$, we obtain the required result \eqref{513}. 
\end{proof}

	For $2\beta\mu\geq 1$, the results obtained in the Theorem \ref{exps1} can be established for $\mu>\frac{L}{\uplambda_1}$, using the estimate \eqref{2.26}. Let us now establish the uniqueness of invariant measures for the system \eqref{32} obtained in Theorem \ref{EIM}. We prove the case of $r>3$ only and the case of $r=3$ follows similarly.

\begin{theorem}\label{UEIM}
	Let the conditions given in Theorem \ref{exps1} hold true and $\u_0\in\H$ be given. Then, for  the condition given in \eqref{6.3a}, there is a unique invariant measure $\upeta$ to system (\ref{32}). The measure $\upeta$ is ergodic and strongly mixing, that is, \begin{align}\label{6.9a}\lim_{t\to\infty}\mathrm{P}_t\varphi(\u_0)=\int_{\H}\varphi(\v_0)\d\upeta(\v_0), \ \upeta\text{-a.s., for all }\ \u_0\in\H\ \text{ and }\  \varphi\in\C_b(\H).\end{align} 
\end{theorem}
\begin{proof}
For $\varphi\in \text{Lip}(\H)$ (Lipschitz $\varphi$), since $\upeta$ is an invariant measure, we have 
	\begin{align}
	&	\left|\mathrm{P}_t\varphi(\u_0)-\int_{\H}\varphi(\v_0)\upeta(\d \v_0)\right|\nonumber\\&=	\left|\E[\varphi(\u(t,\u_0))]-\int_{\H}\mathrm{P}_t\varphi(\v_0)\upeta(\d \v_0)\right|\nonumber\\&=\left|\int_{\H}\E\left[\varphi(\u(t,\u_0))-\varphi(\u(t,\v_0))\right]\upeta(\d \v_0)\right|\nonumber\\&\leq L_{\varphi}\int_{\H}\E\left[\left\|\u(t,\u_0)-\u(t,\v_0)\right\|_{\H}\right]\upeta(\d \v_0)\nonumber\\&\leq  L_{\varphi} e^{-\frac{(\mu\uplambda_1-(2\eta+L))t}{2}}\int_{\H}\|\u_0-\v_0\|_{\H}\upeta(\d \v_0)\nonumber\\&\leq  L_{\varphi}e^{-\frac{(\mu\uplambda_1-(2\eta+L))t}{2}}\left(\|\u_0\|_{\H}+\int_{\H}\|\v_0\|_{\H}\upeta(\d \v_0)\right)\to 0\ \text{ as } \ t\to\infty,
	\end{align}
	since $\int_{\H}\|\v_0\|_{\H}\upeta(\d \v_0)<+\infty$. Hence, we deduce (\ref{6.9a}), for every $\varphi\in \C_b (\H)$, by the density of $\text{Lip}(\H)$ in $\C_b (\H)$. Note that, we have a stronger result that $\mathrm{P}_t\varphi(\u_0)$ converges exponentially fast to equilibrium, which is the exponential mixing property. This easily gives uniqueness of the invariant measure also. Indeed, if  $\wi\upeta$ is an another invariant measure, then
	\begin{align}
	&	\left|\int_{\H}\varphi(\u_0)\upeta(\d \u_0)-\int_{\H}\varphi(\v_0)\wi\upeta(\d \v_0)\right|\nonumber\\&= \left|\int_{\H}\mathrm{P}_t\varphi(\u_0)\upeta(\d \u_0)-\int_{\H}\mathrm{P}_t\varphi(\v_0)\wi\upeta(\d \v_0)\right|\nonumber\\&=\left|\int_{\H}\int_{\H}\left[\mathrm{P}_t\varphi(\u_0)-\mathrm{P}_t\varphi(\v_0)\right]\upeta(\d \u_0)\wi\upeta(\d \v_0)\right|\nonumber\\&\leq L_{\varphi}e^{-\frac{(\mu\uplambda_1-(2\eta+L))t}{2}}\int_{\H}\int_{\H}\|\u_0-\v_0\|_{\H}\upeta(\d \u_0)\wi\upeta(\d \v_0)\to 0\ \text{ as }\  t\to\infty.
	\end{align}
	By Theorem 3.2.6, \cite{GDJZ}, since $\upeta$ is the unique invariant measure for $(\mathrm{P}_t)_{t\geq 0}$, we know that it is ergodic. 
\end{proof}

 \medskip\noindent
{\bf Acknowledgments:} M. T. Mohan would  like to thank the Department of Science and Technology (DST), India for Innovation in Science Pursuit for Inspired Research (INSPIRE) Faculty Award (IFA17-MA110). The author would also like to thank Prof. J. C. Robinson, University of Warwick for useful discussions and providing the crucial reference \cite{CLF}.


\begin{thebibliography}{99}
	
\bibitem{SNA}	S.N. Antontsev and H.B. de Oliveira, The Navier–Stokes problem modified by an absorption term, \emph{Applicable Analysis}, {\bf 89}(12),  2010, 1805--1825. 
	
	
	
	
	
	\bibitem{CTA} 	C. T. Anh and  P. T. Trang,  On the 3D Kelvin-Voigt-Brinkman-Forchheimer equations in 	some unbounded domains, \emph{Nonlinear Analysis: Theory, Methods $\&$ Applications}, {\bf 89} (2013), 36--54.
	
	
	
	\bibitem{Ap} D. Applebaum, \emph{L\'{e}vy processes and stochastic calculus},
	Cambridge Studies in Advanced Mathematics, Vol. {\bf 93}, Cambridge
	University press, 2004.
	
	
	\bibitem{VB} V. Barbu, {\it Analysis and control of nonlinear infinite dimensional	systems}, Academic Press, Boston, 1993.
	
	
	
	
	
	
	
	\bibitem{HBAM}		H. Bessaih and A. Millet,	On stochastic modified 3D Navier–Stokes equations with anisotropic viscosity, \emph{Journal of Mathematical Analysis and Applications}, 462 (2018), 915--956. 
	
	
	



\bibitem{ZBEH} Z. Brze\'zniak , E. Hausenblas and J. Zhu, 2D stochastic Navier-Stokes equations driven by jump noise, \emph{Nonlinear Analysis}, {\bf 79} (2013), 122-139. 

\bibitem{ZBGD}  Z. Brze\'zniak and Gaurav Dhariwal, Stochastic tamed Navier-Stokes equations on $\mathbb{R}^3$: the existence and the uniqueness of solutions and the existence of an invariant measure,  https://arxiv.org/pdf/1904.13295.pdf. 

	
	

	\bibitem{DLB}	D. L. Burkholder,	The best constant in the Davis inequality for the expectation of the martingale square function, \emph{Transactions of the
	American Mathematical Society} {\bf 354} (1), 91--105.
	
	
	\bibitem{ZCQJ} 	Z. Cai and Q. Jiu Weak and Strong solutions for the incompressible Navier-Stokes equations with damping, \emph{Journal of Mathematical Analysis and Applications}, {\bf 343} (2008), 799--809.
	
	
		\bibitem{CTRJ1}  T. Caraballo , J. Langa and T. Taniguchi, The exponential behavior and stabilizability of
	stochastic 2D-Navier–Stokes equations, \emph{Journal of Differential Equations} {\bf 179} (2002) 714--737.
	
	
	\bibitem{AC} A. Chorin, \emph{A Mathematical Introduction to Fluid Mechanics}, Springer-Verlag, 1992. 
	
		\bibitem {CHKH} P.-L. Chow and R. Khasminskii,  Stationary solutions of nonlinear stochastic evolution equations,
	\emph{Stochastic Analysis and Applications}, {\bf 15} (1997),
	671--699.
	
		\bibitem{chow}  P.-L. Chow, {\it Stochastic partial differential equations}, Chapman $\&$ Hall/CRC, New York, 2007.
		
		
	
	
	
	
	\bibitem {ICAM} I. Chueshov and A. Millet,	Stochastic 2D hydrodynamical type systems: Well posedness and Large Deviations, \emph{Applied Mathematics and  Optimization}, {\bf 61} (2010), 379--420.
	
	
	
	
	
	
	
	
		\bibitem{DaZ}
	\newblock G. Da Prato and J. Zabczyk,
	\newblock \emph{Stochastic Equations in Infinite Dimensions},
	\newblock Cambridge University Press, 1992.
	
		\bibitem{GDJZ}
	\newblock G. Da Prato and J. Zabczyk,
	\newblock \emph{Ergodicity for Infinite Dimensional Systems},
	\newblock London Mathematical Society Lecture Notes, {\bf 229}, Cambridge
	University Press, 1996.
	
		\bibitem {BD}	B. Davis, On the integrability of the martingale square function, \emph{Israel Journal of Mathematics} {\bf 8}(2) (1970), 187--190.
		
		
	
	
		\bibitem{ADe}
	\newblock A. Debussche,
	\newblock Ergodicity results for the stochastic Navier-Stokes equations: An introduction,
	\newblock Topics in Mathematical Fluid Mechanics, Volume {\bf 2073} of the series Lecture Notes in Mathematics, Springer,
	23--108, 2013.
	
\bibitem{ZDRZ} Z.  Dong and R.  Zhang,	3D tamed Navier-Stokes equations driven by multiplicative L\'evy noise: Existence, uniqueness and large deviations, https://arxiv.org/pdf/1810.08868.pdf
	
	
	
	
	
	
	
	
	
	
	
	
	
	\bibitem{CLF} 	C. L. Fefferman, K. W. Hajduk and J. C. Robinson,	\emph{Simultaneous approximation in Lebesgue and Sobolev norms via eigenspaces}, https://arxiv.org/abs/1904.03337.

	\bibitem{FMRT} C. Foias, O. Manley, R. Rosa,  and R. Temam,  \emph{Navier-Stokes Equations and Turbulence},  Cambridge University Press, 2008.
	
		\bibitem{GGP} G. P. Galdi,  An introduction to the Navier–Stokes initial-boundary value problem. pp. 11-70 in \emph{Fundamental directions in mathematical fluid mechanics}, Adv. Math. Fluid Mech. Birkha\"user, Basel 2000.
		
		\bibitem{HGHL} H. Gao,and H. Liu,	Well-posedness and invariant measures for a class of stochastic 3D Navier-Stokes equations with damping driven by jump noise, \emph{Journal of Differential Equations}, {\bf 267} (2019), 5938--5975. 
	
		\bibitem{GK1} I. Gy\"ngy and  N. V. Krylov,  On stochastic equations with respect to semimartingales II, It\^o formula in Banach spaces, Stochastics, {\bf 6}(3--4) (1982), 153--173. 
		
		\bibitem{GK2} I. Gy\"ngy and D. Siska,	It\^o formula for processes taking values in intersection	of finitely many Banach spaces, \emph{Stoch PDE: Anal Comp}, {\bf 5}(2017), 428--455. 
	
	
	
	


	\bibitem{KWH}	K. W. Hajduk and J. C. Robinson, Energy equality for the 3D critical convective Brinkman-Forchheimer equations, \emph{Journal of Differential Equations}, {\bf 263} (2017), 7141--7161.
	
\bibitem{KWH1}	K. W. Hajduk, J. C. Robinson and W. Sadowski,	Robustness of regularity for the 3D convective Brinkman--Forchheimer equations, https://arxiv.org/pdf/1904.03311.pdf. 
	
	
	

	
\bibitem{AI} A. Ichikawa, Some inequalities for martingales and stochastic convolutions, Stochastic Analysis and Applications, {\bf 4}(3) (1986), 329--339.
	
	
	
	
	
	
	
	
	\bibitem{KT2}  V. K. Kalantarov and S. Zelik, Smooth attractors for the Brinkman-Forchheimer equations with fast growing nonlinearities, \emph{Commun. Pure Appl. Anal.}, {\bf 11}	(2012) 2037--2054.
	
	
	
	\bibitem{OAL}	O. A. Ladyzhenskaya, \emph{The Mathematical Theory of Viscous Incompressible Flow}, Gordon and Breach, New York, 1969.
	

\bibitem{RSL} R. Sh. Liptsera, Strong law of large numbers for local martingales, \emph{Stochastics}, {\bf 3} (1980), 217--228.

	
	\bibitem{LHGH}	H. Liu and H. Gao, 	Ergodicity and dynamics for the stochastic 3D Navier-Stokes equations with damping, \emph{Commun. Math. Sci.}, {\bf 16}(1) (2018), 97--122.
	
		\bibitem{LHGH1}	H. Liu and H. Gao, Stochastic 3D Navier–Stokes equations with nonlinear damping: martingale solution, strong solution and small time LDP, Chapter 2 in \emph{Interdisciplinary Mathematical SciencesStochastic PDEs and Modelling of Multiscale Complex System}, 9--36, 2019.
		
		\bibitem{LHGH2}	H. Liu,  L. Lin, C. Sun and Q. Xiao, The exponential behavior and stabilizability of the stochastic 3D Navier–Stokes equations with damping, \emph{Reviews in Mathematical Physics},
			{\bf 31} (7) (2019), 1950023. 
	
	
\bibitem{WLMR} W. Liu and M. R\"ockner,	Local and global well-posedness of SPDE with generalized coercivity conditions, \emph{Journal of Differential Equations}, {\bf 254} (2013), 725--755. 

\bibitem{WL}  W. Liu, Well-posedness of stochastic partial differential equations with Lyapunov condition, \emph{Journal of Differential Equations}, {\bf 255} (2013), 572--592. 


	
		
	
	

\bibitem{VMBR} V.  Mandrekar, B. R\"udiger, \emph{Stochastic Integration in Banach Spaces,	Theory and Applications}, Springer, 2015. 

\bibitem{UMMT}  U. Manna, M. T. Mohan, Shell model of turbulence perturbed by L\'evy noise,
\emph{Nonlinear Differential Equations and Applications}, {\bf  18}(6) (2011), 615--648.

\bibitem{UMMT1} U. Manna, M. T. Mohan and S. S. Sritharan, Stochastic non-resistive magnetohydrodynamic system with L\'evy noise, \emph{Random Operators and Stochastic Equations}, {\bf 25}(3) (2017), 155--194.
	
	\bibitem{GJM} G. J. Minty, Monotone (nonlinear) operators in Hilbert space, \emph{Duke Math. J.}, {\bf 29}(3) (1962), 341--346.
	
		\bibitem{Me} M. M\'{e}tivier, \emph{Stochastic partial differential equations in infinite dimensional	spaces}, Quaderni, Scuola Normale Superiore, Pisa, 1988.
		
		\bibitem{MJSS} J. L. Menaldi and S. S. Sritharan, {Stochastic $2$-D Navier-Stokes equation},	\emph{Applied Mathematics and Optimization}, {\bf 46} (2002), 31--53.
	
	
	
	\bibitem{MSS} M. T. Mohan, K. Sakthivel  and S. S. Sritharan, Ergodicity for the 3D stochastic Navier-Stokes equations perturbed by L\'{e}vy noise,  \emph{Mathematische Nachrichten}, {\bf 292} (5), 1056--1088, 2019. 
	
	\bibitem{MTM} M. T. Mohan, Deterministic and stochastic equations of motion arising in Oldroyd fluids of  order one: Existence, uniqueness, exponential stability and invariant measures, \emph{Stochastic Analysis and Applications}, {\bf 38}(1) (2020), 1--61.
	

	
	
	
	
	
	
	
	
	\bibitem{MoSS2} {M. T. Mohan and S. S. Sritharan,} Stochastic Navier-Stokes equation perturbed by L\'{e}vy noise with hereditary viscosity,   \emph{Infinite Dimensional Analysis, Quantum Probability and Related Topics},  {\bf  22} (1), 1950006 (32 pages), 2019.
	
	\bibitem{MTM4} M. T. Mohan, 	Stochastic convective Brinkman-Forchheimer equations, \emph{submitted}.  

		\bibitem{PEP} P. E. Protter, {\it Stochastic integration and  differential equations},
	Springer-Verlag, Second edition, New York, 2005.

\bibitem{JCR3}	J. C. Robinson and W. Sadowski,	A local smoothness criterion for solutions of the 3D Navier-Stokes equations, \emph{Rendiconti del Seminario Matematico della Universit\'a di Padova} {\bf 131} (2014), 159--178.
	
		
	\bibitem{JCR4}	J.C. Robinson,  J.L. Rodrigo,  W. Sadowski, \emph{The three-dimensional Navier–Stokes equations, classical theory}, Cambridge Studies in Advanced Mathematics, Cambridge
		University Press, Cambridge, UK, 2016. 
		

		
		
		\bibitem{MRXZ}	M. R\"ockner and X. Zhang, Tamed 3D Navier-Stokes equation: existence, uniqueness and
		regularity, \emph{Infinite Dimensional Analysis, Quantum Probability and Related Topics}, {\bf 12} (2009), 525--549.
		
			\bibitem{MRXZ1}	M. R\"ockner and X. Zhang, Stochastic tamed 3D Navier-Stokes equation: existence, uniqueness and ergodicity, \emph{Probability Theory and Related Fields}, {\bf 145} (2009) 211--267.
			
		\bibitem{MRTZ1}	M. R\"ockner, T. Zhang and X. Zhang,	Large deviations for stochastic tamed 3D Navier-Stokes equations, \emph{Applied Mathematics and Optimization}, {\bf 61} (2010), 267--285. 
		
			\bibitem{MRTZ}	 M. R\"ockner and T. Zhang, Stochastic 3D tamed Navier-Stokes equations: Existence, uniqueness and small time large deviations principles, J\emph{ournal of Differential Equations}, {\bf 252} (2012), 716--744.
	
	
	
		\bibitem{SSSP} S. S. Sritharan and P. Sundar,
	{Large deviations for the two-dimensional Navier-Stokes
		equations with multiplicative noise}, \emph{Stochastic Processes and their
		Applications}, {\bf 116} (2006), 1636--1659.
	
	
	\bibitem{Te} R. Temam,  \emph{Navier-Stokes Equations, Theory and Numerical Analysis}, North-Holland, Amsterdam, 1984.
	

	
	
	
	\bibitem{Te1} R. Temam, 	\emph{Navier-Stokes Equations and Nonlinear Functional Analysis}, Second Edition, CBMS-NSF Regional Conference Series in Applied Mathematics, 1995.
	
	
	
\bibitem{BYo}	B. You,	The existence of a random attractor for the three dimensional damped Navier-Stokes equations with additive noise, \emph{Stochastic Analysis and Applications},  {\bf 35}(4) (2017), 691--700. 

	
	
	\bibitem{ZZXW}	Z. Zhang, X. Wu and M. Lu, On the uniqueness of strong solution to the incompressible Navier-Stokes equations with damping, \emph{Journal of Mathematical Analysis and Applications}, {\bf 377} (2011), 414--419.
	
	
	\bibitem{YZ} 	Y. Zhou, Regularity and uniqueness for the 3D incompressible Navier-Stokes equations with damping, \emph{Applied Mathematics Letters}, {\bf 25} (2012), 1822--1825.
	
	
	
	

	
\end{thebibliography}
\end{document}